\documentclass[11pt,american,english]{article}

\usepackage[english]{babel}
\usepackage{constants}
\usepackage[utf8x]{inputenc}
\usepackage[T1]{fontenc}
\usepackage{blindtext}
\usepackage{titlesec}
\usepackage{float}
\usepackage{xcolor}
\usepackage{enumerate}
\restylefloat{table}
\usepackage{authblk}
\usepackage[a4paper,top=3cm,bottom=3cm,left=2cm,right=1.5cm,marginparwidth=1.75cm]{geometry}
\usepackage{amsmath}
\allowdisplaybreaks[4]
\usepackage{bm}
\usepackage{graphicx}
\usepackage[colorinlistoftodos]{todonotes}
\usepackage[colorlinks=true, allcolors=blue]{hyperref}
\usepackage{tikz}
\usepackage{caption}
\usetikzlibrary{decorations.pathreplacing}
\numberwithin{equation}{section}
\newtheorem{theorem}{Theorem}[section]

\newtheorem{prpstn}[theorem]{Proposition}

\newtheorem{definition}[theorem]{Definition}

\newtheorem{remark}[theorem]{Remark}

\usepackage{amssymb}
\usepackage{mathrsfs}
\usepackage{color}
\newcommand{\qed}{\nobreak \ifvmode \relax \else
      \ifdim\lastskip<1.5em \hskip-\lastskip
      \hskip1.5em plus0em minus0.5em \fi \nobreak
      \vrule height0.75em width0.5em depth0.25em\fi}
\newenvironment{proof}[1][Proof]{\begin{trivlist}
\item[\hskip \labelsep {\bfseries #1}]}{\hfill\qed\end{trivlist}}
\usepackage{dsfont}
\def\no{\nonumber}
\usepackage{mathtools}
\usepackage{stmaryrd}

\newcommand{\R}{\mathbb{R}}

\newcommand{\ot}{\#}

\newcommand{\p}{\partial}
\newcommand{\mc}{\mathcal}
\renewcommand{\div}[0]{\text{div}}
\renewcommand{\/}[2]{\frac{#1}{#2}}

\def\eps{\epsilon}
\def\wt{\widetilde}
\def\wh{\widehat}
\def\wb{\overline}
\DeclareFontFamily{U}{mathx}{}
\DeclareFontShape{U}{mathx}{m}{n}{<-> mathx10}{}
\DeclareSymbolFont{mathx}{U}{mathx}{m}{n}
\DeclareMathAccent{\widecheck}{0}{mathx}{"71}
\newcommand{\vertiii}[1]{{\left\vert\kern-0.25ex\left\vert\kern-0.25ex\left\vert #1 
    \right\vert\kern-0.25ex\right\vert\kern-0.25ex\right\vert}}
\newcommand{\Lip}{\mathrm{Lip}}

\DeclareMathOperator*{\argmin}{\mathrm{arg}\,\mathrm{min}}

\title{Sections and Chapters}

\author[,1]{\normalsize Alain Bensoussan\footnote{E-mail: axb046100@utdallas.edu.}}
\author[,2]{\normalsize Tak Kwong Wong\footnote{E-mail: takkwong@maths.hku.hk.}}
\author[,3]{\normalsize Sheung Chi Phillip Yam\footnote{E-mail: scpyam@sta.cuhk.edu.hk.}}
\author[,3,4]{\normalsize Hongwei Yuan\footnote{E-mail: hwyuan@link.cuhk.edu.hk; hwyuan@um.edu.mo.}}
\affil[1]{\small\it Naveen Jindal School of Management, University of Texas at Dallas}
\affil[2]{\small\it Department of Mathematics, The University of Hong Kong}
\affil[3]{\small\it Department of Statistics, The Chinese University of Hong Kong}
\affil[4]{\small\it Department of Mathematics, University of Macau}

\title{Global Well-Posedness of First-Order Mean Field Games and Master Equations with Nonlinear Dynamics}
\begin{document}
\maketitle

\begin{abstract}
This article presents the variant of the approach introduced in the recent work of Bensoussan, Wong, Yam and Yuan \cite{bensoussan2023theory} to the generic first-order mean field game problem. A major contribution here is the provision of new crucial {\it a priori} estimates, whose establishment is fundamentally different from the mentioned work since the associated forward-backward ordinary differential equation (FBODE) system is notably different. In addition, we require monotonicity conditions intimately on the coefficient functions but not on the Hamiltonians to handle their non-separable nature and nonlinear dynamics; as tackling Hamiltonians directly, it potentially dissolves much useful information. Compared with the assumptions used in \cite{bensoussan2023theory}, we introduce an additional requirement that the first-order derivative of the drift function in the measure variable cannot be too large relative to the convexity of the running cost function; this requirement only arises when the Hamiltonian is non-separable, and this phenomenon can also be seen in the existing literature. On the other hand, we require less here for the second-order differentiability of the coefficient functions in comparison to that in \cite{bensoussan2023theory}. Our approach involves first demonstrating the local existence of a solution over small time interval, followed by the provision of new crucial {\it a priori} estimates for the sensitivity of the backward equation with respect to the initial condition of forward dynamics; and finally, smoothly gluing the local solutions together to form a global solution. In addition, we establish the local and global existence and uniqueness of classical solutions for the mean field game and its master equation. Finally, to illustrate the effectiveness of our proposed general theory, we base on our theory to provide the resolution of a non-trivial non-linear-quadratic example with non-separable Hamiltonian and nonlinear dynamics which cannot yet be explained in the contemporary literature.
	
\end{abstract}
\begin{center}
\begin{minipage}{5.5in}
\end{minipage}
\end{center}

\tableofcontents

\section{Introduction}


Since the last two decades, mean field games and control problems have become one of the popular research in mathematics and engineering. The topics involved the study of systems with a large number of interactive agents; in particular, as the number of agents approaches to infinity, a mean field approximation becomes increasingly applicable. This methodology, first introduced by Huang-Malham\'e-Caines \cite{huang2006large} under the name ``large population stochastic dynamic games'' and Larsy-Lions \cite{LL07} under another name ``mean field games'' independently, facilitates the analysis and comprehension of these complex systems composed of large agent populations.
In recent years, the theory has been enriched with substantial advancements, fueled by both theoretical progress and practical applications. The pioneering monographs by Cardaliaguet, Delarue, Lasry and Lions~\cite{cardaliaguet2019master}, Carmona and Delarue~\cite{carmona2018probabilistic}, Gomes, Pimentel and Voskanyan~\cite{gomes2016regularity}, and Bensoussan, Frehse and Yam \cite{bensoussan2013mean} have made remarkable strides in elucidating the fundamental problem formulation and optimization procedure.
For a thorough understanding of the stochastic (Pontryagin) maximum principle, McKean-Vlasov dynamics, and forward-backward stochastic differential equations (FBSDEs), as well as their relationship to mean field games and control problems, interested readers may refer to the following key resources \cite{andersson2011maximum,bensoussan2020control,bensoussan2016linear,homan2023game,homan2023control,bensoussan2015well,
buckdahn2011general,buckdahn2009meanLim,buckdahn2017meanSDC,buckdahn2009mean,buckdahn2017mean,carmona2015forward,
chassagneux2022probabilistic,cosso2023optimal,cosso2019zero,djete2022mckean,huang2006large,pardoux2005backward,pham2016linear,PHW,pham2018bellman}.
To delve deeper into the dynamic programming principle, Bellman equations, master equations and their connections with mean field games and control problems, the following references can provide invaluable insights \cite{bensoussan2015master,bensoussan2019control,cardaliaguet2022splitting,cosso2021master,gangbo2020global,gangbo2022mean,
gangbo2015existence,lauriere2014dynamic,li2012stochastic}.

One representative model of the mean field games problem is the linear quadratic setting, which has been investigated by Bensoussan-Sung-Yam-Yung \cite{bensoussan2016linear}, Carmona-Delarue-Lachapelle, \cite{carmona2013control}, Pham \cite{pham2016linear}, Yong \cite{yong2013linear}. For the problem with a more complicated structure, there are three mainstream approaches in the literature. The first one is to solve the forward-backward Fokker-Planck (FP) and Hamilton-Jacobi-Bellman (HJB) system. Cardaliaguet \cite{C15} and Graber \cite{G14} studied the well-posedness of weak solution of the first-order (in the absence of Brownian noise) FP-HJB system, with various structures of the Hamiltonian. Subsequently, Cardaliaguet-Graber \cite{CG15} generalized their findings under more general conditions. 
For the second-order case (with only individual Brownian noise), when there is a nonlocal smoothing coupling, Lasry-Lions \cite{LL07} showed the classical well-posedness of the FP-HJB system under the Lasry-Lions monotonicity. Porretta \cite{P15} showed the well-posedness of weak solution for more general cost functionals. All these aforementioned articles are working on separable Hamiltonians and on a torus. Further, Ambrose \cite{A18} established the well-posedness of strong solution to the FP-HJB system on a torus with a non-separable Hamiltonian, provided that the initial and terminal data or the Hamiltonian is small.

The second approach gets a popularity recently, which is about to solve an infinite-dimensional PDE, namely the {\it master equation}. This idea was first proposed by P.-L. Lions in the lecture \cite{L14} (see also Cardaliaguet \cite{C13}, and the derivations in Bensoussan-Frehse-Yam \cite{bensoussan2017interpretation,bensoussan2015master},  Carmona-Delarue \cite{carmona2018probabilistic}), and then it gained significant attention from scholars. For the first-order master equation (without Brownian noise), a local (in time) classical solution on a torus was established by Gangbo-{\'S}wi{\k{e}}ch \cite{gangbo2015existence} for a quadratic Hamiltonian; Mayorga \cite{M20} extended the work to a generic separable Hamiltonian on $\mathbb{R}^{d}$. For the second-order master equation (with individual Brownian noise), Ambrose-{\'S}wi{\k{e}}ch \cite{AM23} constructed local solutions on a torus with a non-separable Hamiltonian locally depending on the measure argument. In Cardaliaguet-Cirant-Porretta \cite{cardaliaguet2022splitting}, local solutions were constructed on $\mathbb{R}^{d}$ with a non-separable Hamiltonian that also involves common noise via a novel splitting method. Buckdahn-Li-Peng-Rainer \cite{buckdahn2017mean} utilized probabilistic techniques to study the global well-posedness of a linear master equation. In our knowledge, up to the present understanding, in order to extend to the global well-posedness of a generic problem, certain monotonicity assumptions are required on the Hamiltonian. The first type of monotonicity, known as {\it Lasry-Lions monotonicity} (LLM) as mentioned above, was introduced in Lasry-Lions \cite{LL07}. Under this condition, Cardaliaguet-Delarue-Lasry-Lions \cite{cardaliaguet2019master} established the global well-posedness of the master equation on a torus with a generic separable Hamiltonian and a common noise using a PDE approach and H\"older estimates. Chassagneux-Crisan-Delarue \cite{chassagneux2022probabilistic} built upon the LLM, utilizing a probabilistic approach and the associated FBSDEs, to establish the global well-posedness of the master equation on $\mathbb{R}^{d}$ with a separable Hamiltonian. Additionally, the global well-posedness of different senses of weak solutions, with a separable Hamiltonian and the presence of a common noise, were established in Mou-Zhang \cite{mou2019wellposedness} (on $\mathbb{R}^{d}$), Bertucci \cite{B21} (on a torus) and Cardaliaguet-Souganidis \cite{CS21} (on $\mathbb{R}^{d}$, without individual noise, still, there is a common noise). Also, see the LLM for master equations in mean field type control problems in Bensoussan-Tai-Yam \cite{homan2023control}. Finally, two examples in Cecchin-Pra-Fischer-Pelino \cite{CPFP19} (two-state model) and Delarue-Tchuendom \cite{DT20} (linear quadratic model) illustrated the non-unique existence of Nash equilibrium in the absence of LLM. Another monotonicity assumption more recently received is called the {\it displacement monotonicity} (DM), see \cite{gangbo2022mean}. The concept of DM in the context of mean field games was initially introduced as ``weak monotonicity'' by Ahuja \cite{A16} which employed this concept to examine the well-posedness of mean field games with common noise using simple cost functions via the FBSDE approach. The DM was first used in Gangbo-M{\'e}sz{\'a}ros \cite{gangbo2020global} for the global well-posedness of first-order master equation with separable Hamiltonians on $\mathbb{R}^{d}$, then in Gangbo-M{\'e}sz{\'a}ros-Mou-Zhang \cite{gangbo2022mean} for second-order master equations with a presence of common noise and non-separable Hamiltonian with bounded derivatives up to fifth order. Also, see the role of DM for master equations in mean field type control problems in Bensoussan-Graber-Yam \cite{bensoussan2020control}, Bensoussan-Tai-Yam \cite{homan2023control} and also Bensoussan-Wong-Yam-Yuan \cite{bensoussan2023theory}. One can read Section 2.1 in \cite{homan2023game}, Remark 3.1 in \cite{homan2023control}, Section 5.3 in \cite{bensoussan2023theory} and Graber-M{\'e}sz{\'a}ros \cite{GM22mono} for more discussion about these two monotonicity assumptions. Also, Cecchin-Delarue \cite{CD22weak} established the well-posedness of the weak solution to the master equation on a torus for a separable Hamiltonian, without imposing any monotonicity on the coefficients.

The third approach of resolving mean field games is to employ probabilistic tools. Carmona-Lacker \cite{CL15} utilized this approach by first weakly solving the optimal control problem and then using the modified Kakutani's theorem to find the equilibrium measure. Carmona-Delarue-Lacker \cite{CDL16} found the weak equilibria of mean field games subject to common noise by applying the Kakutani-Fan-Glicksberg fixed-point theorem for set-valued functions, as well as a discretization procedure for the common noise.  The weak equilibria can be shown to be strong under additional assumptions. We can also use the stochastic maximum principle to characterize the mean field games problem to the mean field (or McKean-Vlasov) forward-backward stochastic differential equations (FBSDEs). A major advantage of this approach is that the FBSDE is symbolically a finite-dimensional problem, as opposed to solving the master equation which is in-born infinite dimensional. The mean field SDEs were studied in Buckdahn-Li-Peng-Rainer \cite{buckdahn2017mean}. The study of mean field BSDEs was first conducted by Buckdahn-Djehiche-Li-Peng \cite{buckdahn2009meanLim} and Buckdahn-Li-Peng \cite{buckdahn2009mean}. The well-posedness of the coupled mean field FBSDEs was first studied in Carmona-Delarue \cite{carmona2013mean}. Later, Carmona-Delarue \cite{carmona2015forward} relaxed the assumptions and established the well-posedness of FBSDEs under the framework of linear forward dynamics and convex coefficients of the backward equation. Bensoussan-Yam-Zhang \cite{bensoussan2015well} also addressed the well-posedness of FBSDEs with another set of generic assumptions on the coefficients. We also refer to the recent article \cite{homan2023game} which resolved the mean field games problem under the small mean field effect. In the article, the authors provided a brand new set of estimates for establishing the globally well-posedness of the FBSDEs and hereto a new set of estimates for showing the regularity of the value function. The classical well-posedness of master equation was also provided in the paper. We also refer readers to \cite{cosso2023optimal,cosso2021master,djete2022mckean,lauriere2014dynamic,li2012stochastic,PHW,pham2018bellman} for the mean field type control problems.

In this paper, we analyze the problem with a generic dynamics in which the nonlinear drift function $f(x,\mu,\alpha)$ depends not only on the control $\alpha$, but also on the state variable $x$ and mean-field measure $\mu$. Contrary to the simpler cases considered in existing literature, such as $f(x,\mu,\alpha)\equiv\alpha$, our model assumes that drift function can be at most linear-growth but quite arbitrarily non-linear. In addition, we consider a running cost function $g(x,\mu,\alpha)$ that can be non-separable and of quadratic-growth in the entire space of $\R^{d_x}\times \mc{P}_2(\R^{d_x})$ in $(x,\mu)$, and thus the Hamiltonian $\displaystyle h(x,\mu,z)= \min_{\alpha} \left(f(x,\mu,\alpha)\cdot z  + g(x,\mu,\alpha)\right)$ can also be non-separable and of quadratic-growth in the entire space $\R^{d_x}\times \mc{P}_2(\R^{d_x})\times \R^{d_x}$ in $(x,\mu,z)$. In contrast, results in the existing literature mostly consider compact torus space $x\in\mathbb{T}$ with quite arbitrary Hamiltonian (see \cite{cardaliaguet2019master} for example), or only allows Hamiltonians of linear-growth but not of quadratic-growth at the far-end of the space of $\R^{d_x}\times \mc{P}_2(\R^{d_x})$ (see \cite{gangbo2022mean} for example).

More precisely, we assume that $f(x,\mu,\alpha)$ is Lipschitz continuous and grows at most linearly, while the running cost function $g(x,\mu,\alpha)$ is convex in the control $\alpha$ and satisfies a monotonicity condition. This monotonicity condition can be motivated by the positive definiteness of the Schur complement of the Hessian matrix of the Lagrangian in the lifted version and was proposed in our previous paper \cite{bensoussan2023theory}. Even though the calculus in the lifted variables provides useful insights, it creates additional technical difficulties in the mathematical analysis, and is not that effective in analyzing the problem with the generic drift rate; see Section 4 of \cite{bensoussan2023theory} for the detailed discussions. The proposed monotonicity conditions \eqref{positive_g_mu}, \eqref{positive_k_mu} and \eqref{positive_H_mu} are closely related to the Lasry-Lions monotonicity condition and displacement monotonicity condition; these conditions were first introduced in the works of \cite{cardaliaguet2019master,carmona2018probabilistic} and \cite{gangbo2022mean,gangbo2020global}, respectively, which we find to be interesting and important. For precise assumptions on the coefficient functions, refer to Section \ref{assumptions}, while Section \ref{subsec:dis} and Remark \ref{remark_h2} provide more elaborative discussion on the relationship and comparison between our convexity assumptions and other monotonicity conditions.

Under these assumptions, our main results include the global-in-time solvability of the generic mean field games, as well as the global-in-time existence, uniqueness, and classical regularity of the corresponding forward-backward ordinary differential equation (FBODE) system. Theorems \ref{GlobalSol} and \ref{Thm6_4} respectively cover existence and uniqueness of the solution to the FBODE system; while Theorem \ref{thm_master} and \ref{Unique_B} address the respective issue for the master equation.

To demonstrate the usefulness of our newly proposed general theory, we provide a resolution of a non-trivial non-linear-quadratic example with a non-separable Hamiltonian in Section \ref{sec:nonLQ}. This example is unlikely to be covered in the contemporary literature and serves as an illustration of the relevance of our proposed theory.

To solve the generic mean field games (see \eqref{SDE1}-\eqref{valfun}), we follow the approach used in our previous paper \cite{bensoussan2023theory}. Firstly, unlike the simpler cases where the optimal control can be solved explicitly from the first-order condition, in our generic setting of the drift and running cost functions, the optimal control cannot be solved explicitly at all. To still understand better about its behavior, we introduce a cone condition (see \eqref{c_k_0}), which is a property (not an assumption) that our constructed solution of the FBODE system can be shown to fulfill (see \eqref{eq_7_44}). This condition is incorporated with the assumption on the asymptotic behavior of second-order derivatives of the drift function (see \eqref{bdd_d2_f}), which is generic enough and yet guarantees the solvability of the optimal control via the first-order condition (see \eqref{first_order_condition}). To this end, we use a generalized version of the implicit function theorem (see the Appendix of \cite{bensoussan2023theory}) that includes measure arguments.
Secondly, we show that solving the generic mean field games (see \eqref{SDE1}-\eqref{valfun}) can be reduced to solving a FBODE system (see \eqref{FBODE}). 
Thirdly, we use fixed point arguments to prove the local-in-time existence, uniqueness, and regularity of solutions to the general FBODE system \eqref{fbodesystem}. This system is essentially the same as the FBODE system \eqref{FBODE}, except that it imposes a more general terminal data $p(x,\mu)$. The minimal lifespan (that can be guaranteed by the fixed point argument) of local solution depends mainly on the bound of the derivatives of the terminal data $p(x,\mu)$, as well as some other uniform-in-time constants related to the given drift and running cost functions; more discussions are given in Theorems \ref{Thm6_1} and \ref{Thm6_2}.
To obtain the global-in-time classical solution of the FBODE system \eqref{FBODE}, we paste together the local-in-time solutions. The key point is to obtain {\it a priori} uniform-in-time estimates on the derivatives of the backward equation $Z^{t,m}_s(x)$, since $Z^{t,m}_{t_i}(x)$, $i=1,2,...,N+1$ are the terminal data $p$ on the respective sub-interval $[t_{i-1},t_i]$. We can obtain our new crucial {\it a priori} estimates (see Theorem \ref{Crucial_Estimate}) under the monotonicity conditions \eqref{positive_g_mu}, \eqref{positive_k_mu}, and \eqref{positive_H_mu}. 
Using both the local-in-time results and the new crucial {\it a priori} estimates on the derivatives of the backward equation $Z^{t,m}_s(x)$, we can prove our global-in-time well-posedness claim. Most importantly, the current approach we propose can be easily extended to handle generic second-order mean field games and control problems in the presence of Brownian driving noise.

The main contribution of this work is the provision of new crucial {\it a priori} estimates, whose proofs differ from those given in \cite{bensoussan2023theory} as the associated forward-backward ordinary differential equation (FBODE) system is fundamentally different. These estimates are the key to guaranteeing the cone conditions and our global-in-time well-posedness.

Our approach has several direct consequences and advantages. Firstly, by using the push-forward operator $\ot$, we can transform the study of $\mu_s$ in the Wasserstein metric space to $X_s$ in the tangent space of Hilbert space $L^2_m$ at the measure $m$. This allows us to avoid the difficulties caused by the lack of vector space structure in the Wasserstein metric space. Secondly, the optimal control $\wh{\alpha}(x,\mu,z)$ is solved from the first-order condition $\p_\alpha f(x,\mu,\wh{\alpha})\cdot z+\p_\alpha g(x,\mu,\wh{\alpha})=0$. However, in our generic setting, the drift function $f$ can be non-linear in control, and when the magnitude of variable $z$ is large, we are not able to apply the implicit function theorem to guarantee the unique solvability of the optimal control from the first-order condition, even with the convexity in $\alpha$ of $g(x,\mu,\alpha)$. Nonetheless, the proposed ``cone condition'', which our constructed solution of the FBODE system fulfills, suggests that, for any given $(x,\mu)$, only those $z$'s in a bounded subset governed by $(x,\mu)$ are related to the solutions of FBODE. This allows us to apply the implicit function theorem in the related region while solving the FBODE. 
Thirdly, while deriving our new crucial \emph{a priori} estimates, we have to carefully deal with several terms (refer to Equation \eqref{eq_7_26_1}) that could be cancelled out in our previous paper \cite{bensoussan2023theory} on mean field type control problems, but this is no longer possible in the current work. To manage these terms, we employ a monotonicity condition related to both the running cost function and drift function (see Equation \eqref{positive_H_mu}). This condition is analogous to the displacement monotonicity condition proposed in \cite{gangbo2022mean}, and we only require our monotonicity condition being valid within a cone subset.
The monotonicity condition \eqref{positive_H_mu} can be directly deduced from the hypothesis (see \eqref{h3_eq_1}) that the first-order derivative of the drift function in the measure argument cannot be excessively large relative to the convexity of the running cost function. This requirement emerges only when the Hamiltonian is non-separable; for more details, refer to Remark \ref{remark_h2}. 
Finally, the solution of the FBODE system fulfills the ``cone condition'' on the whole time interval $[0,T]$ and satisfies the crucial \emph{a priori} estimates at the same time. Using our new crucial \emph{a priori} estimates of the derivatives of the $Z^{t,m}_s(x)$, we can obtain a uniform lower bound on the minimal time mesh size of local solutions. This allows us to construct a global solution over the whole time interval $[0,T]$ for any fixed large time $T$ by gluing these local solutions together.

The rest of this article is organized as follows.
In Section \ref{sec:ProblemSetting}, we introduce the Wasserstein space of measures, the push-forward operator and the formulation of first-order mean field game. Section \ref{sec:Derivatives} recalls various derivatives of functionals defined on the Wasserstein space.
In Section \ref{sec:assumptions}, we introduce the assumptions used in our results, as well as some direct consequences and useful properties implied by these assumptions. Furthermore, we also compare these assumptions with other assumptions used in the literature.
Section \ref{sec:master} introduces the HJ-FP system, master equation and forward-backward ordinary differential equation (FBODE) system, and explains why in order to solve the original first-order mean field game, it suffices to solve the FBODE system. 
In Section \ref{sec:local}, we prove the existence and regularity of local-in-time solutions of the FBODE system, and then establish the global-in-time solution of the FBODE system \eqref{FBODE} and master equation \eqref{Master_eq} in Section \ref{sec:global} by using some new crucial {\it a priori} estimates. We also establish the uniqueness in Section \ref{sec:unique}.
Finally, in Section \ref{sec:nonLQ}, we provide a non-trivial example of non-linear-quadratic mean field game with non-separable Hamiltonian and nonlinear dynamics that satisfies our assumptions in Theorem \ref{GlobalSol}.

\section{Generic First-Order Mean Field Games}\label{sec:ProblemSetting}

Let $d$ and $p$ be positive integers, and $\mathcal{P}_{p}(\mathbb{R}^{d})$ be the space of all probability measures on $\mathbb{R}^{d}$, each of which has a finite $p$-th order moment, equipped with the Wasserstein metric $W_{p}(\mu,\nu)$ defined by: 
\begin{equation} \label{eq_2_2}
W_{p}(\mu,\nu):=\bigg(\inf_{\pi\in\Pi(\mu,\nu)}\int_{\mathbb{R}^{d}\times \mathbb{R}^{d}}|\xi-\eta|^{p}d\pi(\xi,\eta)\bigg)^{\frac{1}{p}} , 
\end{equation}
where $\Pi(\mu,\nu)$ denotes the set of all joint probability measures
on $\mathbb{R}^{d}\times \mathbb{R}^{d}$ such that the marginals are the probability measures $\mu$ and $\nu$, respectively. Also define $\|\mu\|_p=\left(\int_{\R^{d}}|x|^p\,d\mu(x)\right)^{1/p}$. Clearly, $\mathcal{P}_{p'}(\mathbb{R}^d)\subset \mathcal{P}_{p}(\mathbb{R}^d)$ for any $1\leq p< p'<\infty$.

For any fixed $m \in \mathcal{P}_p(\mathbb{R}^d)$, we denote $L^{p,d_1,d_2}_m:= L^p_m(\mathbb{R}^{d_1};\mathbb{R}^{d_2})$, the set of all measurable maps $\Phi:\R^{d_1}\rightarrow \R^{d_2}$ such that $\int_{\mathbb{R}^{d_1}}|\Phi(x)|^pd m(x)<\infty$. We equip $L^{p,d_1,d_2}_m$ with a norm $\| \cdot \|_{L^{p,d_1,d_2}_m}$ or simply $\| \cdot \|_{L^{p,d}_m}$ when $d_1=d_2=d$:
\begin{align}\label{eq_2_3}
    \|X\|_{L^{p,d_1,d_2}_m}:= \bigg(\int_{\mathbb{R}^{d_1}}|X(x)|^pd m(x)\bigg)^{\frac{1}{p}},\ \forall\ X\in L^{p,d_1,d_2}_m;
\end{align} 
clearly, $L^{p,d_1,d_2}_m$ is a Banach space and $L^{p',d_1,d_2}_m\subset L^{p,d_1,d_2}_m$ for any $1\leq p<p'<\infty$.
For $p=2$, we further equip $L^{2,d_1,d_2}_m$ with an inner product:
\begin{align}
    \langle X,Y\rangle_{L^{2,d_1,d_2}_m}:= \int_{\mathbb{R}^{d_1}} X(x)\cdot Y(x)d m(x),\ \forall\ X,Y\in L^{2,d_1,d_2}_m;
\end{align}
clearly, $L^{2,d_1,d_2}_m$ is a Hilbert space. In addition, we can regard the space $L^{2,d_1,d_2}_m$ as a tangent space attached to a point $m$ in the Wasserstein metric space $\mathcal{P}_2(\mathbb{R}^d)$, and we use subscript $m$ to indicate this nature.
\begin{definition}\label{defot}
For any given probability measure $m\in\mathcal{P}_2(\mathbb{R}^d)$ and measurable map $X\in L^{2,d}_m$, define another probability measure $ X \ot m  \in \mathcal{P}_2(\mathbb{R}^d)$ so that for every measurable function $\phi:\mathbb{R}^d\to\mathbb{R}$ such that ${\displaystyle \sup_{x\in\R^d}} \dfrac{|\phi(x)|}{1+|x|^2}<\infty$, 
\begin{align}
    \int_{\R^d} \phi(x)d( X \ot m )(x)=\int_{\R^d} \phi(X(x))dm(x).
\end{align}
\end{definition}
This new probability measure $ X \ot m$ is actually the push-forward measure of $m$ under the map $X(\cdot)$.

In this article, we denote the state by $x\in\mathbb{R}^{d_x}$ and the control by $\alpha\in\mathbb{R}^{d_\alpha}$,
also consider continuous coefficient and cost functions:
\begin{align*}
&f(x,\mu,\alpha)=(f_1,...,f_{d_x})(x,\mu,\alpha): \mathbb{R}^{d_x}\times\mathcal{P}_2(\mathbb{R}^{d_x})\times\mathbb{R}^{d_\alpha}\to\mathbb{R}^{d_x};\\
&g(x,\mu,\alpha): \mathbb{R}^{d_x}\times\mathcal{P}_2(\mathbb{R}^{d_x})\times\mathbb{R}^{d_\alpha} \to\mathbb{R};\\ &k(x,\mu) : \mathbb{R}^{d_x}\times\mathcal{P}_2(\mathbb{R}^{d_x})\to\mathbb{R};
\end{align*}
where these differentiable, in $x$ and $\alpha$, functions can be allowed to depend on time, but we here omit this dependence to avoid unnecessary technicalities; yet all the following discussions remain valid with time dependent coefficients and cost functions. Besides, $f$ plays the role as a drift function, $g$ is the running cost, and $k$ stands for the terminal cost functional. 
The argument $\mu$ is the place for the mean field term of the state. 

To describe the generic first-order mean field game, let us first introduce the notations for the dynamics and objective function as follows:  given an arbitrary process of probability measures $\wt m:=(\wt m_\tau)_{\tau\in[0,T]}$ with $\sup_{\tau\in[0,T]}\|\wt m_t\|_1<\infty$ and any control process $\alpha:=(\alpha_\tau)_{\tau\in[0,T]}$ with $\int_0^T |\alpha_\tau|^2 d\tau<\infty$, for any fixed $t\in[0,T]$ and $x\in\R^{d_x}$,
consider the following dynamics of controlled state $X^{t,\wt m,\alpha}_s(x)$:
\begin{align}\label{SDE1}
X^{t,\wt m,\alpha}_s(x) =&\ \ x+\int_t^s f\left(X^{t,\wt m,\alpha}_\tau(x),\wt m_\tau,\alpha_\tau\right)d\tau,\ s\in[t,T].
\end{align}
The existence and uniqueness of solutions to \eqref{SDE1} can be proven under standard Lipschitz and linear-growth conditions of $f$; moreover this solution satisfies $\sup_{s\in[t,T]}\left|X^{t,\wt m,\alpha}_s(x)\right|\leq C(1+|x|)<\infty$, where the constant $C$ depends on $T$. 
Now, we define the objective function 
\begin{align}\label{cost}
j(t,x,\wt m,\alpha) :=&\  \int^T_t  g(X^{t,\wt m,\alpha}_s(x),\wt m_s,\alpha_s)ds +k(X^{t,\wt m,\alpha}_T(x),\wt m_T),\text{ for }t\in[0,T],
\end{align}
and the optimal control process $\alpha^*:=(\alpha^*_s)_{s\in[t,T]}$ as 
\begin{align}\label{optimal_alpha}
 \alpha^*(t,x,\wt m):=\argmin_{\alpha\in\mathbb{A}} j(t,x,\wt m,\alpha),
\end{align}
where the set of all admissible controls is defined as $\mathbb{A}:=L^2([t,T];\R^{d_\alpha})$.

The main purpose of this work is to solve the generic first order mean field game, which can be stated as follows: for any given initial data $m_0\in\mathcal{P}_2(\R^{d_x})$ at the initial time $t=0$, we aim at finding the Nash equilibrium probability measure $\wh m:=(\wh m_s)_{s\in[0,T]}$ such that the corresponding optimal control process $\alpha^*(0,x,\wh m)$, obtained by minimizing the objective function $j$ under the dynamics \eqref{SDE1} with $\wt m_\tau:=\wh m_\tau$, will generate the same probability measure $\wh m$ via \eqref{SDE1} with $t=0$, namely
\begin{align}\label{optimal_probability}
\wh{m}_s=X^{0,\wh{m},*}_s\ot m_0,
\end{align}
where the solution map $X^{t,\wh{m},*}_s$ is defined by solving
\begin{align}\label{SDE_star}
X^{t,\wh m,*}_s(x) =&\ \ x+\int_t^s f\left(X^{t,\wh m,*}_\tau(x),\wh m_\tau,\alpha^*_\tau (t, x,\wh m)\right)d\tau,\ s\in[t,T].
\end{align}
Furthermore, we shall also solve for the corresponding value function
\begin{align}\label{valfun}
v(t,x):=j(t,x,\wh{m},\alpha^*(t,x,\wh m)),\text{ for any }(t,x)\in[0,T]\times \R^{d_x}.
\end{align} 
In our knowledge, our solution will provide a settlement of a standing problem in this one of the most popular research areas of mean field games; indeed, we shall study the well-posedness of the generic first-order mean field game \eqref{SDE1}-\eqref{valfun} under Assumptions $\bf{(a1)}$-$\bf{(a3)}$,
to be specified in Section \ref{assumptions}.

\section{Derivatives in Wasserstein Metric Space}\label{sec:Derivatives}
We first list a few useful definitions related to the derivatives of functions in measure arguments. 

\begin{definition}\label{first_LFD}
(Linear Functional Derivative \cite{carmona2018probabilistic}). A function $f: \mu\in \mathcal{P}_2(\mathbb{R}^{d})\mapsto f(\mu)\in \mathbb{R}$ is said to have a {\it linear functional derivative} if there exists a function
\begin{align*}
    \/{\delta f}{\delta \mu}:\mathcal{P}_2(\mathbb{R}^{d})\times \mathbb{R}^{d} \to  \mathbb{R}
\end{align*}
which is continuous with respect to the product topology of $\mathcal{P}_2(\R^d)$ and $\R^d$ so that the followings hold:
\begin{enumerate}
    \item [(i)]
there is another function $c: \mathcal{P}_2(\mathbb{R}^{d})\rightarrow [0,\infty)$ which is bounded on any bounded subsets of $\mathcal{P}_2(\R^d)$, such that
\begin{align}
\left| \/{\delta f}{\delta \mu}(\mu)(x)\right|\leq c(\mu)(1+|x|^2),\text{ for any }\mu\in\mathcal{P}_2(\R^d)\text{ and }x\in \R^d;
\end{align}
\item [(ii)] for all $\nu_1$, $\nu_2\in \mathcal{P}_2(\mathbb{R}^{d})$, it also holds that
\begin{align}\label{delta_mu_f}
    f(\nu_1)-f(\nu_2) = \int_0^1 \int_{\mathbb{R}^{d}}\/{\delta f}{\delta \mu}(\theta \nu_1+(1-\theta)\nu_2)(x)d(\nu_1-\nu_2)(x)d\theta.
\end{align}
\end{enumerate}
Note that $\/{\delta f}{\delta \mu}$ defined here is unique only up to an additive constant, and the following normalization condition is taken in the rest of this article:
\begin{align}\label{eq_3_5_n}
\int_{\mathbb{R}^{d}} \/{\delta f}{\delta \mu}(\mu)(x)d\mu(x)=0,
\end{align}
which in turn ensures the functional derivative of a constant function is always zero. 

\end{definition}

Let $\mathcal{H}$ and $\mathcal{H}'$ be two Hilbert spaces and $\mathcal{L}(\mathcal{H};\mathcal{H}')$ be the set of all bounded linear operators from $\mathcal{H}$ to $\mathcal{H}'$ equipped with the usual operator norm: 
\begin{align}
	\|F\|_{\mathcal{L}(\mathcal{H};\mathcal{H}')} := \sup_{X\in\mathcal{H},X\neq 0} \/{\|F(X)\|_\mathcal{H'}}{\|X\|_\mathcal{H}}.
\end{align}
\begin{definition}
(\textit{Fr\'echet differentiability}). A function $F: \mathcal{H}\to\mathcal{H'}$ is said to be \textit{Fr\'echet differentiable} at $X$ if there is a bounded linear operator $D_XF (X) \in \mathcal{L}(\mathcal{H};\mathcal{H}')$, such that for all $Y \in \mathcal{H}$,
\begin{align*}
	\lim_{\|Y\|_\mathcal{H}\to 0}\/{\|F (X+Y)-F (X)-D_XF (X)(Y)\|_\mathcal{H'}}{\|Y\|_\mathcal{H}} = 0.
\end{align*}
\end{definition}
Consider a probability space $\left(\Omega, \mathcal{F},\mathbb{P}\right)$ 
 and all its square-integrable random $\R^d$-vectors, namely
\begin{align*}
\mathcal{H}^d:=L^{2}(\Omega,\mathcal{F},\mathbb{P};\mathbb{R}^{d}),
\end{align*}
which is equipped with an inner product:
\begin{align}
    \langle X,Y\rangle_{\mathcal{H}^d}:= \int_{\Omega} X(\omega)\cdot Y(\omega)d\mathbb{P}(\omega),\ \text{for all}\ X,Y\in\mathcal{H}^d.
\end{align} 
\begin{definition}\label{def_L_diff}
({\it L-differentiability} at a $\mu_0\in \mathcal{P}_2(\mathbb{R}^{d}$) \cite{bensoussan2017interpretation,carmona2018probabilistic}). A function $f: \mu\in \mathcal{P}_2(\mathbb{R}^{d})\mapsto f(\mu)\in \mathbb{R}$ is said to be {\it $L$-differentiable} at a $\mu_0 \in \mathcal{P}_2(\mathbb{R}^{d})$ if there exists a $X_0\in\mathcal{H}^{d}$ with the law $\mu_0$, denoted by $\mathbb{L}_{X_0} = \mu_0$, such that its lifted version $F (X):= f(\mathbb{L}_X)$ is Fr\'echet differentiable at $X_0$; and we also call this Fr\'echet derivative $D_X F(X_0)(\cdot)\in\mathcal{L}(\mathcal{H}^d;\R)$ as the {\it $L$-derivative} of $f(\mu)$ at $\mu_0$, denote it by
$\partial_\mu f(\mu_0)$.
\end{definition}
Note that this {\it L-derivative} $\partial_\mu f(\mu_0)$ is uniquely defined, except $\mathbb{P}$-null set of points, in the sense that its definition is independent of the choice of $X_0$ corresponding to the same $\mu_0$; also see Proposition 5.24 in \cite{carmona2018probabilistic} and \cite{bensoussan2017interpretation} for details. 
Generally, the function $f$ is {\it $L$-differentiable} if it is {\it $L$-differentiable} at every $\mu\in \mathcal{P}_2(\mathbb{R}^{d})$, so that the {\it $L$-derivative} as a function denoted by $\partial_\mu f(\cdot)(\cdot):(\mu,x)\in \mathcal{P}_2(\mathbb{R}^{d})\times \R^d \mapsto \partial_\mu f(\mu)(x)\in\mathbb{R}^{d}$ is jointly measurable. Certainly, we also have
\begin{equation}
\partial_{\mu}f(\mathbb{L}_X)(X)=D_{X}F(X) \:\text{ a.s.}, \quad \forall X\,\in\mathcal{H}^{d}. \label{eq:4-110}
\end{equation}

Next, we have the following properties for linear functional derivatives and $L$-derivatives. 

\begin{prpstn}\label{prop_3_7}(Proposition 5.51 in \cite{carmona2018probabilistic}) 
Suppose that $f: \mu\in \mathcal{P}_2(\mathbb{R}^{d})\mapsto f(\mu)\in \mathbb{R}$ is {\it $L$-differentiable} on $\mathcal{P}_2(\R^d)$ and its Fr\'echet derivative $D_{X}F(\cdot)$ is uniformly Lipschitz in its argument. Also assume that its $L$-derivative $\partial_\mu f(\cdot)(\cdot):(\mu,x)\in \mathcal{P}_2(\R^d)\times \R^d\mapsto \partial_\mu f(\mu)(x)\in \R$ is jointly continuous with respect to the natural product topology. Then $f$ has a linear functional derivative $\/{\delta f}{\delta \mu}(\cdot)(\cdot)$ such that 
\begin{equation}
\partial_\mu f(\mu)(x)=\partial_x \/{\delta f}{\delta \mu}(\mu)(x),\text{ for any }   \mu\in \mathcal{P}_2(\mathbb{R}^{d}). \label{eq:4-111}
\end{equation}
\end{prpstn}
Also see \cite{bensoussan2017interpretation} for \eqref{eq:4-111} for an alternative derivation.
\begin{definition}
A function $f:\mu\in\mathcal{P}_2(\mathbb{R}^{d}) \mapsto f(\mu)\in \mathbb{R}$ is called regularly linear-functionally differentiable 
in $\mu\in\mathcal{P}_2(\mathbb{R}^{d})$ if\\
(i) $f$ has a linear functional derivative $\/{\delta f}{\delta \mu}(\mu)(x)$ which, for each $\mu \in \mathcal{P}_{2}(\mathbb{R}^{d})$, is differentiable in $x$;\\
(ii) $\partial_x\/{\delta f}{\delta \mu}(\cdot)(\cdot):\mathcal{P}_2(\R^d)\times \R^d\rightarrow \R^d$ is jointly continuous with respect to the natural product topology such that
\begin{equation}
\left|\partial_x\/{\delta f}{\delta \mu}(\mu)(x)\right|\leq c(\mu)(1+|x|),\label{eq:4-113}
\end{equation}
where $c(\mu)$ is another function in $\mu$ being bounded on any bounded subsets of $\mathcal{P}_{2}(\mathbb{R}^{d})$.
\end{definition}
Without the cause of ambiguity, we simply call $f$ to be regularly differentiable in $\mu\in\mathcal{P}_2(\mathbb{R}^{d})$. The derivative $\p_x\frac{\delta f}{\delta \mu}:\mathcal{P}_2(\R^d)\times \R^d\rightarrow \R^d$ is sometimes called the {\it intrinsic derivative}, see \cite{cardaliaguet2019master} for example.

\begin{prpstn}\label{Ldif} (Propositions 5.44 and 5.48 in \cite{carmona2018probabilistic})
Suppose that  $f:\mathcal{P}_2(\mathbb{R}^{d}) \to \mathbb{R}$ is regularly differentiable in $\mu\in\mathcal{P}_2(\mathbb{R}^{d})$. Then, $f$ is $L$-differentiable such that \eqref{eq:4-111} holds. Furthermore, for any $\mu\in\mathcal{P}_2(\R^d)$, and any sequence $\mu_n$ with $W_{2}(\mu_n,\mu)\rightarrow0$ as $n \rightarrow \infty$, we also have  
\begin{equation}\label{prop33}
\frac{f(\mu_n)-f(\mu)-\int_{\mathbb{R}^{d}}\/{\delta f}{\delta \mu}(\mu)(x) d(\mu_n-\mu)(x)}{W_{2}(\mu_n,\mu)}\rightarrow0.
\end{equation}
\end{prpstn}


Having laid down the basic principles, we can now shift our focus to the first-order mean field game described by equations \eqref{SDE1}-\eqref{valfun}. Within this context, the push-forward mapping $X(\cdot)$, that resides in $L^{2,d_x}_{m}$ and is governed by \eqref{FBODE}, is deterministic. This deterministic nature is due to the absence of any additional randomness, with the only source of variability being carried over from the initial data. The present approach is readily extended to the second-order cases with the presence of Brownian noise, which is similar to our previous work of \cite{bensoussan2023theory}.

%

\section{Assumptions and Preliminary Results}\label{sec:assumptions}
\subsection{Assumptions}\label{assumptions}

We aim at solving the first-order mean field game \eqref{SDE1}-\eqref{valfun} under the following assumptions.\\
$\bf{(a1)}$ The drift function $f:(x,\mu,\alpha)\in \R^{d_x}\times \mathcal{P}_{2}(\mathbb{R}^{d_x})\times \R^{d_\alpha}\mapsto \R^{d_x}$ is differentiable in $x\in \R^{d_x}$ and $\alpha\in \R^{d_\alpha}$; and $f(x,\mu,\alpha)$, $\p_x f(x,\mu,\alpha)$ and $\p_\alpha f(x,\mu,\alpha)$ are also regularly differentiable in $\mu\in\mathcal{P}_2(\R^{d_x})$, as well as all of the following derivatives exist and are jointly continuous in their corresponding arguments, and they also satisfy the following estimates, respectively:\\ 
$(i)$ For any $\xi=(\xi_1,...,\xi_{d_x})\in\R^{d_x}$,\\ 
\begin{align}\label{positive_f}
&\sum_{i=1}^{d_x}\sum_{j=1}^{d_x}\bigg(\xi_i \p_\alpha f_i(x,\mu,\alpha)\bigg)\cdot \bigg(\xi_j \p_\alpha f_j(x,\mu,\alpha)\bigg) \geq \lambda_f|\xi|^2,
\end{align}
where $f_i$ is the $i$-th component of $f$.\\
$(ii)$ Boundedness on first-order derivatives:  
\begin{align}\no
&\sup_{(x,\mu,\alpha)\in \R^{d_x}\times \mathcal{P}_{2}(\mathbb{R}^{d_x})\times \R^{d_\alpha}}
\|\p_\alpha f(x,\mu,\alpha)\|_{\mathcal{L}(\R^{d_\alpha};\R^{d_x})}\leq \Lambda_1;\\\no 
\ &\sup_{(x,\mu,\alpha,\wt{x})\in \R^{d_x}\times \mathcal{P}_{2}(\mathbb{R}^{d_x})\times \R^{d_\alpha}\times\R^{d_x}}\|\p_\mu f(x,\mu,\alpha)(\wt{x})\|_{\mathcal{L}(\R^{d_x};\R^{d_x})}\leq \Lambda_2;\\
\label{bdd_d1_f}&\sup_{(x,\mu,\alpha)\in \R^{d_x}\times \mathcal{P}_{2}(\mathbb{R}^{d_x})\times \R^{d_\alpha}}\|\p_x f(x,\mu,\alpha)\|_{\mathcal{L}(\R^{d_x};\R^{d_x})}\leq \Lambda_3;
\end{align}
also define $\Lambda_f:=\max\{\Lambda_1,\Lambda_2,\Lambda_3\}$.\\ 
$(iii)$ Boundedness on second-order derivatives:
\begin{align}\no
&\sup_{(x,\mu,\alpha,\wt{x})\in \R^{d_x}\times \mathcal{P}_{2}(\mathbb{R}^{d_x})\times \R^{d_\alpha}\times\R^{d_x}}
\Big\{\|\p_\alpha\p_\alpha f(x,\mu,\alpha)\|_{\mathcal{L}(\R^{d_\alpha}\times \R^{d_\alpha};\R^{d_x})}\\
\no&\ \ \ \ \ \ \ \ \ \ \ \ \ \ \ \ \ \ \ \ \ \ \ \ \ \ \ \ \ \ \ \ \ \ \ \ \ \ \ \ \vee\|\p_\alpha\p_x f(x,\mu,\alpha)\|_{\mathcal{L}(\R^{d_x}\times \R^{d_\alpha};\R^{d_x})}\\
\no&\ \ \ \ \ \ \ \ \ \ \ \ \ \ \ \ \ \ \ \ \ \ \ \ \ \ \ \ \ \ \ \ \ \ \ \ \ \ \ \ \vee\|\p_\mu\p_\alpha f(x,\mu,\alpha)(\wt{x})\|_{\mathcal{L}(\R^{d_\alpha}\times \R^{d_x};\R^{d_x})}\\
\no&\ \ \ \ \ \ \ \ \ \ \ \ \ \ \ \ \ \ \ \ \ \ \ \ \ \ \ \ \ \ \ \ \ \ \ \ \ \ \ \ \vee\|\p_x\p_x f(x,\mu,\alpha)\|_{\mathcal{L}(\R^{d_x}\times \R^{d_x};\R^{d_x})}
\\
\label{bdd_d2_f}&\ \ \ \ \ \ \ \ \ \ \ \ \ \ \ \ \ \ \ \ \ \ \ \ \ \ \ \ \ \ \ \ \ \ \ \ \ \ \ \ \vee\|\p_\mu \p_xf(x,\mu,\alpha)(\wt{x})\|_{\mathcal{L}(\R^{d_x}\times \R^{d_x};\R^{d_x})}
\Big\}\cdot (1+|x|+\|\mu\|_1)\leq \wb{l}_f.
\end{align}
$(iv)$ The above second-order derivatives are jointly Lipschitz continuous in their corresponding arguments with Lipschitz constant decaying at the proximity of infinity: for any $x,\,x',\,\wt{x},\,\wt x'\in\R^{d_x}$, $\alpha,\,\alpha'\in\R^{d_\alpha}$, $\mu,\,\mu'\in\mc{P}_2(\R^{d_x})$,
\begin{align}\no
a)\ \ &\big\|\p_\mu\p_x f(x,\mu,\alpha)(\wt{x})-\p_\mu\p_x f(x',\mu',\alpha')(\wt{x}')\big\|_{\mathcal{L}(\R^{d_x}\times \R^{d_x};\R^{d_x})}\\
\leq &\frac{\Lip_f}{1+\max\{|x|,|x'|\}+\max\{\|\mu\|_1,\|\mu'\|_1\}}\Big(|x-x'|+W_2(\mu,\mu')+|\alpha-\alpha'|+|\wt{x}-\wt{x}'|\Big),\\\no 
b)\ \ &\big\|\p_\alpha\p_\alpha f(x,\mu,\alpha)-\p_\alpha\p_\alpha f(x',\mu',\alpha')\big\|_{\mathcal{L}(\R^{d_\alpha}\times \R^{d_\alpha};\R^{d_x})}\\
\leq &\frac{\Lip_f}{1+\max\{|x|,|x'|\}+\max\{\|\mu\|_1,\|\mu'\|_1\}}\Big(|x-x'|+W_2(\mu,\mu')+|\alpha-\alpha'|\Big),
\end{align}
and so are for $\p_\mu\p_\alpha f(x,\mu,\alpha)(\wt{x})$, $\p_\alpha\p_x f(x,\mu,\alpha)$ and $\p_x\p_x f(x,\mu,\alpha)$.\\
Here $\lambda_f$, $\Lambda_1$, $\Lambda_2$, $\Lambda_3$, $\wb{l}_f$ and $\Lip_f$ are some positive finite constants, and $\|\cdot\|_{\mathcal{L}(\mathcal{A};\mathcal{B})}$ is the operator norm of a linear mapping from $\mathcal{A}$ to $\mathcal{B}$.\\
$\bf{(a2)}$ The running cost function $g:(x,\mu,\alpha)\in \R^{d_x}\times \mathcal{P}_{2}(\mathbb{R}^{d_x})\times \R^{d_\alpha}\mapsto \R$ is differentiable in $x\in \R^{d_x}$ and $\alpha\in \R^{d_\alpha}$; and $g(x,\mu,\alpha)$, $\p_x g(x,\mu,\alpha)$ and $\p_\alpha g(x,\mu,\alpha)$ are regularly differentiable in $\mu\in\mathcal{P}_2(\R^{d_x})$, as well as 
all the second-order derivatives below are jointly Lipschitz continuous in their respective arguments, and they also satisfy the following estimates: \\
$(i)$ For all $\eta\in\R^{d_\alpha}$, 
\begin{align}
\label{positive_g_alpha}
\eta^\top\p_\alpha\p_\alpha g(x,\mu,\alpha)\eta\geq \lambda_g|\eta|^2.
\end{align}
$(ii)$ For all $\xi\in\R^{d_x}$ and any $\mu,\,\mu'\in\mathcal{P}_2(\R^{d_x})$,
\begin{align}\label{positive_g_x}
& a)\ \xi^\top\p_x\p_x g(x,\mu,\alpha)\xi\geq \lambda_{g}|\xi|^2>0,\\\label{positive_g_mu}
& b)\ \int_{\R^{d_x}} \Big(g(x,\mu',\alpha)-g(x,\mu,\alpha)\Big)d\big(\mu'-\mu\big)(x)\geq -l_g \left(\int_{\R} yd(\mu-\mu')(y)\right)^2 \text{ with }l_g<\lambda_g.
\end{align}\normalsize
\\$(iii)$ We have the following bounds on the second-order derivatives of $g$:
\begin{align}\no
\ a)\ &\sup_{(x,\mu,\alpha,\wt{x})\in \R^{d_x}\times \mathcal{P}_{2}(\mathbb{R}^{d_x})\times \R^{d_\alpha}\times\R^{d_x}}
\Big\{\|\p_\alpha\p_\alpha g(x,\mu,\alpha)\|_{\mathcal{L}(\R^{d_\alpha}\times \R^{d_\alpha};\R)}\vee\|\p_\mu \p_xg(x,\mu,\alpha)(\wt{x})\|_{\mathcal{L}(\R^{d_x}\times \R^{d_x};\R)}\\
\label{bdd_d2_g_1}&\ \ \ \ \ \ \ \ \ \ \ \ \ \ \ \ \ \ \ \ \ \ \ \ \ \ \ \ \ \ \ \ \ \ \ \ \ \ \ \ \vee\|\p_x\p_x g(x,\mu,\alpha)\|_{\mathcal{L}(\R^{d_x}\times \R^{d_x};\R)}\Big\}\leq \Lambda_g,\\
\label{bdd_d2_g_2} b)\ &\sup_{(x,\mu,\alpha,\wt{x})\in \R^{d_x}\times \mathcal{P}_{2}(\mathbb{R}^{d_x})\times \R^{d_\alpha}\times\R^{d_x}}
\Big\{\|\p_x\p_\alpha g(x,\mu,\alpha)\|_{\mathcal{L}(\R^{d_\alpha}\times \R^{d_x};\R)}\vee\|\p_\mu \p_\alpha g(x,\mu,\alpha)(\wt{x})\|_{\mathcal{L}(\R^{d_\alpha}\times \R^{d_x};\R)}\Big\}\leq \wb{l}_g.
\end{align}
Here $\lambda_g$, $\Lambda_g$, $\wb{l}_g$, $\lambda_1$ and $\lambda_2$ are some positive finite constants, and $l_{g}$ is a constant which can be negative.\\
$\bf{(a3)}$ The terminal cost function $k:(x,\mu)\in \R^{d_x}\times \mathcal{P}_{2}(\mathbb{R}^{d_x})\mapsto \R$ is differentiable in $x\in \R^{d_x}$, and $k(x,\mu)$ and $\p_x k(x,\mu)$ are regularly differentiable in $\mu\in\mathcal{P}_2(\R^{d_x})$, as well as all 
the second-order derivatives below are jointly Lipschitz continuous in their respective arguments; also these derivatives satisfy the following estimates: 
\\$(i)$ For all $\xi\in\R^{d_x}$ and any $\mu,\,\mu'\in\mathcal{P}_2(\R^{d_x})$,
\begin{align}\label{positive_k}
& a)\ \xi^\top\p_x\p_x k(x,\mu)\xi\geq \lambda_{k}|\xi|^2>0,\\\label{positive_k_mu}
& b)\ \int_{\R^{d_x}} \Big(k(x,\mu')-k(x,\mu)\Big)d\big(\mu'-\mu\big)(x)\geq -l_k \left(\int_{\R} yd(\mu-\mu')(y)\right)^2 \text{ with }l_k<\lambda_k.
\end{align}
\\$(ii)$ We have the following bounds on the second-order derivatives of $k$:
\begin{align}
\label{bdd_d2_k_1}\ &\sup_{(x,\mu,\wt{x})\in \R^{d_x}\times \mathcal{P}_{2}(\mathbb{R}^{d_x})}
\Big\{\|\p_\mu \p_xk(x,\mu)(\wt{x})\|_{\mathcal{L}(\R^{d_x}\times \R^{d_x};\R)}\vee\|\p_x\p_x k(x,\mu)\|_{\mathcal{L}(\R^{d_x}\times \R^{d_x};\R)}\Big\}\leq \Lambda_k.
\end{align}
Here $\lambda_k$ and $\Lambda_k$ are some positive finite constants, and $l_k$ is a constant which can be negative.
\begin{remark}
Compared with the assumptions on coefficient functions $f$, $g$ and $k$ in our previous paper \cite{bensoussan2023theory}, we do not require the existence of $\p_\mu\p_\mu f(x,\mu,\alpha)$, $\p_\mu\p_\mu g(x,\mu,\alpha)$ and $\p_\mu\p_\mu k(x,\mu)$ and their bounds. This is due to the absence of the terms $\p_\mu f(x,\mu,\alpha)$ and $\p_\mu g(x,\mu,\alpha)$ on the right-hand side of the FBODE system \eqref{fbodesystem} in comparison with (7.1) of \cite{bensoussan2023theory}. In terms of problem setting, this is due to the fact that the mean-field term in the underlying mean-field game problem represents the probability measure of all other agents' states, which is an exogenous measure and remains unchanged when we perturb the control variable in the proof of Pontryagin's maximum principle.
\end{remark}

\subsection{Properties of Coefficient Functions and the Optimal Control}\label{properties}
Under Assumptions $\bf{(a1)}$-$\bf{(a3)}$, we have the following useful properties which play a vital role in developing the general theory. All the proofs in this section are almost the same as that in Appendix of our previous paper \cite{bensoussan2023theory}, since these properties are only related to the coefficient functions themselves and the first-order condition, and are independent of the master equation or FBODE system \eqref{FBODE}. Therefore, we omit most details here.

Before delving into the properties of coefficient functions and controls, we first introduce a useful property regarding the Lipschitz continuity of functions defined on $\mathcal{P}_2(\R^{d_x})$. This property is used in deriving the properties of coefficient functions, controls, and solutions of the FBODE system \eqref{FBODE}.
\begin{prpstn}\label{A0} 
For a function $h:\mu\in\mathcal{P}_2(\R^{d_x})\mapsto h(\mu)\in\R^l$ (the dimension $l$ can be any positive integer such as $1$ or $d_x$) having a linear functional derivative $\frac{\delta}{\delta \mu}h(\mu)(\wt{x})$ which is also continuously differentiable in $\wt{x}$ and satisfies the linear growth of $\displaystyle\sup_{\mu\in\mathcal{P}_2(\R^{d_x})}\Big|\p_{\wt{x}} \frac{\delta}{\delta \mu}h(\mu)(\wt{x})\Big|\leq l_1+l_2|\wt{x}|$ for some positive constants $l_1$ and $l_2$. Then we have \begin{align}\label{eq_9_2}
|h(\mu)-h(\mu')|\leq  l_1W_1(\mu,\mu')+l_2\left( \|\mu\|_2+\|\mu'\|_2 \right) W_2(\mu,\mu').
\end{align}
Particularly, when $l_2=0$, we obtain
\begin{align}\label{eq_9_3}
|h(\mu)-h(\mu')|\leq  l_1  W_1(\mu,\mu').
\end{align}
\end{prpstn}

\begin{prpstn}\label{A1} 
The drift function $f(x,\mu,\alpha)$ is of linear-growth and Lipschitz continuous, that is, there exists a positive constant $L_f$, which can be taken as $L_f:=\max\{\Lambda_f,|f(0,\delta_0,0)|\}$\footnote{$\delta_0$ is Dirac measure at $0$.}, such that 
\begin{align}\label{ligf}
&|f(x,\mu,\alpha)|\leq L_f \left(1+|x|+ \|\mu\|_1+|\alpha| \right);\\\label{Lipf}
&\left|f(x,\mu,\alpha) - f(x',\mu',\alpha')\right|\leq L_f \big(|x-x'|+W_1(\mu,\mu')+|\alpha-\alpha'|\big).
\end{align} 
\end{prpstn}

\begin{prpstn}\label{A2} 
(Bounding first-order derivatives) We have the following estimates:\\
\begin{align}
\label{Lipg_x}
&\left|\p_x g(x,\mu,\alpha) - \p_x g(x',\mu',\alpha')\right|\leq \max\{\Lambda_g,\wb{l}_g\} \left(W_1(\mu,\mu')+|x-x'|+|\alpha-\alpha'|\right);\\\label{Lipg_alpha}
&\left|\p_\alpha g(x,\mu,\alpha) - \p_\alpha g(x',\mu',\alpha')\right|\leq \max\{\Lambda_g,\wb{l}_g\} \left(W_1(\mu,\mu')+|x-x'|+|\alpha-\alpha'|\right);
\end{align}
\begin{align}\no
&\sup_{x\in\R^{d_x},\,\mu\in\mathcal{P}_2(\R^{d_x}),\,\alpha\in\R^{d_\alpha}}\frac{\left|\p_x g(x,\mu,\alpha)\right|}{1+|x|+ \|\mu\|_1+|\alpha|}
\vee\frac{\left|\p_\alpha g(x,\mu,\alpha)\right|}{1+|x|+ \|\mu\|_1+|\alpha|}\\\label{LGgDerivatives}
\leq&\ L_g:=\max\Big\{|\p_x g(0,\delta_0,0)|,|\p_\alpha g(0,\delta_0,0)|,\Lambda_g,\wb{l}_g\Big\};
\end{align}
and 
\begin{align}\label{LGkDerivatives}
\sup_{x\in\R^{d_x},\,\mu\in\mathcal{P}_2(\R^{d_x})}\frac{\left|\p_x k(x,\mu)\right|}{1+|x|+ \|\mu\|_1}\leq L_k:=\max\Big\{|\p_x k(0,\delta_0)|,\Lambda_k\Big\}.
\end{align}  
\end{prpstn}
\begin{prpstn}\label{h3}
For each fixed $(x,\mu,z)\in \R^{d_x}\times \mathcal{P}_2(\R^{d_x})\times \R^{d_x}$, we have $\displaystyle\lim_{|\alpha|\to \infty}f(x,\mu,\alpha)\cdot z+g(x,\mu,\alpha)\to \infty$. 
\end{prpstn}
Define, a constant 
\begin{align}\label{k_0}
k_0:=4\max\left\{\Lambda_k,L^*_0\right\}>0,
\end{align}
where $L^*_0$ will be defined later in \eqref{L_star_0} and it is a positive constant depending only on $\lambda_f$, $\Lambda_1$, $\Lambda_2$, $\Lambda_3$, $\lambda_g$, $\Lambda_g$, $l_g$, $\lambda_k$, $\Lambda_k$ and $l_k$.
In fact, we can show later in Remark \ref{remark_lower_bdd_L_star_0} that $\Lambda_k<L^*_0$ and thus $k_0=4 L^*_0$.
 We further define a cone set:
\begin{align}\label{c_k_0}
c_{k_0}:=&\ \left\{(x,\mu,z)\in \R^{d_x}\times \mathcal{P}_2(\R^{d_x})\times \R^{d_x}:|z|\leq \frac{1}{2}k_0\left(1+|x|+\|\mu\|_1\right)\right\},
\end{align}
for the constant $k_0>0$ as defined in \eqref{k_0}; remind that this cone needs not to be convex at all. In fact, Theorem \ref{Thm6_1} addresses the local existence of solution over small time interval, while Theorem \ref{GlobalSol} resolves the global existence of solution over an arbitrarily long time horizon. Both theorems ensure that $\big(X^{t,m}_s(x),X^{t,m}_s\ot m,Z^{t,m}_s(x)\big)\in c_{k_0}$, where $\big(X^{t,m}_s(x),Z^{t,m}_s(x)\big)$ is the solution pair of the FBODE system \eqref{FBODE}. This property is crucial in ensuring the unique existence of the optimal control $\wh{\alpha}(x,\mu,z)$ stated in Proposition \ref{A5}. It is worth noting that this additional growth confinement, which will be stated in the proposition's hypothesis, is indeed justifiable.
\begin{prpstn}\label{A5}
Suppose further that 
\begin{align}\label{p_aa_f}
\sup_{(x,\mu,\alpha)\in \R^{d_x}\times \mathcal{P}_{2}(\mathbb{R}^{d_x})\times \R^{d_\alpha}}\|\p_\alpha\p_\alpha f(x,\mu,\alpha)\|_{\mathcal{L}(\R^{d_\alpha}\times \R^{d_\alpha};\R^{d_x})}\cdot (1+|x|+\|\mu\|_1) \leq \frac{1}{40 \max\{\Lambda_k,L^*_0\}}\lambda_g.
\end{align} 
Then, for each $(x,\mu,z)\in c_{k_0}$, the function $f(x,\mu,\alpha)\cdot z  + g(x,\mu,\alpha)$ has a unique minimizer $\wh{\alpha}(x,\mu,z)\in\R^{d_\alpha}$, which satisfies the first-order condition 
\begin{align}\label{first_order_condition}
\p_\alpha f(x,\mu,\wh{\alpha})\cdot z  + \p_\alpha g(x,\mu,\wh{\alpha})=0,
\end{align} and it is first-order differentiable in $x$, $z\in\R^{d_x}$ and $L$-differentiable in $\mu\in\mathcal{P}_2(\R^{d_x})$ with jointly Lipschitz continuous first-order derivatives such that
\small\begin{align}\label{eq_5_25_new}
\partial_x \wh{\alpha}(x,\mu,z)=&\ -\Big(\p_\alpha\p_\alpha f(x,\mu,\wh{\alpha})\cdot z+\p_\alpha\p_\alpha g(x,\mu,\wh{\alpha})\Big)^{-1} \left(\p_x\p_\alpha f(x,\mu,\wh{\alpha})\cdot z+\p_x\p_\alpha g(x,\mu,\wh{\alpha})\right),\\
\partial_z \wh{\alpha}(x,\mu,z)=&\ -\Big(\p_\alpha\p_\alpha f(x,\mu,\wh{\alpha})\cdot z+\p_\alpha\p_\alpha g(x,\mu,\wh{\alpha})\Big)^{-1} \p_\alpha f(x,\mu,\wh{\alpha}),\\\label{eq_5_27_new}
\p_\mu \wh{\alpha}(x,\mu,z)(\wt{x})=&\ -\Big(\p_\alpha\p_\alpha f(x,\mu,\wh{\alpha})\cdot z+\p_\alpha\p_\alpha g(x,\mu,\wh{\alpha})\Big)^{-1}\left(\p_\mu\p_\alpha f(x,\mu,\wh{\alpha})(\wt{x})\cdot z+\p_\mu\p_\alpha g(x,\mu,\wh{\alpha})(\wt{x})\right).
\end{align}\normalsize
In addition, we have the following estimates:\\ 
(i) for an arbitrary $\xi\in\R^{d_x}$,
\begin{align}\label{positive_h}
&\Big(\Lambda_g+\frac{1}{20}\lambda_g\Big)|\xi|^2\geq \xi^\top\Big(\p_\alpha\p_\alpha f(x,\mu,\wh{\alpha})\cdot z+\p_\alpha\p_\alpha g(x,\mu,\wh{\alpha})\Big)\xi\geq \frac{19}{20}\lambda_g |\xi|^2;
\end{align}
(ii) for all $(x,\mu,z)\in c_{k_0}$,
\begin{align}\label{p_xalpha}
&\Big\|\p_{x}\wh{\alpha}(x,\mu,z)\Big\|_{\mathcal{L}(\R^{d_x};\R^{d_\alpha})}\vee\Big\|\p_{\mu}\wh{\alpha}(x,\mu,z)(\wt{x})\Big\|_{\mathcal{L}(\R^{d_x};\R^{d_\alpha})}\leq\frac{20\big(\wb{l}_g+\frac{1}{2}k_0\cdot\wb{l}_f\big) }{19\lambda_g},\\\label{p_zalpha}
&\Big\|\p_{z}\wh{\alpha}(x,\mu,z)\Big\|_{\mathcal{L}(\R^{d_x};\R^{d_\alpha})}\leq \frac{20\Lambda_f}{19\lambda_g};
\end{align}\normalsize
(iii) for any $(x,\mu,z)\in c_{k_0}$, and any $(x',\mu',z')\in c_{k_0}$ such that $\theta (x,\mu,z)+(1-\theta)(x',\mu',z')\in c_{k_0}$ for all $\theta\in[0,1]$
\begin{align}\label{lipalpha_old}
&|\wh{\alpha}(x,\mu,z)-\wh{\alpha}(x',\mu',z')|\leq\  L_\alpha\big(|x-x'|+W_1(\mu,\mu')+|z-z'|\big),\\\label{linalpha}
&|\wh{\alpha}(x,\mu,z)|\leq\  L_\alpha\cdot(1+|x|+\|\mu\|_1+|z|),
\end{align}
where 
\begin{align}\label{L_alpha}
L_\alpha:=\max\Big\{\frac{20\Lambda_f}{19\lambda_g},\frac{20\big(\wb{l}_g+\frac{1}{2}k_0\cdot\wb{l}_f\big) }{19\lambda_g},\wh{\alpha}(0,\delta_0,0)\Big\};
\end{align}
(iv) for any function $\Gamma:(x,\mu)\in\R^{d_x}\times\mc{P}_2(\R^{d_x})\mapsto \Gamma(x,\mu)\in\R^{d_x}$ which is differentiable in $x$ and is $L$-differentiable in $\mu$ with jointly continuous derivatives such that $\big\|\p_{x}\Gamma(x,\mu)\big\|_{\mathcal{L}(\R^{d_x};\R^{d_x})}\vee\big\|\p_{\mu}\Gamma(x,\mu)(\wt{x})\big\|_{\mathcal{L}(\R^{d_x};\R^{d_x})}\leq L_\Gamma$ for some positive constant $L_\Gamma\leq \frac{1}{2}k_0$ (recall \eqref{c_k_0}), then, $|\Gamma(x,\mu)|\leq L_\Gamma(1+|x|+\|\mu\|_1)$ and
\begin{align}\label{lipalpha}
&|\wh \alpha(x,\mu,\Gamma(x,\mu))-\wh \alpha(x',\mu',\Gamma(x',\mu'))|\leq L_\alpha(1+L_\Gamma)(|x-x'|+W_1(\mu,\mu')).
\end{align}
\end{prpstn}

\begin{prpstn}\label{A6}
For the cone $c_{k_0}$ defined in \eqref{c_k_0} and any $\xi\in\R^{d_x}$, and suppose also that \eqref{p_aa_f} is valid, then we have the following inequality:
\begin{align}\label{c_a7}
&\sup_{(x,\mu,z)\in c_{k_0}}\xi^\top\Big(\sum_{i=1}^{d_\alpha}\p_z \wh{\alpha}_i(x,\mu,z)\otimes \p_{\alpha_i} f(x,\mu,\wh{\alpha})\Big)\xi\leq -\frac{\lambda_f^2}{\Lambda_g+\lambda_g/20}|\xi|^2,
\end{align}\normalsize
where $\wh{\alpha}(x,\mu,z)$ is the unique minimizer solving the first-order condition \eqref{first_order_condition} and we denote $a\otimes b:=(a_ib_j)_{n\times m}$, the tensor product of $a$ and $b$, the $n\times m$ matrix for arbitrary vectors $a\in\R^n$ and $b\in \R^m$.
\end{prpstn}


\subsection{Discussion on Our Assumptions}\label{subsec:dis}
In this subsection, we aim to provide motivation behind and explanation of Assumptions ${\bf (a1)}$-${\bf (a3)}$ and the property \eqref{c_a7} by comparing them with the coercivity and monotonicity conditions used in Cardaliaguet-Delarue-Lasry-Lions' monograph \cite{cardaliaguet2019master}. We shall also refer to Remark \ref{remark_h2} to compare our monotonicity condition with the displacement monotonicity for non-separable Hamiltonians introduced in Gangbo, M\'esz\'aros, Mou, and Zhang \cite{gangbo2022mean}.

To simplify calculus and computations, we consider a special case where the drift function is simply $f(x,\mu,\alpha)=\alpha$ and the running cost function is $g(x,\mu,\alpha)=\lambda\alpha^2+g_1(x,\mu)$, with a positive constant $\lambda$. Note that a more complicated example is presented in Section \ref{sec:nonLQ}, which is not in this special form. Under this assumption, we find that \eqref{positive_f}, \eqref{bdd_d1_f}, \eqref{bdd_d2_f}, and \eqref{bdd_d2_g_2} are automatically satisfied with $\lambda_f=1$, $\Lambda_f=1$, and $\wb{l}_f=\wb{l}_g=0$. Additionally, we can explicitly solve for the optimal control $\wh{\alpha}$ using the first-order condition \eqref{first_order_condition} as $\wh{\alpha}(x,\mu,z)=-\frac{1}{2\lambda}z$. Therefore, the left-hand side of \eqref{c_a7} equals to $-\frac{1}{2\lambda}|\xi|^2$. It is worth noting that the term $\Big(\sum_{i=1}^{d_\alpha}\p_z \wh{\alpha}_i(x,\mu,z)\otimes \p_{\alpha_i} f(x,\mu,\wh{\alpha})\Big)$ in \eqref{c_a7} corresponds to the term $D^2_{zz}H(x,z)$ in (2.4) in \cite{cardaliaguet2019master}. Furthermore, there is no mean field term $\mu$ in $H(x,z)$ because the separable structure of the Hamiltonian was assumed in \cite{cardaliaguet2019master}, however we can also treat the non-separable case in the present work. Thus, \eqref{c_a7} corresponds to the coercivity condition (2.4) in Cardaliaguet-Delarue-Lasry-Lions' monograph \cite{cardaliaguet2019master}.

It is also worth noting that the running cost function $g(x,\mu,\alpha)$ corresponds to the function $F(x,\mu)$ in (2.5) of \cite{cardaliaguet2019master}. However, since the separable structure of the running cost function $g(x,\mu,\alpha)=\lambda\alpha^2+g_1(x,\mu)$ was assumed in \cite{cardaliaguet2019master}, there is no control term $\alpha$ in $F(x,\mu)$. Therefore, the estimate \eqref{positive_g_mu} assumed in Assumption ${\bf (a2)}$ generalizes the Lasry-Lions monotonicity assumption (see (2.5) in \cite{cardaliaguet2019master}) in the sense that it allows $\int_{\R^{d_x}} \Big(F(x,\mu')-F(x,\mu)\Big)d\big(\mu'-\mu\big)(x)$ to be negative, rather than non-negative in \cite{cardaliaguet2019master}.
Moreover, for any $m\in\mathcal{P}_2(\R^{d_x})$ and any $X,\,\wt{X}\in L^{2,d_x}_{m}$, by taking $\mu'=(X+\theta\wt{X})\ot m$ with $\theta>0$, $\mu=X\ot m$ and letting $\theta\to 0^+$ after dividing \eqref{positive_g_mu} by $\theta^2$, the estimate \eqref{positive_g_mu} implies that
\begin{align}\label{positive_g_mu_1}
&\int_{\R^{d_x}}\int_{\R^{d_x}}(\wt{X}(\wt{x})\cdot \p_{x})(\wt{X}(\wh{x})\cdot \p_\mu)g\big(x,\mu,\alpha\big)\bigg|_{x=X(\wt{x}),\mu=X\ot m}\big(X(\wh{x})\big)dm(\wh{x})dm(\wt{x})
\geq -l_g\left\|\wt{X}\right\|_{L^{2,d_x}_m}^2;
\end{align}
indeed in the mathematical analysis in this work, the estimate \eqref{positive_g_mu} assumed in Assumption ${\bf (a2)}$ can be replaced by \eqref{positive_g_mu_1} since we only use \eqref{positive_g_mu_1} in our proof of \eqref{displacement} in Theorem \ref{Crucial_Estimate} but not \eqref{positive_g_mu}. 
Also note that the estimate \eqref{positive_g_x} assumed in Assumption ${\bf (a2)}$ means that $g(x,\mu,\alpha)$ is strictly convex in $x$. 
Moreover, \eqref{positive_g_x} and \eqref{positive_g_mu} assumed in Assumption ${\bf (a2)}$ can be also replaced by 
\small\begin{align}\no
&\int_{\R^{d_x}}\int_{\R^{d_x}}(\wt{X}(\wt{x})\cdot \p_{x})(\wt{X}(\wt{x})\cdot \p_x{})g\big(x,\mu,\alpha\big)\bigg|_{x=X(\wt{x}),\mu=X\ot m}dm(\wt{x})\\\label{dispalcement_lambda}
&+\int_{\R^{d_x}}\int_{\R^{d_x}}(\wt{X}(\wt{x})\cdot \p_{x})(\wt{X}(\wh{x})\cdot \p_\mu)g\big(x,\mu,\alpha\big)\bigg|_{x=X(\wt{x}),\mu=X\ot m}\big(X(\wh{x})\big)dm(\wh{x})dm(\wt{x})\geq \lambda\left\|\wt{X}\right\|_{L^{2,d_x}_m}^2,
\end{align}\normalsize
with $\lambda=\lambda_g>0$. The reasons of this replacement are as follows: i) \eqref{dispalcement_lambda} can imply \eqref{positive_g_x} by  
taking $\wt{X}(\wt{x})=\eps^{-d_x}\wh{X}(\eps^{-1} \wt{x})$ and then passing to the limit as $\eps\to 0^+$ after multiplying \eqref{dispalcement_lambda} by $\eps^{d_x}$ (also see Proposition B.6 in \cite{gangbo2020global}); 
ii) while showing \eqref{displacement} in the proof of Theorem \ref{Crucial_Estimate}, we can directly use \eqref{dispalcement_lambda} instead of \eqref{positive_g_mu_1}. On the other hand, \eqref{positive_g_x} and \eqref{positive_g_mu} can imply \eqref{dispalcement_lambda} with $\lambda=\lambda_g-l_g$. 

Similarly, the terminal cost function $k(x,\mu)$ corresponds to the function $G(x,\mu)$ in (2.5) of \cite{cardaliaguet2019master}. Therefore, the estimate \eqref{positive_k_mu} assumed in Assumption ${\bf (a3)}$ generalizes the Lasry-Lions monotonicity assumption (see (2.5) in \cite{cardaliaguet2019master}) in the sense that it allows $\int_{\R^{d_x}} \Big(G(x,\mu')-G(x,\mu)\Big)d\big(\mu'-\mu\big)(x)$ to be negative, rather than non-negative.
Again, for any $m\in\mathcal{P}_2(\R^{d_x})$ and any $X,\,\wt{X}\in L^{2,d_x}_{m}$, by taking $\mu'=(X+\theta\wt{X})\ot m$ with $\theta>0$, $\mu=X\ot m$ and letting $\theta\to 0^+$ after dividing \eqref{positive_k_mu} by $\theta^2$, the estimate \eqref{positive_k_mu} implies that
\begin{align}\label{positive_k_mu_1}
&\int_{\R^{d_x}}\int_{\R^{d_x}}(\wt{X}(\wt{x})\cdot \p_{x})(\wt{X}(\wh{x})\cdot \p_\mu)k\big(x,\mu\big)\bigg|_{x=X(\wt{x}),\mu=X\ot m}\big(X(\wh{x})\big)dm(\wh{x})dm(\wt{x})
\geq -l_k\left\|\wt{X}\right\|_{L^{2,d_x}_m}^2.
\end{align}
Again, the estimate \eqref{positive_k_mu} assumed in Assumption ${\bf (a3)}$ can be replaced by \eqref{positive_k_mu_1}
since we only use \eqref{positive_k_mu_1} in our proof of \eqref{displacement} in Theorem \ref{Crucial_Estimate} but not \eqref{positive_k_mu}. 
Note again that the estimate \eqref{positive_k} assumed in Assumption ${\bf (a3)}$ means that $k(x,\mu)$ is strictly convex in $x$. 
Again, \eqref{positive_k} and \eqref{positive_k_mu} assumed in Assumption ${\bf (a3)}$ can also be replaced by 
\small\begin{align}\no
&\int_{\R^{d_x}}\int_{\R^{d_x}}(\wt{X}(\wt{x})\cdot \p_{x})(\wt{X}(\wt{x})\cdot \p_x{})k\big(x,\mu\big)\bigg|_{x=X(\wt{x}),\mu=X\ot m}dm(\wt{x})\\\label{dispalcement_lambda_k}
&+\int_{\R^{d_x}}\int_{\R^{d_x}}(\wt{X}(\wt{x})\cdot \p_{x})(\wt{X}(\wh{x})\cdot \p_\mu)k\big(x,\mu\big)\bigg|_{x=X(\wt{x}),\mu=X\ot m}\big(X(\wh{x})\big)dm(\wh{x})dm(\wt{x})\geq \lambda\left\|\wt{X}\right\|_{L^{2,d_x}_m}^2
\end{align}\normalsize
with $\lambda=\lambda_k>0$.

In summary, \eqref{c_a7} has a close relationship with the coercivity condition (2.4) of the Hamiltonian in \cite{cardaliaguet2019master}, while \eqref{positive_g_mu} and \eqref{positive_k_mu} have a close relationship with the monotonicity condition (2.5) in the same monograph. In other words, we are generalizing the coercivity condition (2.4) and the monotonicity condition (2.5) in \cite{cardaliaguet2019master} to the case with a non-separable Hamiltonian and a general drift function. For the comparison with displacement monotonicity, one can refer to Remark \ref{remark_h2}.

\section{Hamilton-Jacobi and Fokker-Planck System, Master equation, Forward-backward System and Their Relationships}\label{sec:master}
It follows from the principle of dynamic programming that the value function $v(t,x)$ defined in \eqref{valfun} is expected to satisfy the Hamilton-Jacobi (HJ) equation:
\begin{align}\label{HJ_eq}
\begin{cases}
\p_t v(t,x) + h\big(x,\wh m_t,\p_x v(t,x)\big)= 0,\\
v(T,x)= k(x,\wh m_T),
\end{cases}
\end{align} 
where the Hamiltonian $h:\R^{d_x}\times \mathcal{P}_2(\R^{d_x})\times \R^{d_x} \to\R$ is defined as
\begin{equation}\label{hamilton1}
h(x,\mu,z):= \min_{\alpha\in \R^{d_\alpha}} \left(f(x,\mu,\alpha)\cdot z  + g(x,\mu,\alpha)\right),
\end{equation}
and, by \eqref{optimal_probability}, the density of the probability measure $\wh m_t$ (if exists), denoted also by $\wh m(t,x)$, satisfies the following Fokker-Planck (FP) equation:
\begin{align}\label{FP_eq}
\begin{cases}
\p_t \wh m(t,x) + \div\Big(\p_z h\big(x,\wh m(t,x),\p_x v(t,x)\big)\wh m(t,x)\Big)= 0,\\
\wh m(0,x)= m_0(x).
\end{cases}
\end{align}
Suppose that for any fixed $x\in\R^{d_x}$, $\mu\in \mathcal{P}_2(\R^{d_x})$ and $z\in\R^{d_x}$ belong to our interested domain\footnote{More precisely, we shall show that only those $(x,\mu,z)$'s belonging to the cone set $c_{k_0}$ defined in \eqref{k_0} are related to the underlying mean field games, so Proposition \ref{A5} provides a solid and rigorous foundation for the discussion below.}, the function $f(x,\mu,\alpha)\cdot z  + g(x,\mu,\alpha)$ has a unique minimizer $\wh{\alpha}(x,\mu,z)\in\R^{d_\alpha}$ and is differentiable in $\alpha\in\R^{d_\alpha}$. Then the unique minimizer $\wh{\alpha}(x,\mu,z)$ must satisfy the usual first-order condition:
\begin{align}\label{foc}
\partial_\alpha f(x,\mu,\widehat{\alpha})\cdot z  + \partial_\alpha g(x,\mu,\widehat{\alpha})=0.
\end{align}
Substituting the optimal control $\wh{\alpha}$ in \eqref{HJ_eq}-\eqref{FP_eq} yields the following HJ-FP system:
\small
\begin{align}\label{HJFP_eq_nomin}
\begin{cases}
\p_t v(t,x) + f\Big(x,\wh m(t,x),\wh{\alpha}\big(x,\wh m(t,x),\p_x v(t,x)\big)\Big)\cdot \p_x v(t,x)  + g\Big(x,\wh m(t,x),\wh{\alpha}\big(x,\wh m(t,x),\p_x v(t,x)\big)\Big)= 0,\\
v(T,x)= k(x,\wh m_T);\\
\p_t \wh m(t,x) + \div\bigg(f\Big(x,\wh m(t,x),\wh{\alpha}\big(x,\wh m(t,x),\p_x v(t,x)\big)\Big)\wh m(t,x)\bigg)= 0,\\
\wh m(0,x)= m_0(x).
\end{cases}
\end{align}\normalsize
Define, for any $(t,x,m)\in[0,T]\times \R^{d_x}\times \mathcal{P}_2(\R^{d_x})$, 
\begin{align}\label{Valfun}
V(t,x,m):=j(t,x,\wh m,\alpha^*(t,x,\wh m)),
\end{align}
where the optimal control $\alpha^*$ is defined in \eqref{optimal_alpha}, the Nash equilibrium $\wh{m}:=(\wh m_s)_{s\in[t,T]}$ is given by 
\begin{align}\label{optimal_probability_m}
\wh{m}_s=X^{t,\wh{m},*}_s\ot m,
\end{align}
and the solution map $X^{t,\wh m,*}_s$ is defined by solving \eqref{SDE_star}. Then, applying the principle of dynamic programming, one can show that $V(t,x,m)$ satisfies the following master equation, also see \cite{bensoussan2017interpretation} and \cite{bensoussan2015master}:
\begin{align}\label{Master_eq}
\begin{cases}
&\p_t V(t,x,m)+f\Big(x,m,\wh{\alpha}\big(x,m,\p_x V(t,x,m)\big)\Big)\cdot \p_x V(t,x,m)  + g\Big(x,m,\wh{\alpha}\big(x,m,\p_x V(t,x,m)\big)\Big)\\
&\ \ \ \ \ \ \ \ \ \ \ \ \ \ \ \ \ \ \ \ \ \ \ \ \ \ \ \ \ \ \ \ \ \ +\int_{\R^{d_x}}\p_m V(t,x,m)(y)\cdot f\Big(y,m,\wh{\alpha}\big(y,m,\p_x V(t,y,m)\big)\Big)dm(y)=0,\\
&V(T,x,m)=k(x,m),
\end{cases}
\end{align}
where $\wh{\alpha}(x,\mu,z)$ satisfies the first-order condition \eqref{foc}. Let $(v,\wh m):=(v(t,x),\wh m(t,x))$ be a smooth enough solution of the HJ-FP system \eqref{HJFP_eq_nomin} with the given initial condition $\wh m(t_0,x)=\wh m_{t_0}(x)$ at time $t=t_0$, then we can directly check that $V(t_0,x,\wh m_{t_0}):=v(t_0,x)$ is a solution of the master equation \eqref{Master_eq}. In other words, we can write the solution $V(t,x,m)$ of the master equation \eqref{Master_eq} in terms of the solution $(v,\wh m)$ of the HJ-FP system \eqref{HJFP_eq_nomin} by the characteristics method; also see section 1.2.4 in the monograph of
Cardaliaguet, Delarue, Lasry and Lions \cite{cardaliaguet2019master}. 

Differentiating master equation \eqref{Master_eq} with respect to $x$ and using the first-order condition \eqref{foc}, we obtain
\begin{align}\label{p_x_Master_eq}
\begin{cases}
&\p_t \p_x V(t,x,m)+f\Big(x,m,\wh{\alpha}\big(x,m,\p_x V(t,x,m)\big)\Big)\cdot \p_x \p_x V(t,x,m)\\
&+\p_x f\Big(x,m,\wh{\alpha}\big(x,m,\p_x V(t,x,m)\big)\Big)\cdot \p_x V(t,x,m)+ \p_x g\Big(x,m,\wh{\alpha}\big(x,m,\p_x V(t,x,m)\big)\Big)\\
&+\int_{\R^{d_x}}\p_m \p_x V(t,x,m)(y)\cdot f\Big(y,m,\wh{\alpha}\big(y,m,\p_x V(t,y,m)\big)\Big)dm(y)=0,\\
&\p_x V(T,x,m)=\p_x k(x,m).
\end{cases}
\end{align}\normalsize
We can also write the solution $\p_x V(t,x,m)$ of \eqref{p_x_Master_eq} in terms of the solution $(X,Z)$ of a system of forward-backward ordinary differential equations (FBODE) defined as follows by using the characteristics method and the push-forward operator $\ot$ defined in Definition \ref{defot}:\begin{equation}\label{FBODE}
\left\{\begin{aligned}
\dfrac{d}{ds}X^{t,m}_s(x) &=f\Big(X^{t,m}_s(x),X^{t,m}_s\ot m,\wh{\alpha}\big(X^{t,m}_s(x),X^{t,m}_s\ot m,Z^{t,m}_s(x)\big)\Big),\ s\in[t,T],\\
X^{t,m}_t(x)&=x,\\
\dfrac{d}{ds}Z^{t,m}_s(x)&=-\p_x f\Big(X^{t,m}_s(x),X^{t,m}_s\ot m,\wh{\alpha}\big(X^{t,m}_s(x),X^{t,m}_s\ot m,Z^{t,m}_s(x)\big)\Big)\cdot Z^{t,m}_s(x)\\
&\ \ - \p_x g\Big(X^{t,m}_s(x),X^{t,m}_s\ot m,\wh{\alpha}\big(X^{t,m}_s(x),X^{t,m}_s\ot m,Z^{t,m}_s(x)\big)\Big),\ s\in[t,T],\\
Z^{t,m}_T(x)&=\p_x k\big(X^{t,m}_T(x),X^{t,m}_T\ot m\big);
\end{aligned}\right.
\end{equation}
indeed, we can directly check that $\p_x V(t,x,m)=Z^{t,m}_t(x)$ and $\p_x V(s,X^{t,m}_s(x),X^{t,m}_s\ot m)=Z^{t,m}_s(x)$. Once we solve the FBODE \eqref{FBODE}, then we can obtain the solution of master equation \eqref{Master_eq} by a direct integration:
\begin{align}\label{Master_eq_sol}
V(t,x,m):=k\big(X^{t,m}_T(x),X^{t,m}_T\ot m\big)+\int_t^T g\Big(X^{t,m}_s(x),X^{t,m}_s\ot m,\wh{\alpha}\big(X^{t,m}_s(x),X^{t,m}_s\ot m,Z^{t,m}_s(x)\big)\Big)ds.
\end{align}
In addition, we can obtain the optimal control process $$\big(\alpha^*_t(0,x,X^{0,m_0}_\cdot\ot m_0)\big)_{t\in[0,T]}=\Big(\wh{\alpha}\big(X^{0,m_0}_t(x),X^{0,m_0}_t\ot m_0,\p_x V(t,X^{0,m_0}_t(x),X^{0,m_0}_t\ot m_0)\big)\Big)_{t\in[0,T]}.$$ To avoid too much cumbersome of notation, we use $\alpha$ in place of $\wh{\alpha}$ to denote the optimal feedback control, which satisfies the first-order condition \eqref{foc} in the rest of this article.

\section{Local Existence of Forward-backward System \eqref{fbodesystem}}\label{sec:local}
For any fixed $0\leq t\leq \wt{T}\leq T$ and $m\in \mathcal{P}_2(\R^{d_x})$, consider the following forward-backward ordinary differential equation (FBODE) system, now with possibly a generic terminal cost functional $p$:  for $s\in[t,\wt T]$,
\footnotesize\begin{align}\label{fbodesystem}
\begin{cases}       
\dfrac{d}{ds}X^{t,m}_s(x) = f\left(X^{t,m}_s(x),X^{t,m}_s\ot m,\alpha(X^{t,m}_s(x),X^{t,m}_s\ot m,Z^{t,m}_s(x))\right),\ 
X^{t,m}_t(x)= x,\\
\,\dfrac{d}{ds}Z^{t,m}_s(x) =-\p_x f(X^{t,m}_s(x),X^{t,m}_s\ot m,\alpha(X^{t,m}_s(x),X^{t,m}_s\ot m,Z^{t,m}_s(x)))\cdot Z^{t,m}_s(x)\\
\ \ \ \ \ \ \ \ \ \ \ \ \ \ \ \ \ \ -\p_x g(X^{t,m}_s(x),X^{t,m}_s\ot m,\alpha(X^{t,m}_s(x),X^{t,m}_s\ot m,Z^{t,m}_s(x))),
\ \ \ \ Z^{t,m}_{\wt T}(x)=p(X^{t,m}_{\wt T}(x),X^{t,m}_{\wt T}\ot m),
\end{cases}
\end{align}\normalsize
where the optimal control $\alpha(x,\mu,z)$ solves the first-order condition \eqref{foc},
and the terminal function $p$, which can be arbitrary now, satisfies the following assumptions:\\ 
${\bf(P)}$ The terminal function $p:(x,\mu)\in \R^{d_x}\times \mathcal{P}_2(\R^{d_x})\mapsto p(x,\mu)\in \R^{d_x}$ is first-order differentiable in $x\in\R^{d_x}$ and $L$-differentiable in $\mu\in\mathcal{P}_2(\R^{d_x})$ with their first-order derivatives being Lipschitz continuous in their corresponding arguments and satisfying, for some finite positive constants $L_p$ and $\wb{L}_p$,
\begin{align}\label{p1}
&\sup_{x\in\R^{d_x},\,\mu\in\mathcal{P}_2(\R^{d_x}),\,\wt{x}\in\R^{d_x}} \left\|\p_x p(x,\mu)\right\|_{\mathcal{L}(\R^{d_x};\R^{d_x})}\vee\left\|\p_\mu p(x,\mu)(\wt{x})\right\|_{\mathcal{L}(\R^{d_x};\R^{d_x})}
\leq L_p,\\\label{p2}
&|p(x,\mu)|\leq \wb{L}_p\cdot(1+|x|+\|\mu\|_1).
\end{align}
Assumption ${\bf(P)}$ echoes Assumption ${\bf(a3)}(ii)$ in Section \ref{assumptions} about the nature of the terminal data, and this structure will be preserved, upon enlarging the bound $L_p$ gradually, while solving the FBODE over consecutive sub-intervals in a backward manner. In addition, \eqref{p1} implies that, for any $x,\,x'\in\R^{d_x}$ and $\mu,\,\mu'\in\mathcal{P}_2(\R^{d_x})$,
\begin{align}\label{lipp}
\left|p(x,\mu)-p(x',\mu')\right|\leq L_p \left(|x-x'|+W_1(\mu,\mu')\right).
\end{align}
Then, clearly, we see that \eqref{p2} is valid 
for any finite constant $\wb{L}_p\geq \max\{|p(0,\delta_0)|,L_p\}$. For the sake of convenience, we also define 
\begin{align}\label{L_B}
L_B:=L_f(1+L_\alpha+2L_pL_\alpha)\text{ and }\wb{L}_B:=L_f(1+L_\alpha+2\wb{L}_pL_\alpha).
\end{align}
For local (in time) solutions of \eqref{FBODE}, for the last time interval, $p(x,\mu)$ is taken to be $\p_x k(x,\mu)$ and thus, by \eqref{bdd_d2_k_1}, one can set $L_p=\Lambda_k$. In the proof for global (in time) existence theorem~\ref{GlobalSol}, we shall glue (smoothly) consecutively together local (in time) solutions of \eqref{fbodesystem} so as to obtain a global (in time) solution of \eqref{FBODE} over $[0,T]$ of arbitrary length $T$; for a variety of terminal data $p$ in different small time intervals $[t_i,t_{i+1}]$, $i=1,2,...,N$, by using Theorem \ref{Crucial_Estimate} to be proven later in Section \ref{sec:global}, one can set $\wb{L}_p=\max\left\{\Lambda_k,L^*_0\right\}=L^*_0$ uniformly for various terminal data $p$, provided that $L^*_0$ is to be defined in \eqref{L_star_0}. In these cases, by \eqref{cone_ZX}, one has, for all $(s,x,m)\in[t,{\wt{T}}] \times \R^{d_x}\times \mathcal{P}_2(\R^{d_x})$, $\left(X^{t,m}_s(x),X^{t,m}_s\ot m,Z^{t,m}_s(x)\right)\in  c_{k_0} $ where the cone $c_{k_0}$ was defined in \eqref{c_k_0} and the constant $k_0$ was defined in \eqref{k_0}, namely $k_0=4\wb{L}_p$. Thus, when \eqref{p_aa_f} is valid, by Proposition \ref{A5}, the optimal control $\alpha\left(X^{t,m}_s(x),X^{t,m}_s\ot m,Z^{t,m}_s(x)\right)$ is well-defined and satisfies the first-order condition \eqref{first_order_condition}.

Define
\begin{align}\label{Gamma_1}
\text{(a norm) }\vertiii{\gamma }_{1,[t,T]}:=&\  \sup_{s\in[t,T],\,x\in \R^{d_x},\,\mu\in\mathcal{P}_2(\R^{d_x})}\frac{|\gamma(s,x,\mu)|}{1+|x|+\|\mu\|_1},\\\label{Gamma_2}
\text{(a semi-norm) }\vertiii{\gamma }_{2,[t,T]}:=&\  \sup_{s\in[t,T],x\in\R^{d_x},\,\mu\in\mathcal{P}_2(\R^{d_x}),\,\wt{x}\in\R^{d_x}} \left\|\p_x \gamma(s,x,\mu)\right\|_{\mathcal{L}(\R^{d_x};\R^{d_x})}\vee\left\|\p_\mu \gamma(s,x,\mu)(\wt{x})\right\|_{\mathcal{L}(\R^{d_x};\R^{d_x})}.
\end{align}
When there is no cause of ambiguity, we denote $\vertiii{\gamma }_{i,[t,T]}$ by $\vertiii{\gamma }_{i}$, for $i=1,\,2$, for the notational simplicity. Before stating our main results in this section, let us first recall the following concept.

\begin{definition}\label{decouple}(Decoupling field) 
We call a function $\gamma:[t,\wt{T}]\times \R^{d_x}\times \mc{P}_2(\R^{d_x})\to \R^{d_x}$ a decoupling field for the FBODE system \eqref{fbodesystem} on $[t,\wt{T}]$ if $\gamma$ solves the following system for all $m\in\mc{P}_2(\R^{d_x})$, $x\in\R^{d_x}$ and $s\in[t,\wt{T}]$:
\small\begin{align}\no
\gamma (s,x,m)
= &\ p(X^{s,m}_{\wt{T}}(x),X^{s,m}_{\wt{T}}\ot m)
+ \int_s^{\wt{T}} \Bigg(\p_x f(X^{s,m }_\tau(x),X^{s,m }_\tau\ot m,\alpha(X^{s,m }_\tau(x),X^{s,m }_\tau\ot m,Z^{s,m }_\tau(x)))\cdot Z^{s,m }_\tau(x)\\\label{gammaeq} 
&\ \ \ \ \ \ \ \ \ \ \ \ \ \ \ \ \ \ \ \ \ \ \ \ \ \ \ \ \ \ \ \ \ \ \ \ \ +\p_x g(X^{s,m }_\tau(x),X^{s,m }_\tau\ot m,\alpha(X^{s,m }_\tau(x),X^{s,m }_\tau\ot m,Z^{s,m }_\tau(x)))\Bigg)d\tau,
\end{align}\normalsize
where $Z^{s,m }_\tau(x):=\gamma\big(\tau,X^{s,m }_\tau(x),X^{s,m }_\tau\ot m\big)$ and
\small\begin{align}
\label{xeq}       
X^{s,m}_\tau(x) = x+\int_s^\tau f\Big(X^{s,m }_{\wt{\tau}}(x),X^{s,m }_{\wt{\tau}}\ot m,\alpha\big(X^{s,m }_{\wt{\tau}}(x),X^{s,m }_{\wt{\tau}}\ot m,\gamma(\tau,X^{s,m }_{\wt{\tau}}(x),X^{s,m }_{\wt{\tau}}\ot m)\big)\Big)d\wt{\tau},
\end{align}\normalsize
and $\alpha(x,\mu,z)$ is the optimal control solved via the first-order condition \eqref{first_order_condition}.
\end{definition}
\begin{remark}
In other words, the decoupling field $\gamma:=\gamma(s,x,m)$ has the following property: using the given $\gamma$, one can first solve for $X^{s,m}_\tau(x)$ in \eqref{xeq}, and then define $Z^{s,m }_\tau(x):=\gamma\big(\tau,X^{s,m }_\tau(x),X^{s,m }_\tau\ot m\big)$. Then, it can be checked directly that the couple $\big(X^{t,m}_s(x),Z^{t,m}_s(x)\big)_{x\in\R^{d_x},s\in[t,\wt{T}]}$ is a solution to the FBODE system \eqref{fbodesystem}.
\end{remark}


Our main results in this section are as follows.
\begin{theorem}\label{Thm6_1} (Local-in-time existence)
Assume that the drift function $f$ and the running cost $g$ satisfy Assumptions ${\bf(a1)}$ and ${\bf(a2)}$, respectively, the terminal function $p$ satisfies Assumption ${\bf(P)}$ with $L_p\leq \wb{L}_p:=\max\left\{\Lambda_k,L^*_0\right\}$ and the inequality relation\footnote{It echoes the condition \eqref{p_aa_f}.}
\begin{align}\label{p_aa_f_new}
\sup_{(x,\mu,\alpha)\in \R^{d_x}\times \mathcal{P}_{2}(\mathbb{R}^{d_x})\times \R^{d_\alpha}}\|\p_\alpha\p_\alpha f(x,\mu,\alpha)\|_{\mathcal{L}(\R^{d_\alpha}\times \R^{d_\alpha};\R^{d_x})}\cdot (1+|x|+\|\mu\|_1) \leq \frac{1}{40 \wb{L}_p}\lambda_g,
\end{align}  
among $f$, $g$ and $p$ is valid, then there exists a constant $\eps_1=\eps_1(\wb{L}_p; L_f,\Lambda_f,\wb{l}_f,L_g,\Lambda_g,\wb{l}_g,L_\alpha)>0$, such that, for any $0\leq t\leq \wt{T}\leq T$ with $\wt{T}-t\leq \eps_1$, there exists a unique continuous decoupling field $\gamma:(s,x,m)\in [t,\wt{T}]\times \R^{d_x}\times \mathcal{P}_2(\R^{d_x})\mapsto \gamma(s,x,m)\in \R^{d_x}$ for the FBODE system \eqref{fbodesystem} on $[t,\wt{T}]$, which also satisfies $\vertiii{\gamma}_{1,[t,\wt{T}]}\leq 2\wb{L}_p$. Morover, for any $m\in \mathcal{P}_2(\R^{d_x})$, define $X^{t,m}_s(x)$ by solving \eqref{xeq} and $Z^{t,m}_s(x):=\gamma(s,X^{t,m}_s(x),X^{t,m}_s\ot m)$, then the pair $(X_s^{t,m}(x),Z_s^{t,m}(x))_{x\in\R^{d_x},s\in[t,\wt{T}]}$ is a solution of FBODE system \eqref{fbodesystem}, each of them is continuous in $(t,m,x)\in[t,\wt{T}]\times\mc{P}_2(\R^{d_x})\times\R^{d_x}$ and is continuously differentiable in $s\in[t,\wt{T}]$; and it also satisfies the following estimate:
\begin{align}\label{cone_ZX}
\left|Z^{t,m}_s(x)\right|\leq 2\wb{L}_p\cdot\left(1+\left|X^{t,m}_s(x)\right|+\left\|X^{t,m}_s\right\|_{L^{1,d_x}_m}\right),
\end{align}
which means that $(X_s^{t,m}(x),Z_s^{t,m}(x))$ satisfies the cone condition $(X_s^{t,m}(x),X_s^{t,m}\ot m, Z_s^{t,m}(x))\in c_{k_0}$ with the cone set $c_{k_0}$ defined in \eqref{c_k_0}. In addition, $X^{t,m}_s(x)$ has the following flow property of $X^{t,m }_s(x)$, for any $0\leq t\leq s\leq \tau\leq T$,
\begin{align}\label{coro_6_4}
X^{t,m }_\tau(x)=X^{s,X^{t,m }_s\ot m }_\tau\big(X^{t,m }_s(x)\big).
\end{align}
\end{theorem}
The proof of Theorem \ref{Thm6_1} is similar to, but simpler\footnote{The simplification is due to the absence of the terms $\p_\mu f(x,\mu,\alpha)$ and $\p_\mu g(x,\mu,\alpha)$ on the right-hand side of the FBODE system \eqref{fbodesystem}, in comparison with (7.1) of \cite{bensoussan2023theory}.} than, the proof of Theorem 7.3 presented in Section 7.2 of our previous paper \cite{bensoussan2023theory}.  Therefore, we leave the technical details to the interested readers.

\begin{theorem}\label{Thm6_2} (Regularity of the local solution)
Assume that the drift function $f$ and the running cost $g$ satisfy Assumptions ${\bf(a1)}$ and ${\bf(a2)}$ respectively, the terminal function $p$ satisfies Assumption ${\bf(P)}$ with $L_p\leq \wb{L}_p:=\max\left\{\Lambda_k,L^*_0\right\}$ and the inequality relation \eqref{p_aa_f_new} among $f$, $g$ and $p$ is valid, then there exists a constant $\eps_2=\eps_2(\wb{L}_p; L_f,\Lambda_f,\wb{l}_f,L_g,\Lambda_g,\wb{l}_g,L_\alpha)>0$, such that, for any $0\leq t\leq \wt{T}\leq T$ with $\wt{T}-t\leq \eps_2$, the continuous decoupling field $\gamma(s,x,\mu)$ for the FBODE system \eqref{fbodesystem} on $[t,\wt{T}]$ constructed in Theorem \ref{Thm6_1} has the following regularities:\\
(i) $\gamma(s,x,\mu)$ is differentiable in $s\in[\wt{T}-\eps_2,\wt{T}]$ and $x\in\R^{d_x}$, and $L$-differentiable in $\mu\in\mc{P}_2(\R^{d_x})$, and its derivatives $\p_s\gamma(s,x,\mu)$, $\p_x\gamma(s,x,\mu)$ and $\p_\mu\gamma(s,x,\mu)(\wt{x})$ are jointly continuous in their corresponding arguments $(s,x,\mu)\in [\wt{T}-\eps_2,\wt{T}]\times \R^{d_x} \times \mathcal{P}_2(\R^{d_x})$ and $(s,x,\mu, \wt{x})\in [\wt{T}-\eps_2,\wt{T}]\times \R^{d_x} \times \mathcal{P}_2(\R^{d_x})\times \R^{d_x}$ respectively; they also satisfy the following estimates:
\begin{align}\label{p_x_gamma_bdd}
\sup_{s\in[\wt{T}-\eps_2,\wt{T}],\,x\in\R^{d_x},\,\mu\in\mathcal{P}_2(\R^{d_x})} \left\|\p_x\gamma(s,x,\mu)\right\|_{\mathcal{L}(\R^{d_x};\R^{d_x})}\leq 2L_p,\\
\label{p_m_gamma_bdd}
\sup_{s\in[\wt{T}-\eps_2,\wt{T}],\,x\in\R^{d_x},\,\mu\in\mathcal{P}_2(\R^{d_x}),\,\wt{x}\in\R^{d_x}} \left\|\p_\mu \gamma(s,x,\mu)(\wt{x})\right\|_{\mathcal{L}(\R^{d_x};\R^{d_x})}\leq 2L_p.
\end{align}
Moreover, the solution pair $\big(X^{t,m}_s(x),Z^{t,m}_s(x)=\gamma(s,X^{t,m}_s(x),X^{t,m}_s\ot m)\big)_{x\in\R^{d_x},s\in[\wt{T}-\eps_2,\wt{T}]}$ of FBODE system \eqref{fbodesystem} constructed in Theorem \ref{Thm6_1} through $\gamma$ has the following regularities:\\
(ii) $X^{t,m}_s(x)$ is differentiable in $x\in\R^{d_x}$ and $L$-differentiable in $m\in\mc{P}_2(\R^{d_x})$, and its derivatives $\p_x\big(X^{t,m}_s(x)\big)$ and $\p_m\big(X^{t,m}_s(x)\big)(y)$ are jointly continuous in their corresponding arguments $(t,m,x)\in [\wt{T}-\eps_2,\wt{T}]\times \mathcal{P}_2(\R^{d_x})\times \R^{d_x}$ and $(t,m,x,y)\in [\wt{T}-\eps_2,\wt{T}]\times \mathcal{P}_2(\R^{d_x})\times \R^{d_x}\times \R^{d_x}$ respectively; $X^{t,m}_s(x)$ and these derivatives are also continuously differentiable in $s\in[t,\wt{T}]$ and satisfy the following estimates:
\begin{align}\label{p_x_X_bdd}
&\left\|\p_x\big(X^{t,m}_s(x)\big)\right\|_{\mc{L}(\R^{d_x};\R^{d_x})}\leq \exp\big(L_B(s-t)\big),\\\label{p_m_X_bdd}
&\left\|\p_m\big(X^{t,m}_s(x)\big)(y)\right\|_{\mc{L}(\R^{d_x};\R^{d_x})}\leq L_B(s-t)\Big(L_B(s-t)\exp\Big(2L_B(s-t)\Big)+1\Big)\exp\Big(2L_B(s-t)\Big);
\end{align}
(iii) $X^{t,m}_s(x)$ is differentiable in $t\in[\wt{T}-\eps_2,\wt{T}]$; its derivative $\p_t\big(X^{t,m}_s(x)\big)$ is jointly continuous in its arguments $(t,m,x)\in [\wt{T}-\eps_2,\wt{T}]\times \mathcal{P}_2(\R^{d_x})\times \R^{d_x}$; this derivative is also continuously differentiable in $s\in[t,\wt{T}]$ and satisfies the following estimates:
\begin{align}\label{p_t_X_bdd}
\left|\p_t\big(X^{t,m}_s(x)\big)\right|\leq \wb{L}_B\big(1+|x|+\|m\|_1\big)\Big(L_B(s-t)\exp\Big(2L_B(s-t)\Big)+1\Big)\exp\Big(L_B(s-t)\Big);
\end{align}
here $L_B$ and $\wb{L}_B$ are constants defined in \eqref{L_B}.\\
In particular, $Z^{t,m}_s(x)=\gamma(s,X^{t,m}_s(x),X^{t,m}_s\ot m)$ is differentiable in $t\in[\wt{T}-\eps_2,\wt{T}]$ and $x\in\R^{d_x}$, and $L$-differentiable in $m\in\mc{P}_2(\R^{d_x})$; and its derivatives $\p_t\big(Z^{t,m}_s(x)\big)$, $\p_x\big(Z^{t,m}_s(x)\big)$ and $\p_m\big(Z^{t,m}_s(x)\big)(y)$ are jointly continuous in their corresponding arguments $(t,m,x)\in [\wt{T}-\eps_2,\wt{T}]\times \mathcal{P}_2(\R^{d_x})\times \R^{d_x}$ and $(t,m,x,y)\in [\wt{T}-\eps_2,\wt{T}]\times \mathcal{P}_2(\R^{d_x})\times \R^{d_x}\times \R^{d_x}$ respectively; $Z^{t,m}_s(x)$ and these derivatives also continuously differentiable in $s\in[t,\wt{T}]$ and satisfy the following estimates:
\footnotesize\begin{align}\label{p_x_Z_bdd_n}
&\left\|\p_x\big(Z^{t,m}_s(x)\big)\right\|_{\mc{L}(\R^{d_x};\R^{d_x})}\leq 2L_p\exp\big(L_B(s-t)\big);\\\label{p_m_Z_bdd_n}
&\left\|\p_m\big(Z^{t,m}_s(x)\big)(y)\right\|_{\mc{L}(\R^{d_x};\R^{d_x})}\leq 4L_pL_B(s-t)\Big(L_B(s-t)\exp\Big(2L_B(s-t)\Big)+1\Big)\exp\Big(2L_B(s-t)\Big)+2L_p\exp\big(L_B(s-t)\big);\\\label{p_t_Z_bdd_n}
&\left|\p_t\big(Z^{t,m}_s(x)\big)\right|\leq 4L_p\wb{L}_B\big(1+|x|+\|m\|_1\big)\Big(L_B(s-t)\exp\Big(2L_B(s-t)\Big)+1\Big)\exp\Big(L_B(s-t)\Big).
\end{align}\normalsize
\end{theorem}
The proof of Theorem \ref{Thm6_2} is similar to, but simpler than, the proof of Theorem 7.4 presented in Section 7.3 of our previous paper \cite{bensoussan2023theory}. 
Therefore, we also leave the technical details of this proof to the interested readers.
\begin{remark}
The assumption on the validity of \eqref{p_aa_f_new} condition is used only for ensuring the unique solvability of the optimal control $\alpha(x,\mu,z)$, which serves a sufficient condition but not necessary one. However, this assumption simplifies the calculations involved, as shown in Proposition \ref{A5}. If the optimal control can be found explicitly and is regular enough, the assumption \eqref{p_aa_f_new} in Theorems \ref{Thm6_1} and \ref{Thm6_2} can even be removed.
\end{remark}

\section{Crucial Estimate and Global-in-time Existence}\label{sec:global}
\subsection{Hypotheses and Properties}\label{sec:hypotheses}
To obtain a global-in-time solution of \eqref{FBODE}, we only need to establish the global existence of a decoupling field $\gamma$ for the FBODE system \eqref{FBODE} on the interval $[0,T]$, which also satisfies \eqref{gammaeq}. This is because the pair $\big(X^{t,m}_s(x),Z^{t,m}_s(x)\big)$ defined by solving \eqref{xeq} and $Z^{t,m}_s(x):=\gamma(s,X^{t,m}_s(x),X^{t,m}_s\ot m)$ is a solution of the FBODE system \eqref{fbodesystem}. Moreover, the FBODE system \eqref{fbodesystem} reduces to the system \eqref{FBODE} when $\wt{T}=T$ and $p(x,\mu)=\p_x k(x,\mu)$. Therefore, establishing the global existence of a decoupling field $\gamma$ yields a global-in-time solution to the FBODE system \eqref{FBODE}.

To extend the local-in-time solution of \eqref{gammaeq} to a global-in-time solution over the interval $[0,T]$, we need to derive several new {\it a priori} estimates, which play a crucial role of extending the existence lifespan and can be established under Assumptions ${\bf(a1)}$-${\bf(a3)}$ and the following additional \textbf{Hypotheses}:\\ 
$\bf{(h1)}$ $f(0,\delta_0,0)=0$, $\p_x g(0,\delta_0,0)=\p_\alpha g(0,\delta_0,0)=0$, $\p_\mu g(0,\delta_0,0)(0)=0$, $\p_x k(0,\delta_0)=0$ and $\p_\mu k(0,\delta_0)(0)=0$, where $\delta_0$ is the Dirac measure with a point mass at $0$.\\
$\bf{(h2)}$ We set the requirement: there are positive constants $\lambda_1$ and $\lambda_2$ 
 such that, 
for any $m\in \mathcal{P}_2(\R^{d_x})$ and $X,\ \wt{X},\ Z,\ \wt{Z}\in L^{2,d_x}_m$ satisfying $(X(x),X\ot m,Z(x))\in c_{k_0}$,
\small\begin{align}\no
&\ (\lambda_g-l_g)\int_{\R^{d_x}}|\wt{X}(x)|^2dm(x)-\int_{\R^{d_x}}\int_{\R^{d_x}}\wt{Z}(x)^\top\p_\mu\p_z h(X(x),X\ot m,Z(x))\big(X(\wt{x})\big)\wt{X}(\wt{x})dm(\wt{x})dm(x)\\\label{positive_H_mu}
&\ -\int_{\R^{d_x}}\wt{Z}(x)^\top\p_z\p_z h(X(x),X\ot m,Z(x)) \wt{Z}(x)dm(x)
\geq \lambda_1\int_{\R^{d_x}}|\wt{X}(x)|^2dm(x)+\lambda_2\int_{\R^{d_x}}|\wt{Z}(x)|^2dm(x),
\end{align}\normalsize  
where $h(x,\mu,z)$ is the Hamiltonian defined in \eqref{hamilton1} and the cone set $c_{k_0}$ was defined in \eqref{c_k_0}.\\ 
$\bf{(h3)}$ Recall the constants defined in ${\bf(a1)}-{\bf(a3)}$, we here specify $\wb{l}_f\leq \frac{1}{40 \max\{\Lambda_k,L^*_0\}} \min\{\lambda_1,\lambda_g\}$ and $8\wb{l}_g\leq 
\min\{\lambda_1,\lambda_g\}$,
where the positive constant $L^*_0$ defined in \eqref{L_star_0} depends only on $\lambda_f$, $\Lambda_1$, $\Lambda_2$, $\Lambda_3$, $\lambda_1$, $\lambda_2$, $\lambda_g$, $\Lambda_g$, $l_g$, $\lambda_k$, $\Lambda_k$ and $l_k$ but not on $\wb{l}_g$ and $\wb{l}_f$.


\begin{remark}
(i) Hypotheses $\bf{(h1)}$-$\bf{(h3)}$ are satisfied in many interesting cases including Linear-Quadratic setting (see \cite{bensoussan2016linear}), also see an immediate non-linear quadratic examples in Section \ref{sec:nonLQ}. $\bf{(h1)}$ only serves for reducing the technicalities involved, it is not crucial and our methodology can even apply to the case without assuming it; we include here just for the sake of simplified illustration; indeed for the case with linear drift function $f$, we can even drop Hypothesis $\bf{(h1)}$ and set $\wb{l}_f=0$, so that $\bf{(h3)}$ can be automatically partially satisfied. \\
(ii)  $\bf{(h3)}$ tells us that exaggerated cross interactions among control, state and its measure should be avoided in order to get an equilibrium strategy for the game in long run. Many practical examples in the literature also justifies the fact that self-reinforced interactions among them will jeopardize the long term stability of the game; this is a common understanding in the literature, and some authors of the present work also have obtained such a consensus in a private communication with Peter Caines \cite{Peter2015}.
The constants chosen in  $\bf{(h3)}$ are just for convenience and we do not intend to optimize them.\\
(iii) By \eqref{positive_H_mu}, we can derive that $\lambda_1\leq\lambda_g-l_g$ by setting $\wt{Z}(x)=0$, and thus $\lambda_1\leq\lambda_g$ when $l_g$ is non-negative. In addition, under Assumptions ${\bf(a1)}$-${\bf(a3)}$ and Hypothesis $\bf{(h3)}$, we have the results stated in Propositions \ref{A1}-\ref{A6} since $\wb{l}_f\leq \frac{1}{40 \max\{\Lambda_k,L^*_0\}}\lambda_g$ implies \eqref{p_aa_f}.
\\ (iv) The condition $\wb{l}_f\leq \frac{1}{40 \max\{\Lambda_k,L^*_0\}}\min\{\lambda_1,\lambda_g\}$ in Hypothesis $\bf{(h3)}$ means that the rate of the second-order derivatives of drift function $f$ is relatively small, it means that the drift rate cannot have a large quadratic-growth effect; even for the classical control problems, the drift function with quadratic-growth may sometimes lead to ill-posedness, and this limitation on modeling can also be observed in a lot of practical applications; see the monograph \cite{bensoussan2018estimation} for instance.\\
(v) We can also set the requirement in $\bf{(h2)}$ as follows: there are positive constants $\wt{\lambda}_1$ and $\lambda_2$ 
 such that, 
for any $m\in \mathcal{P}_2(\R^{d_x})$ and $X,\ \wt{X},\ Z,\ \wt{Z}\in L^{2,d_x}_m$ satisfying $(X(x),X\ot m,Z(x))\in c_{k_0}$,
\small\begin{align}\no
&\ \left(\frac12\lambda_g-l_g\right)\int_{\R^{d_x}}|\wt{X}(x)|^2dm(x)-\int_{\R^{d_x}}\int_{\R^{d_x}}\wt{Z}(x)^\top\p_\mu\p_z h(X(x),X\ot m,Z(x))\big(X(\wt{x})\big)\wt{X}(\wt{x})dm(\wt{x})dm(x)\\\label{positive_H_mu_alter}
&\ -\int_{\R^{d_x}}\wt{Z}(x)^\top\p_z\p_z h(X(x),X\ot m,Z(x)) \wt{Z}(x)dm(x)
\geq \wt{\lambda}_1\int_{\R^{d_x}}|\wt{X}(x)|^2dm(x)+\lambda_2\int_{\R^{d_x}}|\wt{Z}(x)|^2dm(x),
\end{align}\normalsize  
which can imply \eqref{positive_H_mu} with $\lambda_1=\wt{\lambda}_1+\frac{1}{2}\lambda_g$. Thus, a sufficient condition for Hypothesis $\bf{(h3)}$ is $\wb{l}_f\leq \frac{1}{40 \max\{\Lambda_k,L^*_0\}} \cdot\frac12 \lambda_g$ and $8\wb{l}_g\leq \frac{1}{2}\lambda_g$; in particular, $\wt{\lambda}_1$ can be chosen arbitrary close to $0$.
\\(vi) The second condition of Hypothesis $\bf{(h3)}$ in the form $\wb{l}_g\leq C\lambda_g$ for some constant $C>0$ is also implicitly involved in the displacement monotonicity condition (see Assumption 3.5 (ii) in \cite{gangbo2022mean}), even in the following simplest case: consider the simplified case that, $d_x=d_\alpha=1$, $f(x,\mu,\alpha)=\alpha$, and $g(x,\mu,\alpha)=\frac{1}{2}\lambda_g x^2+\frac{1}{2}\lambda_g \alpha^2+\wb{l}_g x\alpha$ where the constants $\lambda_g>0$ and $\wb{l}_g\geq 0$. Then $\wh{\alpha}(x,\mu,z)=-\frac{1}{\lambda_g}(z+\wb{l}_g x)$ and $h(x,\mu,z)=-\frac{1}{2\lambda_g}(z+\wb{l}_g x)^2+\frac{1}{2}\lambda_g x^2$. In this case, since the Hamiltonian $h$ is independent of $\mu$, the displacement monotonicity condition is reduced to $\p_x\p_x h=-\frac{\wb{l}_g^2}{\lambda_g}+\lambda_g\geq 0$, that is, $\lambda_g\geq\wb{l}_g$.
\\(vii) 
The first condition of Hypothesis $\bf{(h3)}$ in the form $\bar{l}_f \leq C \min\{ \lambda_1, \lambda_g\}$ for some constant $C>0$ prevents the drift $f$ in forward dynamics from exhibiting excessively strong interactions among $x$, $\mu$, and $\alpha$. This requirement is a natural sufficient condition for the global-in-time solvability of the forward dynamics, as it avoids $f$ having some sort of quadratic growth that could potentially lead to a finite time blowup of the state. In practical applications, further structural information about $f$ is usually available, and one may further relax this condition accordingly. However, for the sake of convenience, we impose this condition in this work to simplify our mathematical analysis for the generic case.

\end{remark}

\begin{remark}\label{remark_h2}
(i) A sufficient condition (on the coefficient functions $f$ and $g$) for hypothesis $\bf{(h2)}$ is to assume that
\begin{align}\label{h3_eq_1}
\Big(\Lambda_2+\frac{1}{4}\Lambda_1\Big)^2-4(\lambda_g-l_g)\cdot\frac{\lambda_f^2}{\Lambda_g+\lambda_g/20}<0.
\end{align}
Under the condition \eqref{h3_eq_1}, the corresponding $\lambda_1$ and $\lambda_2$ in \eqref{positive_H_mu} can be computed as follows. First, by Hypothesis ${\bf(h2)}$, \eqref{bdd_d1_f}, \eqref{p_xalpha} and Proposition \ref{A6}, we have, for any $(x,\mu,z)\in c_{k_0}$, where $k_0$ was defined in \eqref{k_0}, and for any $y,\,\xi\in\R^{d_x}$, 
\begin{align}\no
&\|\p_\mu\p_z h(x,\mu,z)(y)\|_{\mc{L}(\R^{d_x}\times \R^{d_x};\R)}=\|\p_\mu f(x,\mu,\alpha(x,\mu,z))(y)+\p_\alpha f(x,\mu,\alpha(x,\mu,z))\p_\mu \alpha(x,\mu,z)(y)\|_{\mc{L}(\R^{d_x}; \R^{d_x})}\\\label{eq_8_3_1}
&\leq \Lambda_2 +\Lambda_1 \frac{20\big(\wb{l}_g+\frac{1}{2}k_0\cdot\wb{l}_f\big) }{19\lambda_g}<\Lambda_2+\frac{1}{4}\Lambda_1<\frac54 \Lambda_f,\\\no
& \text{ and }-\xi^\top\p_z\p_z h(x,\mu,z)\xi=-\xi^\top\p_\alpha f(x,\mu,\alpha(x,\mu,z))\p_z \alpha(x,\mu,z)\xi\geq \frac{\lambda_f^2}{\Lambda_g+\lambda_g/20}|\xi|^2,
\end{align}
which imply that, for any $m\in \mathcal{P}_2(\R^{d_x})$, and for any $X,\ \wt{X},\ Z,\ \wt{Z}\in L^{2,d_x}_m$ so that  $(X(x),X\ot m,Z(x))\in c_{k_0}$,
\small\begin{align}\no
&\ (\lambda_g-l_g)\int_{\R^{d_x}}|\wt{X}(x)|^2dm(x)-\int_{\R^{d_x}}\int_{\R^{d_x}}\wt{Z}(x)^\top\p_\mu\p_z h(X(x),X\ot m,Z(x))\big(X(\wt{x})\big)\wt{X}(\wt{x})dm(\wt{x})dm(x)\\\no
&\ -\int_{\R^{d_x}}\wt{Z}(x)^\top\p_z\p_z h(X(x),X\ot m,Z(x)) \wt{Z}(x)dm(x)\\\no
\geq &\ (\lambda_g-l_g)\int_{\R^{d_x}}|\wt{X}(\wt{x})|^2dm(\wt{x})+\frac{\lambda_f^2}{\Lambda_g+\lambda_g/20}\int_{\R^{d_x}}|\wt{Z}(x)|^2dm(x)\\\no
&\ -\Big(\Lambda_2+\frac{1}{4}\Lambda_1\Big) \int_{\R^{d_x}}|\wt{X}(\wt{x})|dm(\wt{x}) \int_{\R^{d_x}}|\wt{Z}(x)|dm(x)\\\no
\geq&\  \lambda_1\int_{\R^{d_x}}|\wt{X}(\wt{x})|^2dm(\wt{x})+\lambda_2\int_{\R^{d_x}}|\wt{Z}(x)|^2dm(x),
\end{align}\normalsize  
for some positive constants $\lambda_1$ and $\lambda_2$ depending only on $\lambda_f$, $\Lambda_1$, $\Lambda_2$, $\lambda_g$,  $\Lambda_g$ and $l_g$ since we have \eqref{h3_eq_1}; for instant, we can set $\lambda_1:=(1-\theta)(\lambda_g-l_g)$ and $\lambda_2:=(1-\theta)\frac{\lambda_f^2}{\Lambda_g+\lambda_g/20}$ where $\theta=\bigg(\Big(\Lambda_2+\frac{1}{4}\Lambda_1\Big)^2/\Big(4(\lambda_g-l_g)\cdot\frac{\lambda_f^2}{\Lambda_g+\lambda_g/20}\Big)\bigg)^{1/2}\in[0,1)$.  

In fact, by \eqref{eq_8_3_1}, the hypothesis \eqref{positive_H_mu} means that $\p_\mu \alpha(x,\mu,z)(y)$ (and thus $\p_\mu\p_\alpha f(x,\mu,\alpha)(y)$ and $\p_\mu\p_\alpha g(x,\mu,\alpha)(y)$) should be suitably small, and the derivative $\p_\mu f(x,\mu,\alpha(x,\mu,z))(y)$ should not be too large compared with the absolute value of $\p_\alpha f(x,\mu,\wh{\alpha}(x,\mu,z))(y)$ and the convexity of the running cost function $g(x,\mu,\alpha)$ in $(x,\mu)$. 
Also note that, because of \eqref{eq_8_3_1}, a sufficient condition for the hypothesis \eqref{positive_H_mu} is $\frac{25}{16}\Lambda_f^2-4(\lambda_g-l_g)\cdot\frac{\lambda_f^2}{\Lambda_g+\lambda_g/20}<0$, and readers can also find that the proposed non-linear quadratic example in Section \ref{sec:nonLQ} actually fulfills this sufficient condition.

(ii) Also note that if $\p_\mu f(x,\mu,\alpha(x,\mu,z))(y)=\p_\mu\p_\alpha f(x,\mu,\alpha)(y)=\p_\mu\p_\alpha g(x,\mu,\alpha)(y)=0$, then $\p_\mu \p_z h(x,\mu,z)(y)=0$, which means that the Hamiltonian $h(x,\mu,z)$ is separable; in this case, \eqref{positive_H_mu} is automatically satisfied with the choices of $\lambda_1=\lambda_g-l_g>0$ and $\lambda_2=\frac{\lambda_f^2}{\Lambda_g+\lambda_g/20}$.
In addition, the displacement monotonicity condition for Hamiltonian, which is similar to \eqref{positive_H_mu} but not the same, is also assumed in Assumption 3.5 (ii) (that is (3.2)) in \cite{gangbo2022mean} to deal with non-separable Hamiltonian. We here need to cope with nonlinear drift functions, both expressions $\p_x\p_x h(x,\mu,z)(y)=\p_x\p_x f(x,\mu,\alpha(x,\mu,z))\cdot z+\p_x\p_x g(x,\mu,\alpha(x,\mu,z))$ and $\p_\mu\p_x h(x,\mu,z)(y)=\p_\mu\p_x f(x,\mu,\alpha(x,\mu,z))(y)\cdot z+\p_\mu\p_x g(x,\mu,\alpha(x,\mu,z))(y)$ allow $z$ to take values without a bound, and to obtain more tractable conditions, one should impose assumptions to the coefficient functions $f$ and $g$ directly than to the Hamiltonian $h$ only.
\end{remark}

By using Hypothesis ${\bf(h1)}$, we have the following proposition which guides us with better linear-growth estimates for various coefficient functions and the equilibrium control.
\begin{prpstn}
Under Assumptions ${\bf(a1)}$-${\bf(a3)}$ and Hypotheses ${\bf(h1)}$-${\bf(h3)}$, we have $\wh{\alpha}(0,\delta_0,0)=0$ and the following linear-growth estimates hold:
\begin{align}\label{ligfb}
&i)\ |f(x,\mu,\alpha)|\leq L_f \left(|x|+|\alpha| + \|\mu\|_1 \right);\\\label{LGgDerivativesb}
&ii)\ \sup_{x\in\R^{d_x},\,\mu\in\mathcal{P}_2(\R^{d_x}),\,\alpha\in\R^{d_\alpha},\,\wt{x}\in\R^{d_x}}\frac{\left|\p_x g(x,\mu,\alpha)\right|}{|x|+ \|\mu\|_1+|\alpha|}
\vee\frac{\left|\p_\alpha g(x,\mu,\alpha)\right|}{|x|+ \|\mu\|_1+|\alpha|}\leq L_g;\\\label{LGkDerivativesb}
&iii)\ \sup_{x\in\R^{d_x},\,\mu\in\mathcal{P}_2(\R^{d_x}),\,\wt{x}\in\R^{d_x}}\frac{\left|\p_x k(x,\mu)\right|}{|x|+ \|\mu\|_1}\leq L_k;\\\label{linalphab}
&iv)\ |\wh{\alpha}(x,\mu,z)|\leq L_\alpha (|x|+\|\mu\|_1+|z|),
\end{align}
where $L_f$, $L_g$, $L_k$ and $L_\alpha$ were defined in Proposition \ref{A1}, \ref{A2} and \ref{A5}, respectively.
\end{prpstn}
\begin{proof}
By using Hypothesis ${\bf(h1)}$ and the first-order condition \eqref{first_order_condition}, $\alpha=0$ is a root to the first-order condition \eqref{first_order_condition} at the point $(x,\mu,z)=(0,\delta_0,0)$, therefore in light of the uniqueness stated in Proposition \ref{A5}, we know that the unique minimizer satisfies $\alpha(0,\delta_0,0)=0$. By using Hypothesis ${\bf(h1)}$ together with the arguments leading to \eqref{ligf} \eqref{LGgDerivatives}, \eqref{LGkDerivatives} and \eqref{linalpha}, we can similarly deduce \eqref{ligfb}-\eqref{linalphab}; in a certain sense, the proofs for \eqref{ligfb}-\eqref{linalphab} are slightly simpler that those for \eqref{ligf} \eqref{LGgDerivatives}, \eqref{LGkDerivatives} and \eqref{linalpha}.
\end{proof}

\subsection{Notations}\label{sec:notations}
For simplicity, we make use of the following notations by omitting some of their own arguments while only keeping the essential parameters and variables such as $t,\,m,\,x$ and $s$:
\begin{align}\label{notation_1}
&\begin{cases}
\alpha^{t,m}_s(x):= \alpha(X^{t,m}_s(x), X^{t,m}_s\ot m,Z^{t,m}_s(x)),\ g^{t,m}_s(x):=g(X^{t,m}_s(x),X^{t,m}_s\ot m,\alpha^{t,m}_s(x)),\\
f^{t,m}_s(x):=f(X^{t,m}_s(x),X^{t,m}_s\ot m,\alpha^{t,m}_s(x)),\ k^{t,m}_s(x):=k(X^{t,m}_s(x),X^{t,m}_s\ot m).
\end{cases}\ \ \ \  
\end{align} 
Likewise, we adopt the same style of notations for the derivatives of $\alpha$ and $f$, as shown in the following Table \ref{notation_2}: 
\begin{center}
\centering\scriptsize
\renewcommand\arraystretch{2}
\begin{tabular}{|l|l|l|}
\hline
$\big(\p_x\alpha\big)^{t,m}_s(y)=\p_x \alpha(x,\mu,z)$&
$\big(\p_\mu\alpha\big)^{t,m}_s(y,\wt{y})=\p_\mu \alpha(x,\mu,z)(\wt{x})$&
$\big(\p_z\alpha\big)^{t,m}_s(y)=\p_z \alpha(x,\mu,z)$\\ \hline
$\big(\p_x f\big)^{t,m}_s(y)=\p_x f(x,\mu,\alpha)$&
$\big(\p_\mu f\big)^{t,m}_s(y,\wt{y})=\p_\mu f(x,\mu,\alpha)(\wt{x})$&
$\big(\p_\alpha f\big)^{t,m}_s(y)=\p_\alpha f(x,\mu,\alpha)$\\\hline
$\big(\p_x\p_x f\big)^{t,m}_s(y)=\p_x\p_x f(x,\mu,\alpha)$&
$\big(\p_x\p_\mu f\big)^{t,m}_s(y,\wt{y})=\p_x\p_\mu f(x,\mu,\alpha)(\wt{x})$&
$\big(\p_{\wt{x}}\p_\mu f\big)^{t,m}_s(y,\wt{y})=\p_{\wt{x}}\p_\mu f(x,\mu,\alpha)(\wt{x})$\\\hline
$\big(\p_\alpha\p_\alpha f\big)^{t,m}_s(y)=\p_\alpha\p_\alpha f(x,\mu,\alpha)$&
$\big(\p_\alpha\p_\mu f\big)^{t,m}_s(y,\wt{y})=\p_\alpha\p_\mu f(x,\mu,\alpha)(\wt{x})$&
$\big(\p_x\p_\alpha f\big)^{t,m}_s(y)=\p_x\p_\alpha f(x,\mu,\alpha)$\\\hline
\multicolumn{3}{|c|}{$\big(\p_\mu\p_\mu f\big)^{t,m}_s(y,\wt{y},\wh{y})=\p_\mu\p_\mu f(x,\mu,\alpha)(\wt{x},\wh{x})$} \\\hline
\end{tabular}
\captionof{table}{\scriptsize Abbreviations for the derivatives of $\alpha$ and $f$, where $y,\,\wt{y},\,\wh{y}\in\R^{d_x}$, $t\in[0,T]$, $m\in\mathcal{P}_2(\R^{d_x})$, $s\in[t,T]$ and we also evaluate all of them at $(x,\mu,z,\alpha,\wt{x},\wh{x})=(X^{t,m}_s(y), X^{t,m}_s\ot m,Z^{t,m}_s(y),\alpha^{t,m}_s(y),X^{t,m}_s(\wt{y}),X^{t,m}_s(\wh{y}))\in \R^{d_x}\times \mathcal{P}_2(\R^{d_x})\times \R^{d_x}\times\R^{d_\alpha}\times\R^{d_x}\times\R^{d_x}$.}
\label{notation_2}
\end{center}
Using these notations and the chain rule, one can write: 
\footnotesize\begin{equation}
\begin{aligned}
\p_t\big(\alpha^{t,m}_s(y)\big)=&\ \big(\p_x  \alpha\big)^{t,m}_s(y)\cdot\p_t\big(X^{t,m}_s(y)\big)+\displaystyle\int_{\R^{d_x}}\big(\p_\mu \alpha\big)^{t,m}_s(y,\wt{y})\cdot \p_t\big(X^{t,m}_s(\wt{y})\big)dm(\wt{y})+\big(\p_z \alpha\big)^{t,m}_s(y)\cdot\p_t\big(Z^{t,m}_s(y)\big);\\
\p_t\big(f^{t,m}_s(y)\big)=&\ \big(\p_x f\big)^{t,m}_s(y)\cdot\p_t\big(X^{t,m}_s(y)\big)+\displaystyle\int_{\R^{d_x}}\big(\p_\mu f\big)^{t,m}_s(y,\wt{y})\cdot\p_t\big(X^{t,m}_s(\wt{y})\big)dm(\wt{y})+ \big(\p_\alpha f\big)^{t,m}_s(y)\cdot\p_t\big(\alpha^{t,m}_s(y)\big);\\
\p_t\big(g^{t,m}_s(y)\big)=&\ \big(\p_x g\big)^{t,m}_s(y)\cdot\p_t\big(X^{t,m}_s(y)\big)+\displaystyle\int_{\R^{d_x}}\big(\p_\mu g\big)^{t,m}_s(y,\wt{y})\cdot\p_t\big(X^{t,m}_s(\wt{y})\big)dm(\wt{y})+\big(\p_\alpha g\big)^{t,m}_s(y)\cdot\p_t\big(\alpha^{t,m}_s(y)\big);\\
\p_t\big(k^{t,m}_s(y)\big)=&\ \big(\p_x k\big)^{t,m}_s(y)\cdot\p_t\big(X^{t,m}_s(y)\big)+\displaystyle\int_{\R^{d_x}}\big(\p_\mu k\big)^{t,m}_s(y,\wt{y})\cdot\p_t\big(X^{t,m}_s(\wt{y})\big)dm(\wt{y});
\end{aligned}\label{calculus1}
\end{equation}\normalsize
similarly, we also have a set of similar equations for $\p_m \big(\alpha^{t,m}_s(y)\big)(\wt{y})$, $\p_m \big(f^{t,m}_s(y)\big)(\wt{y})$, $\p_m \big(g^{t,m}_s(y)\big)(\wt{y})$ and $\p_m \big(k^{t,m}_s(y)\big)(\wt{y})$, and $\p_x \big(\alpha^{t,m}_s(y)\big)$, $\p_x \big(f^{t,m}_s(y)\big)$, $\p_x \big(g^{t,m}_s(y)\big)$, $\p_x \big(k^{t,m}_s(y)\big)$. For instance, $\p_m \big(\alpha^{t,m}_s(y)\big)(\wt{y})$, $\p_m \big(f^{t,m}_s(y)\big)(\wt{y})$, $\p_y \big(\alpha^{t,m}_s(y)\big)$, $\p_y \big(f^{t,m}_s(y)\big)$ can be written as, for any $y,\,\wt{y}\in\R^{d_x}$, $t\in[0,T]$, $m\in\mathcal{P}_2(\R^{d_x})$ and $s\in[t,T]$,
\footnotesize\begin{equation}
\begin{aligned}
\p_m\big(\alpha^{t,m}_s(y)\big)(\wt{y})=&\ \big(\p_x  \alpha\big)^{t,m}_s(y)\cdot\p_m\big(X^{t,m}_s(y)\big)(\wt{y})+\big(\p_\mu \alpha\big)^{t,m}_s(y,\wt{y})\cdot \p_{\wt{y}}\big(X^{t,m}_s(\wt{y})\big)\\
&+\displaystyle\int_{\R^{d_x}}\big(\p_\mu \alpha\big)^{t,m}_s(y,\wh{y})\cdot \p_m\big(X^{t,m}_s(\wh{y})\big)(\wt{y})dm(\wh{y})+\big(\p_z \alpha\big)^{t,m}_s(y)\cdot\p_m\big(Z^{t,m}_s(y)\big)(\wt{y});\\
\p_m\big(f^{t,m}_s(y)\big)(\wt{y})=&\ \big(\p_x f\big)^{t,m}_s(y)\cdot\p_m\big(X^{t,m}_s(y)\big)(\wt{y})+\big(\p_\mu f\big)^{t,m}_s(y,\wt{y})\cdot \p_{\wt{y}}\big(X^{t,m}_s(\wt{y})\big)\\
&+\displaystyle\int_{\R^{d_x}} \big(\p_\mu f\big)^{t,m}_s(y,\wh{y})\cdot\p_m\big(X^{t,m}_s(\wh{y})\big)(\wt{y}) dm(\wh{y})+ \big(\p_\alpha f\big)^{t,m}_s(y)\cdot \p_m\big(\alpha^{t,m}_s(y)\big)(\wt{y});\\
\p_y\big(\alpha^{t,m}_s(y)\big)=&\ \big(\p_x  \alpha\big)^{t,m}_s(y)\cdot\p_y\big(X^{t,m}_s(y)\big)+\big(\p_z \alpha\big)^{t,m}_s(y)\cdot\p_y\big(Z^{t,m}_s(y)\big);\\
\p_y \big(f^{t,m}_s(y)\big)=&\ \big(\p_x f\big)^{t,m}_s(y)\cdot\p_y \big(X^{t,m}_s(y)\big)+ \big(\p_\alpha f\big)^{t,m}_s(y)\cdot \p_y \big(\alpha^{t,m}_s(y)\big).
\end{aligned}\label{calculus2}
\end{equation}\normalsize

\subsection{Crucial {\it a Priori} Estimate}
Now, we are ready to establish some {\it a priori} estimates.
For any fixed $t\in[0 ,T]$ and $m\in \mathcal{P}_2(\R^{d_x})$, it follows from Theorem \ref{Thm6_1} and \ref{Thm6_2} that there is a (local-in-time) unique solution pair $\big(X^{t,m}_s(x),Z^{t,m}_s(x):=\gamma(s,X^{t,m}_s(x),X^{t,m}_s\ot m)\big)$ to \eqref{FBODE}, each component function of which is continuously differentiable in $x\in\R^{d_x}$ and continuously $L$-differentiable in $m\in\mc{P}_2(\R^{d_x})$ with the corresponding derivatives being continuously differentiable in $s\in [t,T]$. Then the couples $\big(\p_m\big(X^{t,m}_s(x)\big)(y),\p_m\big(Z^{t,m}_s(x)\big)(y)\big)$ and $\big(\p_x\big(X^{t,m}_s(x)\big)$, $\p_x\big(Z^{t,m}_s(x)\big)\big)$ satisfy the linear FBODE systems, respectively:
\footnotesize\begin{align}\label{eq_8_5_new}
\left\{ \begin{aligned}
	\frac{d}{ds}\p_m\big(X^{t,m}_s(x)\big)(y) =& \p_m \big(f^{t,m}_s(x)\big)(y),\\
	\p_m \big(X^{t,m}_t(x)\big)(y) =&0;\\
-\frac{d}{ds}\p_m\big(Z^{t,m}_s(x)\big)(y)=& \big(\p_x f\big)^{t,m}_s(x)\cdot \p_m\big(Z^{t,m}_s(x)\big)(y)\\
&+\p_m\big(X^{t,m}_s(x)\big)(y)^\top\bigg(\big(\p_x\p_x f\big)^{t,m}_s(x)\cdot Z^{t,m}_s(x)+\big(\p_x\p_x g\big)^{t,m}_s(x)\bigg)\\
&+\int_{\R^{d_x}}\p_m\big(X^{t,m}_s(\wt{x})\big)(y)^\top\bigg(\big(\p_\mu\p_x f\big)^{t,m}_s(x,\wt{x})\cdot Z^{t,m}_s(x)+\big(\p_\mu\p_x g\big)^{t,m}_s(x,\wt{x})\bigg)dm(\wt{x})\\
&+\p_y\big(X^{t,m}_s(y)\big)^\top\bigg(\big(\p_\mu\p_x f\big)^{t,m}_s(x,y)\cdot Z^{t,m}_s(x)+\big(\p_\mu\p_x g\big)^{t,m}_s(x,y)\bigg)\\
&+\p_m\big(\alpha^{t,m}_s(x)\big)(y)^\top\bigg(\big(\p_\alpha\p_x f\big)^{t,m}_s(x)\cdot Z^{t,m}_s(x)+\big(\p_\alpha\p_x g\big)^{t,m}_s(x)\bigg),\\
\p_m\big(Z^{t,m}_T(x)\big)(y)=& \big(\p_x\p_x k\big)^{t,m}_T(x)\cdot \p_m\big(X^{t,m}_T(x)\big)(y)\\
&+\displaystyle\int_{\R^{d_x}}\big(\p_\mu\p_x k\big)^{t,m}_T(x,\wt{x})\cdot \p_m\big(X^{t,m}_T(\wt{x})\big)(y)dm(\wt{x})+\big(\p_\mu\p_x k\big)^{t,m}_T(x,y)\cdot \p_y\big(X^{t,m}_T(y)\big);
\end{aligned} \right.
\end{align}\normalsize
and 
\small\begin{align}\label{eq_8_6_new}
\left\{ \begin{aligned}
	\frac{d}{ds}\p_x\big(X^{t,m}_s(x)\big) =& \p_x \big(f^{t,m}_s(x)\big),\\
	\p_x \big(X^{t,m}_t(x)\big) =&\mathcal{I}_{d_x\times d_x};\\
	-\frac{d}{ds}\p_x \big(Z^{t,m}_s(x)\big)  =&\big(\p_x f\big)^{t,m}_s(x)\cdot \p_x\big(Z^{t,m}_s(x)\big)+\p_x\big(X^{t,m}_s(x)\big)^\top\bigg(\big(\p_x\p_x f\big)^{t,m}_s(x)\cdot Z^{t,m}_s(x)+\big(\p_x\p_x g\big)^{t,m}_s(x)\bigg)\\
&+\p_x\big(\alpha^{t,m}_s(x)\big)^\top\bigg(\big(\p_\alpha\p_x f\big)^{t,m}_s(x)\cdot Z^{t,m}_s(x)+\big(\p_\alpha\p_x g\big)^{t,m}_s(x)\bigg),\\
\p_x\big(Z^{t,m}_T(x)\big)=& \big(\p_x\p_x k\big)^{t,m}_T(x)\cdot \p_x\big(X^{t,m}_T(x)\big);
\end{aligned} \right.
\end{align}\normalsize
where $\mathcal{I}_{d_x\times d_x}$ is the $d_x\times d_x$ identity matrix.

\begin{remark}
Here, let us recall that the flow map $X^{t,m}_s(\cdot)$ actually maps the initial point $x$ at the initial time $t$ to another location at the later times $s$ by using the forward dynamics, so as in \eqref{FBODE}, the initial condition of this flow map is $X^{t,m}_s(x)=x$, which is the identity map from $\R^{d_x}$ to $\R^{d_x}$. Thus, the $x$ and $m$-derivatives of this flow map at the initial time $t$ are as  in \eqref{eq_8_5_new} and \eqref{eq_8_6_new}.
\end{remark}

For the sake of convenience, we denote
\begin{align}\label{eq_7_6}
L^*:=\vertiii{\gamma }_{2,[t,T]}<\infty,
\end{align}
where $\vertiii{\cdot}_2$ was defined in \eqref{Gamma_2}.
Here, at least close to the terminal time $T$, the finiteness of $L^*$ is guaranteed by the construction of the local-in-time solution, which was ensured at least for $t\in[0,T]$ with $T-t\leq \eps_2$, where $\eps_2$ was given in Theorem \ref{Thm6_2}. Furthermore, the decoupling field $\gamma$ in \eqref{eq_7_6} was constructed in Theorem \ref{Thm6_1}. Then, by using \eqref{eq_9_3}, 
\begin{align*}
&\big|Z^{t,m}_t(x)-Z^{t,m_0}_t(x_0)\big|=\big|\gamma(t,x,m)-\gamma(t,x_0,m_0)\big|
\leq \ L^*\Big(|x-x_0|+W_1(m,m_0)\Big).
\end{align*}
By using Hypothesis ${\bf (h1)}$, one can easily check that, for any fixed $t\in[0,T]$, the pair $\Big(X^{t,\delta_0}_s(0)\equiv 0,Z^{t,\delta_0}_s(0)\equiv 0\Big)$, as a function couple in timne $s$, solves the FBODE \eqref{FBODE} with fixed $x=0$ and $m=\delta_0$. Therefore, we have $\big|Z^{t,m}_t(x)\big|=\big|\gamma(t,x,m)\big|\leq L^*\Big(|x|+\|m\|_1\Big)$, which further implies that $\big\|Z^{t,m}_t\big\|_{L^{1,d_x}_{m}}=\big\|\gamma(t,\cdot,m)\big\|_{L^{1,d_x}_{m}}\leq 2L^*\|m\|_1$ by taking an integration. Furthermore, we also have 
\begin{align}\label{eq_8_8_cone}
&\big|Z^{t,m}_s(x)\big|=\big|\gamma(s,X^{t,m}_s(x),X^{t,m}_s\ot m)\big|\leq L^*\Big(|X^{t,m}_s(x)|+\|X^{t,m}_s\|_{L^{1,d_x}_{m}}\Big),\\
&\big\|Z^{t,m}_s\big\|_{L^{1,d_x}_{m}}\leq 2 L^*\|X^{t,m}_s\|_{L^{1,d_x}_{m}}.
\end{align}

Now, we have the following crucial {\it a priori} estimates.
\begin{theorem}\label{Crucial_Estimate}
Under Assumptions ${\bf(a1)}$-${\bf(a3)}$ and Hypotheses ${\bf(h1)}$-${\bf(h3)}$, let $\gamma(s,x,m)$ be the decoupling field for the FBODE system \eqref{fbodesystem} on $[t,T]$ with terminal data $p(x,\mu)=\p_x k(x,\mu)$, which is differentiable in 
$x\in\R^{d_x}$ and $L$-differentiable in $\mu\in\mc{P}_2(\R^{d_x})$, and also suppose that $L^*\leq 2\max\{L^*_0,\Lambda_k\}$, we then have the following estimate:
\begin{align}\label{JFE}
L^*\leq L^*_0,
\end{align}
where $L^*$ is defined in \eqref{eq_7_6} and $L^*_0$ to be defined in \eqref{L_star_0}.
\end{theorem}

\begin{remark}
In particular, we shall show that the constant $L^*_0$ defined in \eqref{L_star_0} is always larger than $\Lambda_k$; also see Remark \ref{remark_lower_bdd_L_star_0} for more detailed justification. In the statement of Theorem \ref{Crucial_Estimate}, the use of $\max\{L^*_0,\Lambda_k\}$ (instead of $L^*_0$) is to strongly emphasize that, when the time $t$ is close to the terminal time $T$, the quantity $L^*$ must be less than or equal to $2\Lambda_k$. This fact is extremely important since the estimates of $\gamma$ must strongly depend on the terminal cost function $k$.
\end{remark}

\begin{proof}
Define the pair $\big(X^{t,m}_s(x),Z^{t,m}_s(x)\big)$ by the equation \eqref{xeq} and $Z^{t,m}_s(x):=\gamma(s,X^{t,m}_s(x),X^{t,m}_s\ot m)$, then it is a solution pair of FBODE system \eqref{fbodesystem} by \eqref{gammaeq}. The differentiabilities of $X^{t,m}_s(x)$ and $Z^{t,m}_s(x)$ that were stated in Theorem \ref{Thm6_2} follow by the properties of $\gamma(s,x,\mu)$.

The idea of proof follows from that of Theorem 8.3 in our previous paper \cite{bensoussan2023theory} and will also be progressed in five steps. In Step 1, we shall first provide some calculus identities and elementary estimates that are useful in deriving estimates for the solutions to the Jacobian flows. Step 2 provides the preliminary estimates on different derivatives of $\gamma$ via the backward ODE. To control the upper bounds of these preliminary estimates, we shall derive our major estimates (e.g., \eqref{L_star_1_1} and \eqref{cru_est_2}) by integrating the inner products of forward and backward flows in Steps 3 and 4. Finally, in Step 5, we shall combine all the aforementioned estimates to obtain a uniform (in the terminal time $T$) bound on the derivatives of $\gamma$.

It follows from \eqref{k_0} that $\frac{1}{2}k_0=2\max\{L^*_0,\Lambda_k\}\geq L^*$, so using \eqref{c_k_0} and \eqref{eq_8_8_cone}, we know that $\big(X^{t,m}_s(x),X^{t,m}_s\ot m, Z^{t,m}_s(x)\big)\in c_{k_0}$, and hence the unique minimizer $\alpha^{t,m}_s(x):=\alpha\big(X^{t,m}_s(x),X^{t,m}_s\ot m, Z^{t,m}_s(x)\big)$ always exists according to Proposition \ref{A5}. Furthermore, Hypothesis ${\bf (h3)}$ implies that $\wb{l}_f L^* \leq \frac{1}{2}k_0\wb{l}_f \leq \frac{1}{20}\min\{\lambda_g,\lambda_1\}$ and $\wb{l}_g \leq \frac{1}{8} \lambda_g$.
Denote $\Lambda_h:=\frac{1}{2}k_0\wb{l}_f+\Lambda_g\leq \frac{1}{20}\min\{\lambda_1,\lambda_g\}+\Lambda_g$ and $\wb{l}_h:=\frac{1}{2}k_0\wb{l}_f+\wb{l}_g\leq \frac{1}{20}\min\{\lambda_1,\lambda_g\}+\frac{1}{8}\min\{\lambda_1,\lambda_g\}<\frac{1}{5}\min\{\lambda_1,\lambda_g\}$. 

\textbf{Step 1} (Calculus identities and some elementary estimates) We shall first introduce a few calculus identities that will be useful in Steps 3-4, and derive some elementary estimates that will be used in Steps 2-5.
First of all, differentiating \eqref{first_order_condition} with respect to $x$, $\mu$ and $z$, respectively, it yields that
\begin{align}\no
&\p_x\p_\alpha f(x,\mu,\alpha(x,\mu,z))\cdot z  + \p_x\p_\alpha g(x,\mu,\alpha(x,\mu,z))\\\label{p_x_foc}
&\ \ \ \ \ \ \ \ \ \ \  \ +\Big(\p_\alpha\p_\alpha f(x,\mu,\alpha(x,\mu,z))\cdot z  + \p_\alpha\p_\alpha g(x,\mu,\alpha(x,\mu,z))\Big)\cdot \p_x\alpha(x,\mu,z)=0,\\\no
&\p_\mu\p_\alpha f(x,\mu,\alpha(x,\mu,z))(\wt{x})\cdot z  + \p_\mu\p_\alpha g(x,\mu,\alpha(x,\mu,z))(\wt{x})\\\label{p_mu_foc}
&\ \ \ \ \ \ \ \ \ \ \  \ +\Big(\p_\alpha\p_\alpha f(x,\mu,\alpha(x,\mu,z))\cdot z  + \p_\alpha\p_\alpha g(x,\mu,\alpha(x,\mu,z))\Big)\cdot \p_\mu\alpha(x,\mu,z)(\wt{x})=0,\\\label{p_z_foc}
&\p_\alpha f(x,\mu,\alpha(x,\mu,z))+\Big(\p_\alpha\p_\alpha f(x,\mu,\alpha(x,\mu,z))\cdot z  + \p_\alpha\p_\alpha g(x,\mu,\alpha(x,\mu,z))\Big)\cdot \p_z\alpha(x,\mu,z)=0,
\end{align}
and hence, by considering \eqref{p_z_foc} multiplied with $\p_x\alpha(x,\mu,z)$ minus \eqref{p_x_foc} multiplied with $\p_z\alpha(x,\mu,z)$, and \eqref{p_z_foc} multiplied with $\p_\mu\alpha(x,\mu,z)$ minus \eqref{p_x_foc} multiplied with $\p_z\alpha(x,\mu,z)$, we obtain
\begin{align}\no
&\p_\alpha f(x,\mu,\alpha(x,\mu,z))\cdot \p_x\alpha(x,\mu,z)\\\label{eq_7_17_1}
&\ \ \ \ \ \ \ \ \ \ \  \ -\Big(\p_x\p_\alpha f(x,\mu,\alpha(x,\mu,z))\cdot z  + \p_x\p_\alpha g(x,\mu,\alpha(x,\mu,z))\Big)\cdot \p_z\alpha(x,\mu,z)=0,\\\no
&\p_\alpha f(x,\mu,\alpha(x,\mu,z))\cdot \p_\mu\alpha(x,\mu,z)(\wt{x})\\\label{eq_7_18_1}
&\ \ \ \ \ \ \ \ \ \ \  \ -\Big(\p_\mu\p_\alpha f(x,\mu,\alpha(x,\mu,z))(\wt{x})\cdot z  + \p_\mu\p_\alpha g(x,\mu,\alpha(x,\mu,z))(\wt{x})\Big)\cdot \p_z\alpha(x,\mu,z)=0.
\end{align}
Next, we are going to derive some elementary estimates and understand how these estimates depend on $L^*$ as follows. It follows from \eqref{bdd_d2_f}, \eqref{positive_g_alpha} and \eqref{eq_8_8_cone} that for any $\xi\in\R^{d_\alpha}$, 
\small\begin{align}\label{eq_8_11_new}
\xi^\top\Big((\p_\alpha\p_\alpha f)^{t,m}_s(x)\cdot Z^{t,m}_s(x)+(\p_\alpha\p_\alpha g)^{t,m}_s(x)\Big)\xi \geq \Big(\lambda_g-\wb{l}_fL^*\Big)|\xi|^2\geq \frac{19}{20}\lambda_g|\xi|^2>0.
\end{align}\normalsize
By \eqref{bdd_d1_f}, \eqref{bdd_d2_f}, \eqref{bdd_d2_g_2}, \eqref{eq_8_8_cone}, \eqref{p_x_foc}-\eqref{p_z_foc} and the first inequality of \eqref{eq_8_11_new}, we obtain
\begin{align}\label{p_xalpha_new}
&\Big\|(\p_x \alpha)^{t,m}_s(x)\Big\|_{\mathcal{L}(\R^{d_x};\R^{d_\alpha})}\vee\Big\|(\p_{\mu} \alpha)^{t,m}_s(x,\wt{x})\Big\|_{\mathcal{L}(\R^{d_x};\R^{d_\alpha})}\leq\frac{20\wb{l}_h}{19\lambda_g},\\\label{p_zalpha_new}
\text{ and } & \Big\|(\p_{z} \alpha)^{t,m}_s(x)\Big\|_{\mathcal{L}(\R^{d_x};\R^{d_\alpha})}\leq \frac{20\Lambda_f}{19\lambda_g}.
\end{align}\normalsize
Furthermore, applying \eqref{bdd_d1_f}, \eqref{p_xalpha_new} and \eqref{p_zalpha_new} to \eqref{calculus2}, one yields that
\begin{align}\no
\Big|\p_m\big(\alpha^{t,m}_s(x)\big)(y)\Big|\leq&\ \frac{20\wb{l}_h}{19\lambda_g}\bigg(\Big|\p_m\big(X^{t,m}_s(x)\big)(y)\Big|+\Big|\p_{y}\big(X^{t,m}_s(y)\big)\Big|+\int_{\R^{d_x}}\Big|\p_m\big(X^{t,m}_s(\wt{x})\big)(y)\Big|dm(\wt{x})\bigg)\\\label{p_m_alpha_1}
&+\frac{20\Lambda_f}{19\lambda_g}\Big|\p_m\big(Z^{t,m}_s(x)\big)(y)\Big|;\\\label{p_x_alpha_1}
\Big|\p_x\big(\alpha^{t,m}_s(x)\big)\Big|\leq&\ \ \frac{20\wb{l}_h}{19\lambda_g}\Big|\p_x\big(X^{t,m}_s(x)\big)\Big|+\frac{20\Lambda_f}{19\lambda_g}\Big|\p_x\big(Z^{t,m}_s(x)\big)\Big|;\end{align}
\begin{align}\no
\Big|\p_m\big(f^{t,m}_s(x)\big)(y)\Big|\leq&\ \Lambda_f\bigg(\Big|\p_m\big(X^{t,m}_s(x)\big)(y)\Big|+\Big| \p_{y}\big(X^{t,m}_s(y)\big)\Big|+\int_{\R^{d_x}} \Big|\p_m\big(X^{t,m}_s(\wt{x})\big)(y)\Big| dm(\wt{x})\\\label{p_m_f_1}
&\ \ \ \ \ \ \ \ + \Big|\p_m\big(\alpha^{t,m}_s(x)\big)(y)\Big|\bigg);\\\label{p_x_f_1}
\Big|\p_x \big(f^{t,m}_s(x)\big)\Big|\leq&\ \Lambda_f\bigg(\Big|\p_x \big(X^{t,m}_s(x)\big)\Big|+ \Big|\p_x \big(\alpha^{t,m}_s(x)\big)\Big|\bigg).
\end{align}

\textbf{Step 2} (Preliminary estimates on derivatives of $\gamma$) 
For $i=1,2,3,...,d_x$, denote $\p_{m_i}\gamma(t,x,m)(y):=\p_{y_i}\frac{\delta}{\delta m}\gamma(t,x,m)(y)$. By using Assumptions ${\bf(a1)}$-${\bf(a3)}$ and Hypotheses ${\bf(h1)}$ and ${\bf(h3)}$, we can obtain, for $i=1,2,3,...,d_x$, 
\small\begin{align}\no
&\bigg|\p_{m_i}\gamma(t,x,m)(y)\bigg|^2=\bigg|\p_{m_i}\big(Z^{t,m}_t(x)\big)(y)\bigg|^2\\\no
=&\bigg|\p_{m_i}\big(Z^{t,m}_T(x)\big)(y)\bigg|^2-2\int_t^T \p_{m_i}\big(Z^{t,m}_s(x)\big)(y)\cdot \frac{d}{ds}\p_{m_i}\big(Z^{t,m}_s(x)\big)(y) ds\\\no
\leq &\Lambda_k^2\bigg(\int_{\R^{d_x}}\Big|\p_{m_i}\big(X^{t,m}_T(\wt{x})\big)(y)\Big|dm(\wt{x})+\Big|\p_{m_i}\big(X^{t,m}_T(x)\big)(y)\Big|+\Big|\p_{y_i}\big(X^{t,m}_T(y)\big)\Big|\bigg)^2\\\no
&+2\int_t^T\bigg(\Lambda_h \int_{\R^{d_x}}\Big|\p_{m_i}\big(X^{t,m}_s(\wt{x})\big)(y)\Big|dm(\wt{x})+\Lambda_h\Big|\p_{m_i}\big(X^{t,m}_s(x)\big)(y)\Big|\\\no
&\ \ \ \ \ \ \ \ \ \ \ \ +\Lambda_h\Big|\p_{y_i}\big(X^{t,m}_s(y)\big)\Big|+\wb{l}_h \Big|\p_{m_i}\big(\alpha^{t,m}_s(x)\big)(y)\Big|+\Lambda_f\Big|\p_{m_i}\big(Z^{t,m}_s(x)\big)(y)\Big|\bigg)\cdot\Big|\p_{m_i}\big(Z^{t,m}_s(x)\big)(y)\Big|ds\\\no
&\text{(by using \eqref{bdd_d1_f}, \eqref{bdd_d2_f}, \eqref{bdd_d2_g_1}, \eqref{bdd_d2_g_2} and \eqref{bdd_d2_k_1})}\\\no
\leq &\Lambda_k^2\bigg(\int_{\R^{d_x}}\Big|\p_{m_i}\big(X^{t,m}_T(\wt{x})\big)(y)\Big|dm(\wt{x})+\Big|\p_{m_i}\big(X^{t,m}_T(x)\big)(y)\Big|+\Big|\p_{y_i}\big(X^{t,m}_T(y)\big)\Big|\bigg)^2\\\no
&+2\int_t^T\bigg(\Big(\Lambda_h+\frac{20\wb{l}_h^2}{19\lambda_g}\Big) \int_{\R^{d_x}}\Big|\p_{m_i}\big(X^{t,m}_s(\wt{x})\big)(y)\Big|dm(\wt{x})+\Big(\Lambda_h+\frac{20\wb{l}_h^2}{19\lambda_g}\Big)\Big|\p_{m_i}\big(X^{t,m}_s(x)\big)(y)\Big|\\\no
&\ \ \ \ \ \ \ \ \ \ \ \ +\Big(\Lambda_h+\frac{20\wb{l}_h^2}{19\lambda_g}\Big)\Big|\p_{y_i}\big(X^{t,m}_s(y)\big)\Big|+\Big(\Lambda_f+\frac{20\Lambda_f}{19\lambda_g}\wb{l}_h\Big)\Big|\p_{m_i}\big(Z^{t,m}_s(x)\big)(y)\Big|\bigg)\cdot\Big|\p_{m_i}\big(Z^{t,m}_s(x)\big)(y)\Big|ds\\\no
&\text{(by using \eqref{p_m_alpha_1})}\\\no
\leq &3\Lambda_k^2 \int_{\R^{d_x}}\Big|\p_{m_i}\big(X^{t,m}_T(\wt{x})\big)(y)\Big|^2dm(\wt{x})+3\Lambda_k^2\Big|\p_{m_i}\big(X^{t,m}_T(x)\big)(y)\Big|^2+3\Lambda_k^2\Big|\p_{y_i}\big(X^{t,m}_T(y)\big)\Big|^2\\\no
&+\int_t^T\Bigg(\Big(\Lambda_h+\frac{20\wb{l}_h^2}{19\lambda_g}\Big)\int_{\R^{d_x}}\Big|\p_{m_i}\big(X^{t,m}_s(\wt{x})\big)(y)\Big|^2dm(\wt{x})+\Big(\Lambda_h+\frac{20\wb{l}_h^2}{19\lambda_g}\Big)\Big|\p_{m_i}\big(X^{t,m}_s(x)\big)(y)\Big|^2\\\label{p_m_gamma}
&\ \ \ \ \ \ \ \ \ \ \ \ +\Big(\Lambda_h+\frac{20\wb{l}_h^2}{19\lambda_g}\Big)\Big|\p_{y_i}\big(X^{t,m}_s(y)\big)\Big|^2+\bigg(\Lambda_f+\frac{20\Lambda_f}{19\lambda_g}\wb{l}_h+3\Lambda_h+\frac{60\wb{l}_h^2}{19\lambda_g}\bigg)\Big|\p_{m_i}\big(Z^{t,m}_s(x)\big)(y)\Big|^2\Bigg)ds
\end{align}\normalsize
and, for $i=1,2,3,...,d_x$, we also consider
\small\begin{align}\no
&\bigg|\p_{x_i}\gamma(t,x,m)\bigg|^2=\bigg|\p_{x_i}\big(Z^{t,m}_t(x)\big)\bigg|^2=\bigg|\p_{x_i}\big(Z^{t,m}_T(x)\big)\bigg|^2-2\int_t^T \p_{x_i}\big(Z^{t,m}_s(x)\big)\cdot \frac{d}{ds}\p_{x_i}\big(Z^{t,m}_s(x)\big) ds\\\no
\leq &\Lambda_k^2\Big|\p_{x_i}\big(X^{t,m}_T(x)\big)\Big|^2+2\int_t^T\bigg(\Lambda_h\Big|\p_{x_i}\big(X^{t,m}_s(x)\big)\Big|+\wb{l}_h \Big|\p_{x_i}\big(\alpha^{t,m}_s(x)\big)\Big|+\Lambda_f\Big|\p_{x_i}\big(Z^{t,m}_s(x)\big)\Big|\bigg)\cdot\Big|\p_{x_i}\big(Z^{t,m}_s(x)\big)\Big|ds\\\no
&\text{(by using \eqref{bdd_d1_f}, \eqref{bdd_d2_f}, \eqref{bdd_d2_g_1}, \eqref{bdd_d2_g_2} and \eqref{bdd_d2_k_1})}\\\no
\leq &\Lambda_k^2\Big|\p_{x_i}\big(X^{t,m}_T(x)\big)\Big|^2+2\int_t^T\Bigg(\bigg(\Lambda_h+\frac{20\wb{l}_h^2}{19\lambda_g}\bigg)\Big|\p_{x_i}\big(X^{t,m}_s(x)\big)\Big|+\bigg(\frac{20\Lambda_f}{19\lambda_g}\wb{l}_h+\Lambda_f\bigg)\Big|\p_{x_i}\big(Z^{t,m}_s(x)\big)\Big|\Bigg)\cdot\Big|\p_{x_i}\big(Z^{t,m}_s(x)\big)\Big|ds\\\no
&\text{(by using \eqref{p_x_alpha_1})}\\\label{p_x_gamma}
\leq &\Lambda_k^2\Big|\p_{x_i}\big(X^{t,m}_T(x)\big)\Big|^2+\bigg(\Lambda_h+\frac{20\wb{l}_h ^2}{19\lambda_g}\bigg)\int_t^T \Big|\p_{x_i}\big(X^{t,m}_s(x)\big)\Big|^2ds+\bigg(\frac{40\Lambda_f}{19\lambda_g}\wb{l}_h+2\Lambda_f+\Lambda_h+ \frac{20\wb{l}_h ^2}{19\lambda_g}\bigg)\int_t^T\Big|\p_{x_i}\big(Z^{t,m}_s(x)\big)\Big|^2ds
\end{align}\normalsize
Then, by integrating \eqref{p_m_gamma} with respect to $x$ and and using Jensen's inequality to bound $L^1$-norm by the corresponding $L^2$-norm, we have, for $i=1,2,3,...,d_x$, 
\small\begin{align}\no
&\bigg\|\p_{m_i}\gamma(t,\cdot,m)(y)\bigg\|_{L^{2,d_x}_m}^2=\bigg\|\p_{m_i}\big(Z^{t,m}_t(\cdot)\big)(y)\bigg\|_{L^{2,d_x}_m}^2\\\no
\leq &6\Lambda_k^2\Big\|\p_{m_i}\big(X^{t,m}_T(\cdot)\big)(y)\Big\|_{L^{2,d_x}_m}^2+3\Lambda_k^2\Big|\p_{y_i}\big(X^{t,m}_T(y)\big)\Big|^2\\\no
&+\int_t^T\Bigg(\Big(2\Lambda_h+\frac{40\wb{l}_h^2}{19\lambda_g}\Big)\Big\|\p_{m_i}\big(X^{t,m}_s(\cdot)\big)(y)\Big\|_{L^{2,d_x}_m}^2+\Big(\Lambda_h+\frac{20\wb{l}_h^2}{19\lambda_g}\Big)\Big|\p_{y_i}\big(X^{t,m}_s(y)\big)\Big|^2\\\label{p_m_gamma_int}
&\ \ \ \ \ \ \ \ \ \ \ \ +\bigg(\Lambda_f+\frac{20\Lambda_f}{19\lambda_g}\wb{l}_h+3\Lambda_h+\frac{60\wb{l}_h^2}{19\lambda_g}\bigg)\Big\|\p_{m_i}\big(Z^{t,m}_s(\cdot)\big)(y)\Big\|_{L^{2,d_x}_m}^2\Bigg)ds.
\end{align}\normalsize

To bound the right hand sides of \eqref{p_m_gamma}-\eqref{p_m_gamma_int}, we need the following major estimates.

\textbf{Step 3} (Major  estimate 1 of the following \eqref{L_star_1_1}) For $i=1,2,3,...,d_x$,
\footnotesize\begin{align*}
&\p_{x_i}\big(Z^{t,m}_T(x)\big)\cdot \p_{x_i}\big(X^{t,m}_T(x)\big)-\p_{x_i}\big(Z^{t,m}_t(x)\big)\cdot \p_{x_i}\big(X^{t,m}_t(x)\big) \\
=&\ \int_t^T \bigg(\p_{x_i}\big(Z^{t,m}_s(x)\big) \cdot \frac{d}{ds}\p_{x_i}\big(X^{t,m}_s(x)\big) +\p_{x_i}\big(X^{t,m}_s(x)\big) \cdot \frac{d}{ds}\p_{x_i}\big(Z^{t,m}_s(x)\big) \bigg)ds\\
=&\ \int_t^T \Bigg(\p_{x_i}\big(Z^{t,m}_s(x)\big) \cdot \bigg({\color{green}\big(\p_x f\big)^{t,m}_s(x)\cdot\p_{x_i} \big(X^{t,m}_s(x)\big)}+ \big(\p_\alpha f\big)^{t,m}_s(x)\cdot \Big({\color{green}\big(\p_x  \alpha\big)^{t,m}_s(x)\cdot\p_{x_i}\big(X^{t,m}_s(x)\big)}+\big(\p_z \alpha\big)^{t,m}_s(x)\cdot\p_{x_i}\big(Z^{t,m}_s(x)\big)\Big)\bigg)\\
&\ \ \ \ \ \ \ \ {\color{green}-\p_{x_i}\big(X^{t,m}_s(x)\big)^\top\big(\p_x f\big)^{t,m}_s(x)\cdot \p_{x_i}\big(Z^{t,m}_s(x)\big)}-\p_{x_i}\big(X^{t,m}_s(x)\big)^\top\bigg(\big(\p_x\p_x f\big)^{t,m}_s(x)\cdot Z^{t,m}_s(x)+\big(\p_x\p_x g\big)^{t,m}_s(x)\bigg)\p_{x_i}\big(X^{t,m}_s(x)\big)\\
&\ \ \ \ \ \ \ \ -\Big(\big(\p_x  \alpha\big)^{t,m}_s(x)\cdot\p_{x_i}\big(X^{t,m}_s(x)\big){\color{green}+\big(\p_z \alpha\big)^{t,m}_s(x)\cdot\p_{x_i}\big(Z^{t,m}_s(x)\big)}\Big)^\top\bigg(\big(\p_\alpha\p_x f\big)^{t,m}_s(x)\cdot Z^{t,m}_s(x)+\big(\p_\alpha\p_x g\big)^{t,m}_s(x)\bigg)\p_{x_i}\big(X^{t,m}_s(x)\big)\Bigg)ds\\
&\text{(by using \eqref{calculus2}, \eqref{eq_8_6_new} and we can cancel the terms in green by \eqref{eq_7_17_1})}\\
\leq&\ \int_t^T \Bigg(-\frac{\lambda_f^2}{\Lambda_h}\Big|\p_{x_i}\big(Z^{t,m}_s(x)\big) \Big| ^2-\left(\lambda_g-\frac12k_0\wb{l}_f-\frac{20\wb{l}_h^2}{19\lambda_g}\right)\Big|\p_{x_i}\big(X^{t,m}_s(x)\big)\Big|^2\Bigg)ds\\
&\text{(by using \eqref{bdd_d2_f}, \eqref{positive_g_x}, \eqref{bdd_d2_g_2}, \eqref{c_a7}, \eqref{eq_8_8_cone} and \eqref{p_xalpha_new}),}
\end{align*}\normalsize
which implies, in light of \eqref{eq_8_6_new} for the terminal condition on the left hand side, by using \eqref{positive_k}, we deduce
\footnotesize\begin{align}\no
\lambda_k\Big|\p_{x_i}\big(X^{t,m}_T(x)\big)\Big|^2-\p_{x_i}\big(Z^{t,m}_t(x)\big)_i
\leq&\ \int_t^T \Bigg(-\frac{\lambda_f^2}{\Lambda_h}\Big|\p_{x_i}\big(Z^{t,m}_s(x)\big) \Big| ^2-\bigg(\lambda_g-\frac12k_0\wb{l}_f-\frac{20\wb{l}_h^2}{19\lambda_g}\bigg)\Big|\p_{x_i}\big(X^{t,m}_s(x)\big)\Big|^2\Bigg)ds,\\\label{cru_est_1}
\leq&\ \int_t^T \Bigg(-\lambda_z\Big|\p_{x_i}\big(Z^{t,m}_s(x)\big) \Big| ^2-\lambda_x\Big|\p_{x_i}\big(X^{t,m}_s(x)\big)\Big|^2\Bigg)ds,
\end{align}\normalsize
where
\begin{align}\label{lambda_z}
\lambda_z:=&\frac{\lambda_f^2}{\Lambda_g+\lambda_g/20}\leq \frac{\lambda_f^2}{\Lambda_h};\\\label{lambda_x}
\lambda_x:=&\frac{4}{5}\lambda_g<\lambda_g-\frac{1}{20}\lambda_g-\frac{1}{20}\lambda_g\leq \lambda_g-\frac12k_0\wb{l}_f-\frac{20\wb{l}_h^2}{19\lambda_g}.
\end{align}
Thus, by using \eqref{p_x_gamma} and \eqref{cru_est_1}, we can obtain
\begin{align}\no
&\bigg|\p_{x_i}\gamma(t,x,m)\bigg|^2=\bigg|\p_{x_i}\big(Z^{t,m}_t(x)\big)\bigg|^2\\\no
\leq &\ \frac{\Lambda_k^2}{\lambda_k}\lambda_k\Big|\p_{x_i}\big(X^{t,m}_T(x)\big)\Big|^2+\frac{1}{\lambda_z}\bigg(\Lambda_h+\frac{20\wb{l}_h ^2}{19\lambda_g}\bigg)\lambda_z\int_t^T \Big|\p_{x_i}\big(X^{t,m}_s(x)\big)\Big|^2ds\\\no
&+\frac{1}{\lambda_x}\bigg(\frac{40\Lambda_f}{19\lambda_g}\wb{l}_h+2\Lambda_f+\Lambda_h+ \frac{20\wb{l}_h ^2}{19\lambda_g}\bigg)\lambda_x\int_t^T\Big|\p_{x_i}\big(Z^{t,m}_s(x)\big)\Big|^2ds\\\no
\leq &\ \max\bigg\{\frac{\Lambda_k^2}{\lambda_k},\frac{1}{\lambda_z}\bigg(\Lambda_h+\frac{20\wb{l}_h ^2}{19\lambda_g}\bigg),\frac{1}{\lambda_x}\bigg(\frac{40\Lambda_f}{19\lambda_g}\wb{l}_h+2\Lambda_f+\Lambda_h+ \frac{20\wb{l}_h ^2}{19\lambda_g}\bigg)\bigg\}\Big|\p_{x_i}\big(Z^{t,m}_t(x)\big)\Big|,
\end{align}\normalsize
which implies the following {\it a priori} estimate, for any $i=1,2,3,...,d_x$, by dividing both sides of above by $\Big|\p_{x_i}\big(Z^{t,m}_t(x)\big)\Big|$,
\begin{align}\no
&\bigg|\p_{x_i}\gamma(t,x,m)\bigg|=\bigg|\p_{x_i}\big(Z^{t,m}_t(x)\big)\bigg|\\\no
\leq &\ \max\bigg\{\frac{\Lambda_k^2}{\lambda_k},\frac{1}{\lambda_z}\bigg(\Lambda_h+\frac{20\wb{l}_h ^2}{19\lambda_g}\bigg),\frac{1}{\lambda_x}\bigg(\frac{40\Lambda_f}{19\lambda_g}\wb{l}_h+2\Lambda_f+\Lambda_h+ \frac{20\wb{l}_h ^2}{19\lambda_g}\bigg)\bigg\}\\\label{L_star_1}
\leq &\ \max\bigg\{\frac{\Lambda_k^2}{\lambda_k},\frac{1}{\lambda_z}\bigg(\Lambda_h+\frac{1}{20}\lambda_g\bigg),\frac{1}{\lambda_x}\bigg(\frac{5}{2}\Lambda_f+\Lambda_h+\frac{1}{20}\lambda_g\bigg)\bigg\}=:L^*_1,
\end{align}\normalsize
since $\wb{l}_fL^*\leq \frac{1}{20} \lambda_g$, $\wb{l}_h<\frac{1}{5}\lambda_g$ and $\frac{20\wb{l}_h}{19\lambda_g}<\frac{1}{4}$. 
Therefore, we finally have
\begin{align}\label{L_star_1_1}
&\lambda_k\Big|\p_{x_i}\big(X^{t,m}_T(x)\big)\Big|^2+\int_t^T \Bigg(\lambda_z\Big|\p_{x_i}\big(Z^{t,m}_s(x)\big) \Big| ^2+\lambda_x\Big|\p_{x_i}\big(X^{t,m}_s(x)\big)\Big|^2\Bigg)ds\leq L^*_1.
\end{align}\normalsize

\textbf{Step 4} (Major estimate 2 of the following \eqref {cru_est_2}) For any $i=1,2,3,...,d_x$ and any $y\in\R^{d_x}$, integrating the inner product $\p_{m_i}\big(Z^{t,m}_s(x)\big)(y)\cdot \p_{m_i}\big(X^{t,m}_s(x)\big)(y)$ with respect to $s$ over $[t,T]$ and {\color{blue}using \eqref{eq_8_5_new}} yields\footnote{We have to pay extra attention at the orange terms that could be canceled out in our previous paper \cite{bensoussan2023theory} on mean field type control problems, since it is not possible to cancel them out in the current setting of mean field games.}
\footnotesize\begin{align*}
&\p_{m_i}\big(Z^{t,m}_T(x)\big)(y)\cdot \p_{m_i}\big(X^{t,m}_T(x)\big)(y)-\p_{m_i}\big(Z^{t,m}_t(x)\big)(y)\cdot \p_{m_i}\big(X^{t,m}_t(x)\big)(y)\\
=&\ \int_t^T \bigg(\p_{m_i}\big(Z^{t,m}_s(x)\big)(y)\cdot \frac{d}{ds}\p_{m_i}\big(X^{t,m}_s(x)\big)(y)+\p_{m_i}\big(X^{t,m}_s(x)\big)(y)\cdot \frac{d}{ds}\p_{m_i}\big(Z^{t,m}_s(x)\big)(y)\bigg)ds\\
=&\ \int_t^T \Bigg(\p_{m_i}\big(Z^{t,m}_s(x)\big)(y)\cdot \bigg({\color{green}\big(\p_x f\big)^{t,m}_s(x) \p_{m_i}\big(X^{t,m}_s(x)\big)(y)}+\big(\p_\mu f\big)^{t,m}_s(x,y)  \p_{y_i}\big(X^{t,m}_s(y)\big)\\
&\ \ \ \ \ \ \  \ \ \ \ \ \ \  \ \ \ \ \ \ \ \ \ \ \ \ \ \ \ \ \ \ \ \ \ {\color{orange}+\int_{\R^{d_x}} \big(\p_\mu f\big)^{t,m}_s(x,\wh{x}) \p_{m_i}\big(X^{t,m}_s(\wh{x})\big)(y) dm(\wh{x})}\\
&\ \ \ \ \ \ \  \ \ \ \ \ \ \  \ \ \ \ \ \ \ \ \ \ \ \ \ \ \ \ \ \ \ \ \ + \big(\p_\alpha f\big)^{t,m}_s(x)\cdot \Big({\color{green}\big(\p_x \alpha\big)^{t,m}_s(x)\cdot\p_{m_i}\big(X^{t,m}_s(x)\big)(y)}+\big(\p_\mu \alpha\big)^{t,m}_s(x,y)\cdot \p_{y_i}\big(X^{t,m}_s(y)\big)\\
&\ \ \ \ \ \ \  \ \ \ \ \ \ \  \ \ \ \ \ \ \ \ \ \ \ \ \ \ \ \ \ \ \ \ \ \ \ \ \ \ \ \ \ \ \ \ \ \ \ \ \ \ \ \ \ \ \ {\color{orange}+\int_{\R^{d_x}}\big(\p_\mu \alpha\big)^{t,m}_s(x,\wh{x})\cdot \p_{m_i}\big(X^{t,m}_s(\wh{x})\big)(y)dm(\wh{x})+\big(\p_z \alpha\big)^{t,m}_s(x)\cdot\p_{m_i}\big(Z^{t,m}_s(x)\big)(y)}\Big)\bigg)\\
&\ \ \ \ \ \ \ \ {\color{green}-\p_{m_i}\big(X^{t,m}_s(x)\big)(y)^\top\big(\p_x f\big)^{t,m}_s(x)\cdot \p_{m_i}\big(Z^{t,m}_s(x)\big)(y)}\\
&\ \ \ \ \ \ \ \ {\color{orange}-\p_{m_i}\big(X^{t,m}_s(x)\big)(y)^\top\bigg(\big(\p_x\p_x f\big)^{t,m}_s(x)\cdot Z^{t,m}_s(x)+\big(\p_x\p_x g\big)^{t,m}_s(x)\bigg)\p_{m_i}\big(X^{t,m}_s(x)\big)(y)}\\
&\ \ \ \ \ \ \ \ {\color{orange}-\int_{\R^{d_x}}\p_{m_i}\big(X^{t,m}_s(\wt{x})\big)(y)^\top\bigg(\big(\p_\mu\p_x f\big)^{t,m}_s(x,\wt{x})\cdot Z^{t,m}_s(x)+\big(\p_\mu\p_x g\big)^{t,m}_s(x,\wt{x})\bigg)dm(\wt{x})\p_{m_i}\big(X^{t,m}_s(x)\big)(y)}\\
&\ \ \ \ \ \ \ \ -\p_{y_i}\big(X^{t,m}_s(y)\big)^\top\bigg(\big(\p_\mu\p_x f\big)^{t,m}_s(x,y)\cdot Z^{t,m}_s(x)+\big(\p_\mu\p_x g\big)^{t,m}_s(x,y)\bigg)\p_{m_i}\big(X^{t,m}_s(x)\big)(y)\\
&\ \ \ \ \ \ \ \ -\bigg(\big(\p_x  \alpha\big)^{t,m}_s(x)\cdot\p_{m_i}\big(X^{t,m}_s(x)\big)(y)+\big(\p_\mu \alpha\big)^{t,m}_s(x,y)\cdot \p_{y_i}\big(X^{t,m}_s(y)\big)\\
&\ \ \ \ \ \ \ \ \ \ \ \ \ \  \ \ \ \ \ \ \ +\int_{\R^{d_x}}\big(\p_\mu \alpha\big)^{t,m}_s(x,\wh{x})\cdot \p_{m_i}\big(X^{t,m}_s(\wh{x})\big)(y)dm(\wh{x}){\color{green}+\big(\p_z \alpha\big)^{t,m}_s(x)\cdot\p_{m_i}\big(Z^{t,m}_s(x)\big)(y)}\bigg)\\
&\ \ \ \ \ \ \ \ \ \ \ \ \cdot\bigg(\big(\p_\alpha\p_x f\big)^{t,m}_s(x)\cdot Z^{t,m}_s(x)+\big(\p_\alpha\p_x g\big)^{t,m}_s(x)\bigg)\p_{m_i}\big(X^{t,m}_s(x)\big)(y)\Bigg)ds\\
&\text{(by using \eqref{calculus2}, \eqref{eq_7_17_1} and we can cancel the terms in green)}
\\
\leq &\ \int_t^T \Bigg(\Lambda_f\Big(1+\frac{20\wb{l}_h}{19\lambda_g}\Big)\Big|\p_{m_i}\big(Z^{t,m}_s(x)\big)(y)\Big|\Big| \p_{y_i}\big(X^{t,m}_s(y)\big)\Big|{\color{orange}+\p_{m_i}\big(Z^{t,m}_s(x)\big)(y)^\top\int_{\R^{d_x}} \big(\p_\mu f\big)^{t,m}_s(x,\wh{x}) \p_{m_i}\big(X^{t,m}_s(\wh{x})\big)(y) dm(\wh{x})}\\
&\ \ \ \ \ \ \ \ \ {\color{orange}+\p_{m_i}\big(Z^{t,m}_s(x)\big)(y)^\top\big(\p_\alpha f\big)^{t,m}_s(x)\int_{\R^{d_x}}\big(\p_\mu \alpha\big)^{t,m}_s(x,\wh{x})\p_{m_i}\big(X^{t,m}_s(\wh{x})\big)(y)dm(\wh{x})}\\
&\ \ \ \ \ \ \ \ \ {\color{orange}+\p_{m_i}\big(Z^{t,m}_s(x)\big)(y)^\top\big(\p_\alpha f\big)^{t,m}_s(x)\big(\p_z \alpha\big)^{t,m}_s(x)\p_{m_i}\big(Z^{t,m}_s(x)\big)(y)}\\
&\ \ \ \ \ \ \ \ {\color{orange}-\p_{m_i}\big(X^{t,m}_s(x)\big)(y)^\top\bigg(\big(\p_x\p_x f\big)^{t,m}_s(x)\cdot Z^{t,m}_s(x)+\big(\p_x\p_x g\big)^{t,m}_s(x)\bigg)\p_{m_i}\big(X^{t,m}_s(x)\big)(y)}\\
&\ \ \ \ \ \ \ \ {\color{orange}-\int_{\R^{d_x}}\p_{m_i}\big(X^{t,m}_s(\wt{x})\big)(y)^\top\bigg(\big(\p_\mu\p_x f\big)^{t,m}_s(x,\wt{x})\cdot Z^{t,m}_s(x)+\big(\p_\mu\p_x g\big)^{t,m}_s(x,\wt{x})\bigg)dm(\wt{x})\p_{m_i}\big(X^{t,m}_s(x)\big)(y)}\\
&\ \ \ \ \ \ \ \ +\frac{20\wb{l}_h^2}{19\lambda_g}\bigg(\int_{\R^{d_x}}\Big|\p_{m_i}\big(X^{t,m}_s(\wt{x})\big)(y)\Big|dm(\wt{x})+\Big|\p_{m_i}\big(X^{t,m}_s(x)\big)(y)\Big|\bigg)\cdot\Big|\p_{m_i}\big(X^{t,m}_s(x)\big)(y)\Big|\\
&\ \ \ \ \ \ \ \ +\Big(\Lambda_h+\frac{20\wb{l}_h^2}{19\lambda_g}\Big)\Big|\p_{y_i}\big(X^{t,m}_s(y)\big)\Big|\cdot\Big|\p_{m_i}\big(X^{t,m}_s(x)\big)(y)\Big|\Bigg)ds\\
&\text{(by using \eqref{bdd_d1_f}, \eqref{bdd_d2_f}, \eqref{bdd_d2_g_1}, \eqref{bdd_d2_g_2}, \eqref{eq_8_8_cone} and \eqref{p_xalpha_new}),}
\end{align*}\normalsize
which implies, by using \eqref{positive_k} and \eqref{bdd_d2_k_1} on the left hand side,
\footnotesize\begin{align}\no
&\lambda_k\Big|\p_{m_i}\big(X^{t,m}_T(x)\big)(y)\Big|^2-\Lambda_k\Big|\p_{y_i}\big(X^{t,m}_T(y)\big)\Big|\Big|\p_{m_i}\big(X^{t,m}_T(x)\big)(y)\Big|+\int_{\R^{d_x}} \p_{m_i}\big(X^{t,m}_T(x)\big)(y)^\top\big(\p_\mu\p_x k\big)^{t,m}_T(x,\wt{x})\p_{m_i}\big(X^{t,m}_T(\wt{x})\big)(y)dm(\wt{x})\\\no
\leq &\ \int_t^T \Bigg(\Lambda_f\Big(1+\frac{20\wb{l}_h}{19\lambda_g}\Big)\Big|\p_{m_i}\big(Z^{t,m}_s(x)\big)(y)\Big|\Big| \p_{y_i}\big(X^{t,m}_s(y)\big)\Big|{\color{orange}+\p_{m_i}\big(Z^{t,m}_s(x)\big)(y)^\top\int_{\R^{d_x}} \big(\p_\mu f\big)^{t,m}_s(x,\wh{x}) \p_{m_i}\big(X^{t,m}_s(\wh{x})\big)(y) dm(\wh{x})}\\\no
&\ \ \ \ \ \ \ \ \ {\color{orange}+\p_{m_i}\big(Z^{t,m}_s(x)\big)(y)^\top\big(\p_\alpha f\big)^{t,m}_s(x)\int_{\R^{d_x}}\big(\p_\mu \alpha\big)^{t,m}_s(x,\wh{x})\p_{m_i}\big(X^{t,m}_s(\wh{x})\big)(y)dm(\wh{x})}\\\no
&\ \ \ \ \ \ \ \ \ {\color{orange}+\p_{m_i}\big(Z^{t,m}_s(x)\big)(y)^\top\big(\p_\alpha f\big)^{t,m}_s(x)\big(\p_z \alpha\big)^{t,m}_s(x)\p_{m_i}\big(Z^{t,m}_s(x)\big)(y)}\\\no
&\ \ \ \ \ \ \ \ {\color{orange}-\p_{m_i}\big(X^{t,m}_s(x)\big)(y)^\top\bigg(\big(\p_x\p_x f\big)^{t,m}_s(x)\cdot Z^{t,m}_s(x)+\big(\p_x\p_x g\big)^{t,m}_s(x)\bigg)\p_{m_i}\big(X^{t,m}_s(x)\big)(y)}\\\no
&\ \ \ \ \ \ \ \ {\color{orange}-\int_{\R^{d_x}}\p_{m_i}\big(X^{t,m}_s(\wt{x})\big)(y)^\top\bigg(\big(\p_\mu\p_x f\big)^{t,m}_s(x,\wt{x})\cdot Z^{t,m}_s(x)+\big(\p_\mu\p_x g\big)^{t,m}_s(x,\wt{x})\bigg)dm(\wt{x})\p_{m_i}\big(X^{t,m}_s(x)\big)(y)}\\\no
&\ \ \ \ \ \ \ \ +\frac{20\wb{l}_h^2}{19\lambda_g}\bigg(\int_{\R^{d_x}}\Big|\p_{m_i}\big(X^{t,m}_s(\wt{x})\big)(y)\Big|dm(\wt{x})+\Big|\p_{m_i}\big(X^{t,m}_s(x)\big)(y)\Big|\bigg)\cdot\Big|\p_{m_i}\big(X^{t,m}_s(x)\big)(y)\Big|\\\label{eq_7_26_1}
&\ \ \ \ \ \ \ \ +\Big(\Lambda_h+\frac{20\wb{l}_h^2}{19\lambda_g}\Big)\Big|\p_{y_i}\big(X^{t,m}_s(y)\big)\Big|\cdot\Big|\p_{m_i}\big(X^{t,m}_s(x)\big)(y)\Big|\Bigg)ds.
\end{align}\normalsize
Then, integrating \eqref{eq_7_26_1} with respect to $x$, we have
\footnotesize\begin{align*}
&\lambda_k\Big\|\p_{m_i}\big(X^{t,m}_T(\cdot)\big)(y)\Big\|_{L^{2,d_x}_m}^2-\Lambda_k\Big|\p_{y_i}\big(X^{t,m}_T(y)\big)\Big|\Big\|\p_{m_i}\big(X^{t,m}_T(\cdot)\big)(y)\Big\|_{L^{1,d_x}_m}\\\no
&+\int_{\R^{d_x}}\int_{\R^{d_x}} \p_{m_i}\big(X^{t,m}_T(x)\big)(y)^\top\big(\p_\mu\p_x k\big)^{t,m}_T(x,\wt{x}) \p_{m_i}\big(X^{t,m}_T(\wt{x})\big)(y)dm(\wt{x})dm(x)\\\no
\leq &\ \int_t^T \Bigg(\Lambda_f\Big(1+\frac{20\wb{l}_h}{19\lambda_g}\Big)\Big\|\p_{m_i}\big(Z^{t,m}_s(\cdot)\big)(y)\Big\|_{L^{1,d_x}_m}\Big| \p_{y_i}\big(X^{t,m}_s(y)\big)\Big|\\
&\ \ \ \ \ \ \ \ \ {\color{orange}+\int_{\R^{d_x}}\p_{m_i}\big(Z^{t,m}_s(x)\big)(y)^\top\int_{\R^{d_x}} \big(\p_\mu f\big)^{t,m}_s(x,\wh{x}) \p_{m_i}\big(X^{t,m}_s(\wh{x})\big)(y) dm(\wh{x})dm(x)}\\\no
&\ \ \ \ \ \ \ \ \ {\color{orange}+\int_{\R^{d_x}}\p_{m_i}\big(Z^{t,m}_s(x)\big)(y)^\top\big(\p_\alpha f\big)^{t,m}_s(x)\int_{\R^{d_x}}\big(\p_\mu \alpha\big)^{t,m}_s(x,\wh{x})\p_{m_i}\big(X^{t,m}_s(\wh{x})\big)(y)dm(\wh{x})dm(x)}\\\no
&\ \ \ \ \ \ \ \ \ {\color{orange}+\int_{\R^{d_x}}\p_{m_i}\big(Z^{t,m}_s(x)\big)(y)^\top\big(\p_\alpha f\big)^{t,m}_s(x)\big(\p_z \alpha\big)^{t,m}_s(x)\p_{m_i}\big(Z^{t,m}_s(x)\big)(y)dm(x)}\\\no
&\ \ \ \ \ \ \ \ {\color{orange}-\int_{\R^{d_x}}\p_{m_i}\big(X^{t,m}_s(x)\big)(y)^\top\bigg(\big(\p_x\p_x f\big)^{t,m}_s(x)\cdot Z^{t,m}_s(x)+\big(\p_x\p_x g\big)^{t,m}_s(x)\bigg)\p_{m_i}\big(X^{t,m}_s(x)\big)(y)dm(x)}\\\no
&\ \ \ \ \ \ \ \ {\color{orange}-\int_{\R^{d_x}}\int_{\R^{d_x}}\p_{m_i}\big(X^{t,m}_s(\wt{x})\big)(y)^\top\bigg(\big(\p_\mu\p_x f\big)^{t,m}_s(x,\wt{x})\cdot Z^{t,m}_s(x)+\big(\p_\mu\p_x g\big)^{t,m}_s(x,\wt{x})\bigg)dm(\wt{x})\p_{m_i}\big(X^{t,m}_s(x)\big)(y)dm(x)}\\\no
&\ \ \ \ \ \ \ \ +\frac{20\wb{l}_h^2}{19\lambda_g}\bigg(\Big\|\p_{m_i}\big(X^{t,m}_s(\cdot)\big)(y)\Big\|_{L^{1,d_x}_m}^2+\Big\|\p_{m_i}\big(X^{t,m}_s(\cdot)\big)(y)\Big\|_{L^{2,d_x}_m}^2\bigg)\\
&\ \ \ \ \ \ \ \ +\Big(\Lambda_h+\frac{20\wb{l}_h^2}{19\lambda_g}\Big)\Big|\p_{y_i}\big(X^{t,m}_s(y)\big)\Big|\cdot\Big\|\p_{m_i}\big(X^{t,m}_s(\cdot)\big)(y)\Big\|_{L^{1,d_x}_m}\Bigg)ds,
\end{align*}\normalsize
which implies, by using \eqref{positive_g_x},  \eqref{positive_g_mu_1}, \eqref{positive_k_mu_1}, \eqref{positive_H_mu} and using Jensen's inequality to replace $L^1$-norm by $L^2$-norm, 
\footnotesize\begin{align}\no
&\Big(\lambda_k-l_k-\varepsilon_1\Big)\Big\|\p_{m_i}\big(X^{t,m}_T(\cdot)\big)(y)\Big\|_{L^{2,d_x}_m}^2-\frac{1}{4\varepsilon_1}\Lambda_k^2\Big|\p_{y_i}\big(X^{t,m}_T(y)\big)\Big|^2\\\no
\leq &\ \int_t^T \Bigg(\frac{1}{4\varepsilon_1}\bigg(\Lambda_f^2\Big(1+\frac{20\wb{l}_h}{19\lambda_g}\Big)^2+\Big(\Lambda_h+\frac{20\wb{l}_h^2}{19\lambda_g}\Big)^2\bigg)\Big| \p_{y_i}\big(X^{t,m}_s(y)\big)\Big|^2-\Big(\lambda_2-\varepsilon_1\Big)\Big\|\p_{m_i}\big(Z^{t,m}_s(\cdot)\big)(y)\Big\|_{L^{2,d_x}_m}^2\\\label{displacement}
&\ \ \ \ \ \ \ \ -\Big(\lambda_1-\frac{40\wb{l}_h^2}{19\lambda_g}-k_0\wb{l}_f-\varepsilon_1\Big)\Big\|\p_{m_i}\big(X^{t,m}_s(\cdot)\big)(y)\Big\|_{L^{2,d_x}_m}^2\Bigg)ds,
\end{align}\normalsize
where, we set 
\begin{align}\label{vareps_1}
\varepsilon_1:=\min\Bigg\{\frac{1}{2}(\lambda_k-l_k),\frac{1}{2}\lambda_2,\frac{2}{5}\lambda_1\Bigg\},
\end{align}
and thus we can define the following positive constants, under Hypothesis ${\bf(h3)}$,
\begin{align}\no
&\wb{\lambda}_k:=\frac{1}{2}(\lambda_k-l_k)=\lambda_k-l_k-\varepsilon_1;\ \wb{\lambda}_z:=\frac{1}{2}\lambda_2\leq \lambda_2-\varepsilon_1;\\\label{def_wb_lambda}
&\wb{\lambda}_x:=\frac{2}{5}\lambda_1\leq \lambda_1-\frac{1}{10}\lambda_1-\frac{1}{10}\lambda_1-\varepsilon_1\leq \lambda_1-\frac{40\wb{l}_h^2}{19\lambda_g}-k_0\wb{l}_f-\varepsilon_1.
\end{align}
Therefore, we have 
\footnotesize\begin{align}\no
&\wb{\lambda}_k\Big\|\p_{m_i}\big(X^{t,m}_T(\cdot)\big)(y)\Big\|_{L^{2,d_x}_m}^2+\int_t^T\Bigg(\wb{\lambda}_z\Big\|\p_{m_i}\big(Z^{t,m}_s(\cdot)\big)(y)\Big\|_{L^{2,d_x}_m}^2+\wb{\lambda}_x\Big\|\p_{m_i}\big(X^{t,m}_s(\cdot)\big)(y)\Big\|_{L^{2,d_x}_m}^2\Bigg)ds\\\label{cru_est_2}
\leq &\ \frac{1}{4\varepsilon_1}\bigg(\Lambda_f^2\Big(1+\frac{20\wb{l}_h}{19\lambda_g}\Big)^2+\Big(\Lambda_h+\frac{20\wb{l}_h^2}{19\lambda_g}\Big)^2\bigg)\int_t^T \Big| \p_{y_i}\big(X^{t,m}_s(y)\big)\Big|^2ds+\frac{1}{4\varepsilon_1}\Lambda_k^2\Big|\p_{y_i}\big(X^{t,m}_T(y)\big)\Big|^2.
\end{align}\normalsize 

\textbf{Step 5} (Consequences of major estimates obtained in Steps 3 and 4)
Thus, by using \eqref{p_m_gamma_int}, \eqref{L_star_1_1} and \eqref{cru_est_2}, we can obtain the following {\it a priori} estimate, for any $i=1,2,3,...,d_x$, 
\small\begin{align}\no
&\bigg\|\p_{m_i}\gamma(t,\cdot,m)(y)\bigg\|_{L^{2,d_x}_m}^2=\bigg\|\p_{m_i}\big(Z^{t,m}_t(\cdot)\big)(y)\bigg\|_{L^{2,d_x}_m}^2\\\no
\leq &\ \frac{6\Lambda_k^2}{\wb{\lambda}_k}\wb{\lambda}_k\Big\|\p_{m_i}\big(X^{t,m}_T(\cdot)\big)(y)\Big\|_{L^{2,d_x}_m}^2+3\Lambda_k^2\Big|\p_{y_i}\big(X^{t,m}_T(y)\big)\Big|^2\\\no
&+\int_t^T\Bigg(\Big(2\Lambda_h+\frac{40\wb{l}_h^2}{19\lambda_g}\Big)\frac{1}{\wb{\lambda}_x}\wb{\lambda}_x\Big\|\p_{m_i}\big(X^{t,m}_s(\cdot)\big)(y)\Big\|_{L^{2,d_x}_m}^2+\Big(\Lambda_h+\frac{20\wb{l}_h^2}{19\lambda_g}\Big)\Big|\p_{y_i}\big(X^{t,m}_s(y)\big)\Big|^2\\\no
&\ \ \ \ \ \ \ \ \ \ \ \ +\bigg(\Lambda_f+\frac{20\Lambda_f}{19\lambda_g}\wb{l}_h+3\Lambda_h+\frac{60\wb{l}_h^2}{19\lambda_g}\bigg)\frac{1}{\wb{\lambda}_z}\wb{\lambda}_z\Big\|\p_{m_i}\big(Z^{t,m}_s(\cdot)\big)(y)\Big\|_{L^{2,d_x}_m}^2\Bigg)ds\\\no
\leq &\  \Big(3+\frac{L^*_2}{4\varepsilon_1}\Big)\Lambda_k^2\frac{1}{\lambda_k}\lambda_k\Big|\p_{y_i}\big(X^{t,m}_T(y)\big)\Big|^2\\\no
&+\int_t^T \bigg(\Lambda_h+\frac{20\wb{l}_h^2}{19\lambda_g}+\frac{L^*_2}{4\varepsilon_1}\Big(\Lambda_f^2\Big(1+\frac{20\wb{l}_h}{19\lambda_g}\Big)^2+\Big(\Lambda_h+\frac{20\wb{l}_h^2}{19\lambda_g}\Big)^2\Big)\bigg)\frac{1}{\lambda_x}\lambda_x\Big|\p_{y_i}\big(X^{t,m}_s(y)\big)\Big|^2 ds\\\label{L_star_3_1}
\leq &\ L^*_1L^*_3,
\end{align}\normalsize
where $L^*_1$ was defined in \eqref{L_star_1} and 
\footnotesize\begin{align}\label{L_star_2}
L^*_2:=\max\Bigg\{&\frac{6\Lambda_k^2}{\wb{\lambda}_k},\Big(2\Lambda_g+\frac{1}{10}\lambda_g\Big)\frac{1}{\wb{\lambda}_x},\bigg(\frac54\Lambda_f+3\Lambda_g+\frac{3}{10}\lambda_g\bigg)\frac{1}{\wb{\lambda}_z}\Bigg\}\\\no
\geq \max\Bigg\{&\frac{6\Lambda_k^2}{\wb{\lambda}_k},\Big(2\Lambda_h+\frac{40\wb{l}_h^2}{19\lambda_g}\Big)\frac{1}{\wb{\lambda}_x},\bigg(\Lambda_f+\frac{20\Lambda_f}{19\lambda_g}\wb{l}_h+3\Lambda_h+\frac{60\wb{l}_h^2}{19\lambda_g}\bigg)\frac{1}{\wb{\lambda}_z}\Bigg\},
\end{align}\normalsize
and\footnotesize\begin{align}\label{L_star_3}
L^*_3:=\max\Bigg\{&\Big(3+\frac{L^*_2}{4\varepsilon_1}\Big)\Lambda_k^2\frac{1}{\lambda_k},\bigg(\Lambda_g+\frac{1}{10}\lambda_g+\frac{L^*_2}{4\varepsilon_1}\Big(\frac{25}{16}\Lambda_f^2+\Big(\Lambda_g+\frac{1}{10}\lambda_g\Big)^2\Big)\bigg)\frac{1}{\lambda_x}\Bigg\}\\\no
\geq \max\Bigg\{&\Big(3+\frac{L^*_2}{4\varepsilon_1}\Big)\Lambda_k^2\frac{1}{\lambda_k},\bigg(\Lambda_h+\frac{20\wb{l}_h^2}{19\lambda_g}+\frac{L^*_2}{4\varepsilon_1}\Big(\Lambda_f^2\Big(1+\frac{20\wb{l}_h}{19\lambda_g}\Big)^2+\Big(\Lambda_h+\frac{20\wb{l}_h^2}{19\lambda_g}\Big)^2\Big)\bigg)\frac{1}{\lambda_x}\Bigg\}.
\end{align}\normalsize
In addition, using \eqref{L_star_1_1} and \eqref{cru_est_2}, we have
\footnotesize\begin{align}\no
&\wb{\lambda_k}\Big\|\p_{m_i}\big(X^{t,m}_T(\cdot)\big)(y)\Big\|_{L^{2,d_x}_m}^2+\wb{\lambda}_z\int_t^T\Big\|\p_{m_i}\big(Z^{t,m}_s(\cdot)\big)(y)\Big\|_{L^{2,d_x}_m}^2ds+\wb{\lambda}_x\int_t^T\Big\|\p_{m_i}\big(X^{t,m}_s(\cdot)\big)(y)\Big\|_{L^{2,d_x}_m}^2ds\\\no
\leq &\ \frac{1}{4\varepsilon_1}\Lambda_k^2\frac{1}{\lambda_k}\lambda_k\Big|\p_{y_i}\big(X^{t,m}_T(y)\big)\Big|^2+\frac{1}{4\varepsilon_1}\bigg(\Lambda_f^2\Big(1+\frac{20\wb{l}_h}{19\lambda_g}\Big)^2+\Big(\Lambda_h+\frac{20\wb{l}_h^2}{19\lambda_g}\Big)^2\bigg)\frac{1}{\lambda_x}\lambda_x\int_t^T \Big| \p_{y_i}\big(X^{t,m}_s(y)\big)\Big|^2ds\\\label{L_star_4_1}
\leq &\ L^*_1L^*_4,
\end{align}\normalsize
where 
\footnotesize\begin{align}\label{L_star_4}
L^*_4:=\max\Bigg\{&\frac{1}{4\varepsilon_1}\Lambda_k^2\frac{1}{\lambda_k},\frac{1}{4\varepsilon_1}\bigg(\frac{25}{16}\Lambda_f^2+\Big(\Lambda_g+\frac{1}{10}\lambda_g\Big)^2\bigg)\frac{1}{\lambda_x}\Bigg\}\\\no
\geq \max\Bigg\{&\frac{1}{4\varepsilon_1}\Lambda_k^2\frac{1}{\lambda_k},\frac{1}{4\varepsilon_1}\bigg(\Lambda_f^2\Big(1+\frac{20\wb{l}_h}{19\lambda_g}\Big)^2+\Big(\Lambda_h+\frac{20\wb{l}_h^2}{19\lambda_g}\Big)^2\bigg)\frac{1}{\lambda_x}\Bigg\}.
\end{align}\normalsize

Now, we go back to \eqref{eq_7_26_1}. Using \eqref{bdd_d1_f}, \eqref{bdd_d2_f}, \eqref{bdd_d2_g_1}, \eqref{bdd_d2_g_2}, \eqref{bdd_d2_k_1}, \eqref{p_xalpha_new} and \eqref{p_zalpha_new}, we obtain
\footnotesize\begin{align}\no
&\lambda_k\Big|\p_{m_i}\big(X^{t,m}_T(x)\big)(y)\Big|^2-\Lambda_k\int_{\R^{d_x}} \Big|\p_{m_i}\big(X^{t,m}_T(\wt{x})\big)(y)\Big|dm(\wt{x})\Big|\p_{m_i}\big(X^{t,m}_T(x)\big)(y)\Big|-\Lambda_k\Big|\p_{y_i}\big(X^{t,m}_T(y)\big)\Big|\Big|\p_{m_i}\big(X^{t,m}_T(x)\big)(y)\Big|\\\no
\leq &\ \int_t^T \Bigg(\Lambda_f\Big(1+\frac{20\wb{l}_h}{19\lambda_g}\Big)\cdot\bigg(\Big|\p_{m_i}\big(Z^{t,m}_s(x)\big)(y)\Big|\Big| \p_{y_i}\big(X^{t,m}_s(y)\big)\Big|+\Big|\p_{m_i}\big(Z^{t,m}_s(x)\big)(y)\Big|\int_{\R^{d_x}} \Big|\p_{m_i}\big(X^{t,m}_s(\wh{x})\big)(y)\Big| dm(\wh{x})\bigg)\\\no
&\ \ \ \ \ \ \ \ -\lambda_z\Big|\p_{m_i}\big(Z^{t,m}_s(x)\big)(y)\Big|^2-\lambda_x\Big|\p_{m_i}\big(X^{t,m}_s(x)\big)(y)\Big|^2\\\no
&\ \ \ \ \ \ \ \ +\frac{20\wb{l}_h^2}{19\lambda_g}\bigg(\int_{\R^{d_x}}\Big|\p_{m_i}\big(X^{t,m}_s(\wt{x})\big)(y)\Big|dm(\wt{x})+\Big|\p_{m_i}\big(X^{t,m}_s(x)\big)(y)\Big|\bigg)\cdot\Big|\p_{m_i}\big(X^{t,m}_s(x)\big)(y)\Big|\\\label{eq_8_40_new}
&\ \ \ \ \ \ \ \ +\Big(\Lambda_h+\frac{20\wb{l}_h^2}{19\lambda_g}\Big)\Big|\p_{y_i}\big(X^{t,m}_s(y)\big)\Big|\Big|\p_{m_i}\big(X^{t,m}_s(x)\big)(y)\Big|\Bigg)ds.
\end{align}\normalsize 
Applying Young's inequality to \eqref{eq_8_40_new} yields 
\footnotesize\begin{align}\no
&\big(\lambda_k-2\varepsilon_2\big)\Big|\p_{m_i}\big(X^{t,m}_T(x)\big)(y)\Big|^2-\frac{1}{4\varepsilon_2}\Lambda_k^2\Big\|\p_{m_i}\big(X^{t,m}_T(\cdot)\big)(y)\Big\|_{L^{1,d_x}_m}^2-\frac{1}{4\varepsilon_2}\Lambda_k^2\Big|\p_{y_i}\big(X^{t,m}_T(y)\big)\Big|^2\\\no
\leq &\ \int_t^T \Bigg(\frac{1}{4\varepsilon_2}\Lambda_f^2\Big(1+\frac{20\wb{l}_h}{19\lambda_g}\Big)^2\cdot\bigg(\Big| \p_{y_i}\big(X^{t,m}_s(y)\big)\Big|^2+ \Big\|\p_{m_i}\big(X^{t,m}_s(\cdot)\big)(y)\Big\|_{L^{1,d_x}_m}^2\bigg)\\\no
&\ \ \ \ \ \ \ \ -\big(\lambda_z-2\varepsilon_2\big)\Big|\p_{m_i}\big(Z^{t,m}_s(x)\big)(y)\Big|^2-\Big(\lambda_x-2\varepsilon_2-\frac{20\wb{l}_h^2}{19\lambda_g}\Big)\Big|\p_{m_i}\big(X^{t,m}_s(x)\big)(y)\Big|^2\\\no
&\ \ \ \ \ \ \ \ +\frac{1}{4\varepsilon_2} \bigg(\frac{20\wb{l}_h^2}{19\lambda_g}\bigg)^2 \Big\|\p_{m_i}\big(X^{t,m}_s(\cdot)\big)(y)\Big\|_{L^{1,d_x}_m}^2+\frac{1}{4\varepsilon_2}\Big(\Lambda_h+\frac{20\wb{l}_h^2}{19\lambda_g}\Big)^2\Big|\p_{y_i}\big(X^{t,m}_s(y)\big)\Big|^2\Bigg)ds,
\end{align}\normalsize
which implies, by using \eqref{L_star_1_1} and \eqref{L_star_4_1},
\footnotesize\begin{align}\no
&\big(\lambda_k-2\varepsilon_2\big)\Big|\p_{m_i}\big(X^{t,m}_T(x)\big)(y)\Big|^2+\big(\lambda_z-2\varepsilon_2\big)\int_t^T\Big|\p_{m_i}\big(Z^{t,m}_s(x)\big)(y)\Big|^2ds+\Big(\lambda_x-2\varepsilon_2-\frac{20\wb{l}_h^2}{19\lambda_g}\Big)\int_t^T\Big|\p_{m_i}\big(X^{t,m}_s(x)\big)(y)\Big|^2ds\\\no
\leq &\ \frac{1}{4\varepsilon_2}\Lambda_k^2\frac{1}{\wb{\lambda}_k}\wb{\lambda}_k\Big\|\p_{m_i}\big(X^{t,m}_T(\cdot)\big)(y)\Big\|_{L^{1,d_x}_m}^2+\frac{1}{4\varepsilon_2}\Lambda_k^2\frac{1}{\lambda_k}\lambda_k\Big|\p_{y_i}\big(X^{t,m}_T(y)\big)\Big|^2\\\no
&+\int_t^T \Bigg(\frac{1}{4\varepsilon_2}\Lambda_f^2\Big(1+\frac{20\wb{l}_h}{19\lambda_g}\Big)^2\cdot\bigg(\frac{1}{\lambda_x}\lambda_x\Big| \p_{y_i}\big(X^{t,m}_s(y)\big)\Big|^2+ \frac{1}{\wb{\lambda}_x}\wb{\lambda}_x\Big\|\p_{m_i}\big(X^{t,m}_s(\cdot)\big)(y)\Big\|_{L^{1,d_x}_m}^2\bigg)\\\no
&\ \ \ \ \ \ \ \ \ \ +\frac{1}{4\varepsilon_2}\bigg(\frac{20\wb{l}_h^2}{19\lambda_g}\bigg)^2 \frac{1}{\wb{\lambda}_x}\wb{\lambda}_x\Big\|\p_{m_i}\big(X^{t,m}_s(\cdot)\big)(y)\Big\|_{L^{1,d_x}_m}^2+\frac{1}{4\varepsilon_2}\Big(\Lambda_h+\frac{20\wb{l}_h^2}{19\lambda_g}\Big)^2\frac{1}{\lambda_x}\lambda_x\Big|\p_{y_i}\big(X^{t,m}_s(y)\big)\Big|^2\Bigg)ds\\\no
\leq&\  L^*_1L^*_4\big(1+L^*_5\big),
\end{align}\normalsize
where
\footnotesize\begin{align}\label{L_star_5}
L^*_5:=\max\Bigg\{&\frac{1}{4\varepsilon_2}\Lambda_k^2\frac{1}{\wb{\lambda}_k},\frac{1}{4\varepsilon_2\wb{\lambda}_x}\bigg(\frac{25}{16}\Lambda_f^2+\frac{1}{400}\lambda_g^2\bigg)\Bigg\}\\\no
\geq \max\Bigg\{&\frac{1}{4\varepsilon_2}\Lambda_k^2\frac{1}{\wb{\lambda}_k},\frac{1}{4\varepsilon_2\wb{\lambda}_x}\bigg(\Lambda_f^2\Big(1+\frac{20\wb{l}_h}{19\lambda_g}\Big)^2+\bigg(\frac{20\wb{l}_h^2}{19\lambda_g}\bigg)^2\bigg)\Bigg\}.
\end{align}\normalsize
Here, we choose
\begin{align}\label{vareps_2}
\varepsilon_2:=\min\Big\{\frac{1}{4}\lambda_k,\frac{1}{4}\lambda_z,\frac{1}{8}\lambda_g\Big\},
\end{align}
and thus, $\lambda_k-2\varepsilon_2= \frac{1}{2}\lambda_k >0$, $\lambda_z-2\varepsilon_2=\frac12 \lambda_z>0$ and $\lambda_x-2\varepsilon_2-\frac{20\wb{l}_h^2}{19\lambda_g}\geq \frac{1}{2}\lambda_g>0$ are all positive, and we have 
\footnotesize\begin{align}\label{L_star_5_1}
&\frac12\lambda_k\Big|\p_{m_i}\big(X^{t,m}_T(x)\big)(y)\Big|^2+\frac12 \lambda_z \int_t^T\Big|\p_{m_i}\big(Z^{t,m}_s(x)\big)(y)\Big|^2ds+\frac12 \lambda_g\int_t^T\Big|\p_{m_i}\big(X^{t,m}_s(x)\big)(y)\Big|^2ds\leq L^*_1L^*_4\big(1+L^*_5\big).
\end{align}\normalsize
Thus, by using \eqref{p_m_gamma}, \eqref{L_star_1_1}, \eqref{L_star_4_1} and \eqref{L_star_5_1}, we can obtain the following {\it a priori} estimate, for $i=1,2,3,...,d_x$, 
\footnotesize\begin{align}\no
&\bigg|\p_{m_i}\gamma(t,x,m)(y)\bigg|^2=\bigg|\p_{m_i}\big(Z^{t,m}_t(x)\big)(y)\bigg|^2\\\no
\leq &\ 3\Lambda_k^2\frac{1}{\wb{\lambda}_k}\wb{\lambda}_k \Big\|\p_{m_i}\big(X^{t,m}_T(\cdot)\big)(y)\Big\|_{L^{2,d_x}_m}^2 +6\Lambda_k^2\frac{1}{ \lambda_k}\cdot\frac12 \lambda_k\Big|\p_{m_i}\big(X^{t,m}_T(x)\big)(y)\Big|^2+3\Lambda_k^2\frac{1}{\lambda_k}\lambda_k\Big|\p_{y_i}\big(X^{t,m}_T(y)\big)\Big|^2\\\no
&+\int_t^T\Bigg(\Big(\Lambda_h+\frac{20\wb{l}_h^2}{19\lambda_g}\Big) \frac{1}{\wb{\lambda}_x}\wb{\lambda}_x\Big\|\p_{m_i}\big(X^{t,m}_s(\cdot)\big)(y)\Big\|_{L^{2,d_x}_m}^2+2\Big(\Lambda_h+\frac{20\wb{l}_h^2}{19\lambda_g}\Big)\frac{1}{ \lambda_g}\cdot\frac12 \lambda_g\Big|\p_{m_i}\big(X^{t,m}_s(x)\big)(y)\Big|^2\\\no
&\ \ \ \ \ \ \ \ \ \ \ \ +\Big(\Lambda_h+\frac{20\wb{l}_h^2}{19\lambda_g}\Big)\frac{1}{\lambda_x}\lambda_x\Big|\p_{y_i}\big(X^{t,m}_s(y)\big)\Big|^2+2\bigg(\Lambda_f+\frac{20\Lambda_f}{19\lambda_g}\wb{l}_h+3\Lambda_h+\frac{60\wb{l}_h^2}{19\lambda_g}\bigg)\frac{1}{\lambda_z}\cdot\frac12 \lambda_z\Big|\p_{m_i}\big(Z^{t,m}_s(x)\big)(y)\Big|^2\Bigg)ds\\\label{L_star_6_1}
\leq &\  \Big(L^*_4\big(2+L^*_5\big)+1\Big)L^*_1L^*_6,
\end{align}\normalsize
where
\footnotesize\begin{align}\no
L^*_6:=\max\Bigg\{&6\Lambda_k^2\frac{1}{ \lambda_k},2\bigg(\frac54 \Lambda_f+3\Lambda_g+\frac{3}{10}\lambda_g\bigg)\frac{1}{\lambda_z},2\Big(\Lambda_g+\frac{1}{10}\lambda_g\Big)\frac{1}{ \lambda_g},\\\label{L_star_6}
&3\Lambda_k^2\frac{1}{\wb{\lambda}_k},\Big(\Lambda_g+\frac{1}{10}\lambda_g\Big) \frac{1}{\wb{\lambda}_x},3\Lambda_k^2\frac{1}{\lambda_k},\Big(\Lambda_g+\frac{1}{10}\lambda_g\Big)\frac{1}{\lambda_x}\Bigg\}\\\no
\geq \max\Bigg\{&6\Lambda_k^2\frac{1}{ \lambda_k},2\bigg(\Lambda_f+\frac{20\Lambda_f}{19\lambda_g}\wb{l}_h+3\Lambda_h+\frac{60\wb{l}_h^2}{19\lambda_g}\bigg)\frac{1}{\lambda_z},2\Big(\Lambda_h+\frac{20\wb{l}_h^2}{19\lambda_g}\Big)\frac{1}{ \lambda_g},\\\no
&3\Lambda_k^2\frac{1}{\wb{\lambda}_k},\Big(\Lambda_h+\frac{20\wb{l}_h^2}{19\lambda_g}\Big) \frac{1}{\wb{\lambda}_x},3\Lambda_k^2\frac{1}{\lambda_k},\Big(\Lambda_h+\frac{20\wb{l}_h^2}{19\lambda_g}\Big)\frac{1}{\lambda_x}\Bigg\}.
\end{align}\normalsize

In conclusion, by \eqref{L_star_1} and \eqref{L_star_6_1}, we can obtain,
\begin{align}\no
&\sup_{x\in\R^{d_x},m\in\mc{P}_2(\R^{d_x}),y\in\R^{d_x}}\|\p_x\gamma(t,x,m)\|_{\mathcal{L}(\R^{d_x};\R^{d_x})}\vee \|\p_m\gamma(t,x,m)(y)\|_{\mathcal{L}(\R^{d_x};\R^{d_x})}\\\label{L_star_0}
\leq&\ d_x\cdot\max\Big\{L^*_1,\sqrt{\Big(L^*_4\big(2+L^*_5\big)+1\Big)L^*_1L^*_6}\Big\}=:L^*_0.
\end{align}
Since $\gamma(s,x,m)$ is also a decoupling field for the FBODE system \eqref{fbodesystem} on $[t,T]$, for an arbitrary $s\in[t,T]$, $(X^{s,m}_\tau(x),Z^{s,m}_\tau(x)=\gamma(\tau,X^{s,m}_\tau(x),X^{s,m}_\tau\ot m))_{x\in\R^{d_x},\tau\in[s,T]}$ is also a solution pair to \eqref{fbodesystem} over the time horizon $[s,T]$, so the above estimates, for instance, \eqref{L_star_1_1}, \eqref {cru_est_2}, \eqref{L_star_3_1}, \eqref{L_star_4_1}, \eqref{L_star_5_1}, \eqref{L_star_6_1} and \eqref{L_star_0} are also valid for $\gamma(s,x,m)=\gamma(s,X^{s,m}_s(x),X^{s,m}_s\ot m)=Z^{s,m}_s(x)$; in other words, we still have
\begin{align}\no
&\sup_{x\in\R^{d_x},m\in\mc{P}_2(\R^{d_x}),y\in\R^{d_x}}\|\p_x\gamma(s,x,m)\|_{\mathcal{L}(\R^{d_x};\R^{d_x})}\vee \|\p_m\gamma(s,x,m)(y)\|_{\mathcal{L}(\R^{d_x};\R^{d_x})}\leq L^*_0.
\end{align}
Therefore, $L^*=\vertiii{\gamma }_{2,[t,T]}\leq L^*_0$.
We emphasize here that the positive constants $L^*_i,\,i=0,1,\cdot\cdot\cdot,6$, which are defined in \eqref{L_star_1}, \eqref{L_star_2}, \eqref{L_star_3}, \eqref{L_star_4}, \eqref{L_star_5}, \eqref{L_star_6} and \eqref{L_star_0}, only depend on $\lambda_f$, $\Lambda_1$, $\Lambda_2$, $\Lambda_3$, $\lambda_1$, $\lambda_2$, $\lambda_g$, $\Lambda_g$, $l_g$, $\lambda_k$, $\Lambda_k$ and $l_k$, but not $L^*$, $\wb{l}_f$, $\wb{l}_g$ and $T$. 
\end{proof}

\begin{remark}\label{remark_lower_bdd_L_star_0}
(i) It is worth noting that for our choice of $L^*_0$, we always have $L^*_0>\Lambda_k$. More precisely, it follows from \eqref{positive_k}, \eqref{bdd_d2_k_1} and \eqref{def_wb_lambda} that $\lambda_k\leq 2\Lambda_k$ and $\wb{\lambda}_k\leq \frac{1}{2}\Lambda_k$. Thus, by \eqref{L_star_1} and \eqref{L_star_6}, we have 
\begin{align}\label{lower_bdd_L_star_1}
L^*_1\geq \frac{\Lambda_k^2}{\lambda_k}\geq \frac12\Lambda_k\text{ and } L^*_6\geq \frac{3\Lambda_k^2}{\wb{\lambda}_k}\geq 6\Lambda_k,
\end{align}
and hence, 
\begin{align}\label{lower_bdd_L_star_0}
L^*_0\geq d_x\cdot\sqrt{\Big(L^*_4\big(2+L^*_5\big)+1\Big)L^*_1L^*_6}\geq \sqrt{L^*_1L^*_6}\geq \sqrt{3}\Lambda_k>\Lambda_k.
\end{align}
(ii) By Part $(i)$ of Remark \ref{remark_h2}, if we directly assume \eqref{h3_eq_1}, then $L^*_0$ only depends on $\lambda_f$, $\Lambda_1$, $\Lambda_2$, $\Lambda_3$, $\lambda_g$, $\Lambda_g$, $l_g$, $\lambda_k$, $\Lambda_k$ and $l_k$, since $\lambda_1$ and $\lambda_2$ can be chosen such that they depend only on $\lambda_f$, $\Lambda_1$, $\Lambda_2$, $\lambda_g$,  $\Lambda_g$ and $l_g$.
\end{remark}

\subsection{Global Existence}
With the crucial {\it a priori} estimate \eqref{JFE} in hand, we are ready to establish the global solution existence in time of the FBODE system \eqref{FBODE}.
\begin{theorem}\label{GlobalSol}
(Global in time solvability of the FBODE system \eqref{FBODE}). Under Assumption ${\bf(a1)}$-${\bf(a3)}$ and Hypothesis ${\bf(h1)}$-${\bf(h3)}$, for any $T>0$, there exists a  unique global decoupling field $\gamma(s,x,\mu)$ for the FBODE system \eqref{FBODE} on $[0,T]$ with terminal data $p(x,\mu)=\p_x k(x,\mu)$, which is differentiable in $s\in[0,T]$, $x\in\R^{d_x}$ and $L$-differentiable in $\mu\in\mc{P}_2(\R^{d_x})$, and its derivatives  $\p_s\gamma(s,x,\mu)$, $\p_x\gamma(s,x,\mu)$ and $\p_\mu\gamma(s,x,\mu)(\wt{x})$ are jointly continuous in their corresponding arguments $(s,x,\mu)\in [t,T]\times \R^{d_x} \times \mathcal{P}_2(\R^{d_x})$ and $(s,x,\mu, \wt{x})\in [t,T]\times \R^{d_x} \times \mathcal{P}_2(\R^{d_x})\times \R^{d_x}$, respectively; it also satisfies $\vertiii{\gamma }_{2,[0,T]}\leq L^*_0$ with $L^*_0$ defined in \eqref{L_star_0}.
Moreover, for any $t\in[0,T]$ and $m\in\mc{P}_2(\R^{d_x})$, there exists a solution pair $\big(X^{t,m}_s(x),Z^{t,m}_s(x)\big)_{x\in\R^{d_x},s\in[t,T]}$ of the FBODE system \eqref{FBODE}, which is defined by solving for $X^{t,m}_s(x)$ in \eqref{xeq} and $Z^{t,m}_s(x):=\gamma(s,X^{t,m}_s(x),X^{t,m}_s\ot m)$, such that, both $X^{t,m}_s(x)$ and $Z^{t,m}_s(x)$ are continuously differentiable in $s\in[t,T]$ and, for each $s\in[t ,T]$ and $x\in\R^{d_x}$,
\begin{align}\label{eq_7_43}
&\big|Z^{t,m}_s(x)-Z^{t,m_0}_s(x_0)\big|
\leq \ L^*_0\Big(\big|X^{t,m}_s(x)-X^{t,m_0}_s(x_0)\big|+W_1(X^{t,m}_s\ot m,X^{t,m_0}_s\ot m_0)\Big),\\\label{eq_7_44}
&\left|Z^{t,m}_s(x)\right|\leq L^*_0 \left(\left|X^{t,m}_s(x)\right|+\left\|X^{t,m}_s\right\|_{L^{1,d_x}_m}\right).
\end{align} 
Furthermore, both $X^{t,m}_s(x)$ and $Z^{t,m}_s(x)$ are differentiable in $t\in[0 ,T]$, $x\in\R^{d_x}$ and are $L$-differentiable in $m\in\mc{P}_2(\R^{d_x})$ with their corresponding derivatives $\Big(\p_t\big(X^{t,m}_s(x)\big),\p_t\big(Z^{t,m}_s(x)\big)\Big)$, $\Big(\p_x\big(X^{t,m}_s(x)\big),\p_x\big(Z^{t,m}_s(x)\big)\Big)$ and $\Big(\p_m\big(X^{t,m}_s(x)\big)(y), \p_m\big(Z^{t,m}_s(x)\big)(y)\Big)$ being continuous in their corresponding arguments, and these derivatives are also continuously differentiable in $s\in[t,T]$. In addition, the following estimates hold:
\begin{align}\label{Bp_xX_G}
&\left\|\p_x\big(X^{t,m}_s(x)\big)\right\|_{\mc{L}(\R^{d_x};\R^{d_x})}\leq \exp\Big(L_B'(s-t)\Big);
\\\label{Bp_mX_G}
&\left\|\p_m\big(X^{t,m}_s(x)\big)(y)\right\|_{\mc{L}(\R^{d_x};\R^{d_x})}\leq L_M^{(s-t)};
\\\label{Bp_tX_G}
&\left|\p_t\big(X^{t,m}_s(x)\big)\right|\leq L_B'\big(1+|x|+\|m\|_1\big)\Big(L_B'(s-t)\exp\Big(2L_B'(s-t)\Big)+1\Big)\exp\Big(L_B'(s-t)\Big);\\\label{Bp_xZ_G}
&\left\|\p_x\big(Z^{t,m}_s(x)\big)\right\|_{\mc{L}(\R^{d_x};\R^{d_x})}\leq L^*_0\exp\Big(L_B'(s-t)\Big);
\\\label{Bp_mZ_G}
&\left\|\p_m\big(Z^{t,m}_s(x)\big)(y)\right\|_{\mc{L}(\R^{d_x};\R^{d_x})}\leq 2L^*_0L_M^{(s-t)}+L^*_0\exp\Big(L_B'(s-t)\Big);
\\\label{Bp_tZ_G}
&\left|\p_t\big(Z^{t,m}_s(x)\big)\right|\leq 2L^*_0L_B'\big(1+|x|+\|m\|_1\big)\Big(L_B'(s-t)\exp\Big(2L_B'(s-t)\Big)+1\Big)\exp\Big(L_B'(s-t)\Big),
\end{align}
where 
\begin{align}\label{L_Bp}
L_B':=&\ \Lambda_f(1+L_\alpha+L^*_0L_\alpha),\\\label{L_M^{(s-t)}}
L_M^{(\tau)}:=&\ L_B'\tau\Big(L_B'\tau\cdot\exp\Big(2L_B'\tau\Big)+1\Big)\exp\Big(2L_B'\tau\Big).
\end{align}
\end{theorem}

\begin{proof}
Define $\wb{\eps}_2=:\eps_2(L^*_0; L_f,\Lambda_f,\wb{l}_f,L_g,\Lambda_g,\wb{l}_g,L_\alpha)$ where $\eps_2$ is given in Theorem \ref{Thm6_2}. Recall that $L^*_0$ was defined in \eqref{L_star_0}. Partition the interval $[0 ,T]$ into sub-intervals with endpoints $0  < t_1 <...<t_{N-1}:=T-\wb{\eps}_2<t_N:=T-\frac{1}{2}\wb{\eps}_2<t_{N+1} =:T$ where $t_{N-i}:=T-\frac{i+1}{2}\wb{\eps}_2$, $i=0,1,...,N-1$ and $N$ is the smallest integer such that $N\geq 2T/\wb{\eps}_2-1$. Note that the sub-interval $[0 ,t_1=T-\frac{N}{2}\wb{\eps}_2]$ has a length not longer than those other intervals $[t_1,t_2]$,..., and $[t_{N-1},t_N]$, everyone of which has a common length of size $\frac{1}{2}\wb{\eps}_2$. Now, we are ready to paste the local solution $\gamma$ of \eqref{gammaeq} together to obtain a global one by using the idea depicted in Figure 1.

\begin{figure}[h!]\label{fig_1}
	\centering
	\begin{tikzpicture}[scale=1.4]
		\draw[-{stealth[scale=4]}, very thick, black] (-0.3, 0) -- (10.3, 0);
		\draw (0, -0.06) node[below, scale=0.5]{$t_0 = 0$} -- (0, 0.06);
		\draw (4, -0.06) node[below, scale=0.5]{$t_{N-2}$} -- (4, 0.06);
		\draw (6, -0.06) node[below, scale=0.5]{$t_{N-1}$} -- (6, 0.06);
		\draw (8, -0.06) node[below, scale=0.5]{$t_{N}$} -- (8, 0.06);
		\draw (10, -0.06) node[below, scale=0.5]{$t_{N+1}=T$} -- (10, 0.06);
		
		\draw [decorate, decoration={brace,mirror}]
		(8, -0.4) -- node[below=0.1cm, scale=0.5]{distance = $\wb{\eps}_2/2$} (10, -0.4);
		\draw [decorate, decoration={brace,mirror}]
		(6, -0.8) -- node[below=0.1cm, scale=0.5]{distance = $\wb{\eps}_2/2$} (8, -0.8);
		\draw [decorate, decoration={brace,mirror}]
		(4, -1.2) -- node[below=0.1cm, scale=0.5]{distance = $\wb{\eps}_2/2$} (6, -1.2);

\draw[color=red, dashed] (6, 0.1) -- (6, 1.6) node[above, xshift=1.5cm, align=center, scale=0.5, color=red]{Step 3\\paste together};
		\draw[color=red, dashed] (8, 0.1) -- (8, 1.6);
	
		\draw[->, color=black, thick] (4, 1.0) node[above, scale=0.5, xshift=1cm, color=black]{Step 2}--(8,1.0)--(8,0.5)--(4, 0.5);
           
		\draw[-, color=blue, thick] (6, 1.1) node[above, scale=0.5, xshift=7cm, color=blue]{Step 1}--(10,1.1);
\draw[-, color=blue, line width=0.8mm] (10,1.1) node[above, scale=0.5, xshift=7cm, color=blue]{}--(10,0.4);
\draw[->, color=blue, thick] (10,0.4) node[above, scale=0.5, xshift=7cm, color=blue]{}--(6, 0.4);
\draw[->, color=purple, thick] (4, 1.3) node[above, scale=0.5, xshift=1.8cm, color=purple]{Step 3} --(10,1.3)--(10,0.2)--(4, 0.2);
		
	\end{tikzpicture}
	\caption{Algorithm for pasting the local-in-time solutions.}
\end{figure}
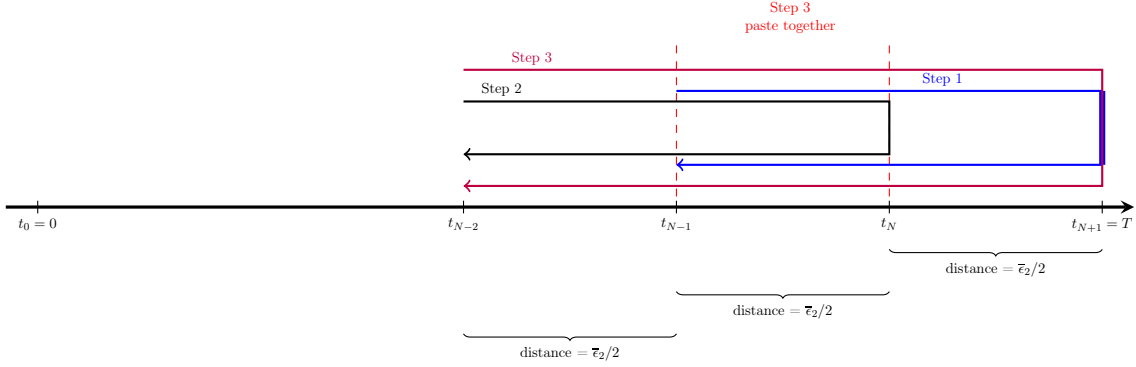

Step 1. Over the interval $ [t_{N-1}=T-\wb{\eps}_2,T]$, take $p(x,\mu)=\p_x k(x,\mu)$, and apply Theorem \ref{Thm6_1} and \ref{Thm6_2}, then, the equation \eqref{gammaeq} has a unique solution $\gamma^{(N)}(s,x,\mu)$ over the time interval $[t_{N-1},T]$ with the regularity $(i)$ stated in Theorem \ref{Thm6_2} and satisfying $\vertiii{\gamma^{(N)}}_{2,[t_{N-1},T]}\leq 2\max\{\Lambda_k,L^*_0\}=2L^*_0$. Moreover, applying Theorem \ref{Crucial_Estimate}, we actually have a better refined estimate $\vertiii{\gamma^{(N)}}_{2,[t_{N-1},T]}\leq L^*_0$.

Step 2. Over the interval $ [t_{N-2}=t_N-\wb{\eps}_2,t_N]$, take $p(x,\mu)=\gamma^{(N)}(t_N,x,\mu)$, and by referring to \eqref{p1} and \eqref{Gamma_2}, we have $L_p\leq L^*_0$, and then applying Theorem \ref{Thm6_1} and \ref{Thm6_2}, the equation \eqref{gammaeq} has a unique solution $\gamma^{(N-1)}(s,x,\mu)$ over the time interval $[t_{N-2},t_N]$ with the regularity $(i)$ stated in Theorem \ref{Thm6_2} and satisfying $\vertiii{\gamma^{(N-1)}}_{2,[t_{N-2},t_N]}\leq 2L^*_0$. 

Step 3. By the local uniqueness of $\gamma$ that was stated in Theorem \ref{Thm6_1}, for $s\in [t_{N-1},t_N]$, 
\small\begin{align}\no
\gamma^{(N-1)}(s,x,\mu)=\gamma^{(N)}(s,x,\mu),
\end{align}\normalsize
since both of them satisfy the following equation: 
for any $(s,x,m)\in [t_{N-1},t_N]\times \R^{d_x}\times\mathcal{P}_2(\R^{d_x})$,
\small\begin{align}\no
\gamma (s,x,m)
= &\ \gamma^{(N)}(t_N,X^{s,m}_{t_N}(x),X^{s,m}_{t_N}\ot m)\\\no
&+ \int_s^{t_N} \Bigg(\p_x f(X^{s,m }_\tau(x),X^{s,m }_\tau\ot m,\alpha(X^{s,m }_\tau(x),X^{s,m }_\tau\ot m,Z^{s,m }_\tau(x)))\cdot Z^{s,m }_\tau(x)\\
&\ \ \ \ \ \ \ \ \ \ \ \ \ \ +\p_x g(X^{s,m }_\tau(x),X^{s,m }_\tau\ot m,\alpha(X^{s,m }_\tau(x),X^{s,m }_\tau\ot m,Z^{s,m }_\tau(x)))\Bigg)d\tau,
\end{align}\normalsize
where $Z^{s,m }_\tau(x):=\gamma\big(\tau,X^{s,m }_\tau(x),X^{s,m }_\tau\ot m\big)$ and
\small\begin{align}     
X^{s,m}_\tau(x) = x+\int_s^\tau f\Big(X^{s,m }_{\wt{\tau}}(x),X^{s,m }_{\wt{\tau}}\ot m,\alpha\big(X^{s,m }_{\wt{\tau}}(x),X^{s,m }_{\wt{\tau}}\ot m,\gamma(\tau,X^{s,m }_{\wt{\tau}}(x),X^{s,m }_{\wt{\tau}}\ot m)\big)\Big)d\wt{\tau}.
\end{align}\normalsize
Therefore, for any $t\in [t_{N-1},t_N]$, we paste together the local solutions $\gamma^{(N)}(s,x,\mu)$ and $\gamma^{(N-1)}(s,x,\mu)$ consecutively in order by extending $\gamma^{(N-1)}(s,x,\mu)$ as follows: for any $s\in[t_N,T]$,  $\gamma^{(N-1)}(s,x,\mu)=\gamma^{(N)}(s,x,\mu)$.
We still denote the unique extended solution over the time horizon up to  the terminal time $[t_{N-2},T]$ by $\gamma^{(N-1)}(s,x,\mu)$ without causing much ambiguity, which thus solves the equation \eqref{gammaeq} over the time horizon up to  the terminal time $[t_{N-2},T]$. In addition, $\gamma^{(N-1)}(s,x,\mu)$ on the time interval $[t_{N-2},t_N]$ still has the regularity $(i)$ stated in Theorem \ref{Thm6_2} and satisfies $\vertiii{\gamma^{(N-1)}}_{2,[t_{N-2},T]}\leq 2L^*_0$. Moreover, applying Theorem \ref{Crucial_Estimate} again, we also have the refined estimate $\vertiii{\gamma^{(N-1)}}_{2,[t_{N-2},T]}\leq L^*_0$. 

Step 4. Over the interval $[t_{N-3},t_{N-1}]$, we repeat Step 2-Step 3. More precisely, due to the uniform (in time) estimate \eqref{JFE} in Theorem \ref{Crucial_Estimate}, the minimal lifespan $\wb{\eps}_2$ is uniform in all sub-intervals, so we can carry out the same pasting procedure for all the rest of the other intervals $[t_{N-4},t_{N-2}]$,...,$[0 ,t_2]$ in a backward manner. As a result, we can obtain a unique global solution $\gamma(s,x,\mu)$ of the equation \eqref{gammaeq} on the whole time horizon $[0,T]$, so that it has the regularity $(i)$ stated in Theorem \ref{Thm6_2}, that is the regularity of $\gamma$ stated in this theorem, and satisfies $\vertiii{\gamma}_{2,[0,T]}\leq L^*_0$ as well. 

Step 5. Finally, for any $t\in[0,T]$ and $m\in\mc{P}_2(\R^{d_x})$, define the pair $\big(X^{t,m}_s(x),Z^{t,m}_s(x)\big)$ by \eqref{xeq} and $Z^{t,m}_s(x):=\gamma(s,X^{t,m}_s(x),X^{t,m}_s\ot m)$ for all $s\in[t,T]$ and $x\in\R^{d_x}$, then it is a solution pair of the FBODE system \eqref{FBODE}, and \eqref{eq_7_43} and \eqref{eq_7_44} are valid. The regularity of $\big(X^{t,m}_s(x),Z^{t,m}_s(x)\big)$ stated in this theorem statement directly follows from that of $\gamma$ stated in Theorem \ref{GlobalSol}. In addition, by the same argument as that in the proof of \eqref{p_x_X_bdd}-\eqref{p_t_Z_bdd_n}, we have \eqref{Bp_xX_G}-\eqref{Bp_tX_G} and then \eqref{Bp_xZ_G}-\eqref{Bp_tZ_G} by using $\vertiii{\gamma}_{2,[0,T]}\leq L^*_0$.
\end{proof}

\begin{theorem}\label{thm_master}
(Global in time solvability of the master equation \eqref{Master_eq}). Under Assumption ${\bf(a1)}$-${\bf(a3)}$ and Hypothesis ${\bf(h1)}$-${\bf(h3)}$, then there exists a global-in-time classical solution $V(t,x,m)$, given by \eqref{Master_eq_sol}, of the master equation \eqref{Master_eq}, and it is differentiable in $t\in[0 ,T]$, $x\in\R^{d_x}$, $L$-differentiable in $m\in\mc{P}_2(\R^{d_x})$ with its corresponding derivatives $\p_t V(t,x,m)$, $\p_x V(t,x,m)$ and $\p_m V(t,x,m)(y)$ being continuous in their respective arguments. In addition, $\p_x V(t,x,m)$ is differentiable in $t\in[0,T]$, $x\in\R^{d_x}$ and $L$-differentiable in $m\in\mc{P}_2(\R^{d_x})$ with its corresponding derivatives being continuous in their respective arguments, and $\p_tV(t,x,m)$ and $\p_m V(t,x,m)(y)$ are continuously differentiable in $x\in\R^{d_x}$.
\end{theorem}
\begin{proof}
Under Assumption ${\bf(a1)}$-${\bf(a3)}$ and Hypothesis ${\bf(h1)}$-${\bf(h3)}$, by Theorem \ref{GlobalSol}, we then have a global solution pair $\big(X^{t,m}_s(x),Z^{t,m}_s(x)\big)_{x\in\R^{d_x},s\in[t,T]}$ of the FBODE system \eqref{FBODE} for any $t\in[0 ,T]$ and $m\in\mc{P}_2(\R^{d_x})$. Define $V(t,x,m)$ by \eqref{Master_eq_sol}, then, the regularity of $V(t,x,m)$ directly follows from the regularity of $\big(X^{t,m}_s(x),Z^{t,m}_s(x)\big)$ and the coefficient functions $f$, $g$ and $k$, and, by the flow property \eqref{coro_6_4} of $X^{t,m}_s(x)$, we have 
\footnotesize\begin{align}
V(s,X^{t,m}_s(x),X^{t,m}_s\ot m)=k\big(X^{t,m}_T(x),X^{t,m}_T\ot m\big)+\int_s^T g\Big(X^{t,m}_\tau(x),X^{t,m}_\tau\ot m, \alpha\big(X^{t,m}_\tau(x),X^{t,m}_\tau\ot m,Z^{t,m}_\tau(x)\big)\Big)d\tau,
\end{align}\normalsize
which implies, by taking derivative with respect to $s$ on both sides, 
\footnotesize\begin{align}\no
&\p_s V(s,X^{t,m}_s(x),X^{t,m}_s\ot m)+\p_x V(s,X^{t,m}_s(x),X^{t,m}_s\ot m)\cdot f\Big(X^{t,m}_s(x),X^{t,m}_s\ot m,\alpha\big(X^{t,m}_s(x),X^{t,m}_s\ot m,Z^{t,m}_s(x)\big)\Big)\\\no
&+\int_{\R^{d_x}}\p_\mu V(s,X^{t,m}_s(x),X^{t,m}_s\ot m)\big(X^{t,m}_s(\wt{x})\big)\cdot f\Big(X^{t,m}_s(\wt{x}),X^{t,m}_s\ot m,\alpha\big(X^{t,m}_s(\wt{x}),X^{t,m}_s\ot m,Z^{t,m}_s(\wt{x})\big)\Big)dm(\wt{x})\\\label{eq_7_69_2}
=&\ -g\Big(X^{t,m}_s(x),X^{t,m}_s\ot m,\alpha\big(X^{t,m}_s(x),X^{t,m}_s\ot m,Z^{t,m}_s(x)\big)\Big).
\end{align}\normalsize
In addition, by taking the derivative of Equation \eqref{Master_eq_sol} with respect to $x$ and using the notations in Section \ref{sec:notations}, we have 
\footnotesize\begin{align}\nonumber
& \p_x V(t,x,m)=\p_x\bigg( k^{t,m}_T(x) +\int_t^T g^{t,m}_s(x) ds\bigg)\\\nonumber
=&\   {\color{blue}\big(\p_x k\big)^{t,m}_T(x)\cdot \p_x\big(X^{t,m}_T(x)\big)}\\\label{eq_53}
&+ \int_t^T \bigg(\p_x\big(\alpha^{t,m}_s(x)\big)\cdot \big(\p_\alpha g\big)^{t,m}_s(x){\color{red}-\bigg(\dfrac{d}{ds}Z^{t,m}_s(x)+\big(\p_x f\big)^{t,m}_s(x)\cdot Z^{t,m}_s(x)\bigg)\cdot \p_x\big(X^{t,m}_s(x)\big)}\bigg)ds,
\end{align}\normalsize
where we use the second last equation of \eqref{FBODE} to obtain the red terms of the last equality.
By the forward equation \eqref{FBODE} and a simple application of integration-by-parts, the following integral in \eqref{eq_53} can be rewritten as
\small\begin{align}\no
&\int_t^T  \dfrac{d}{ds}Z^{t,m}_s(x)\cdot \p_x\big(X^{t,m}_s(x)\big)  ds\\\no
=&\  Z^{t,m}_T(x)\cdot \p_x\big(X^{t,m}_T(x)\big) - Z^{t,m}_t(x)\cdot \p_x\big(X^{t,m}_t(x)\big)- \int_t^T Z^{t,m}_s(x)\cdot \dfrac{d}{ds}\p_x\big(X^{t,m}_s(x)\big)ds  \\\no
=&\  \big(\p_x k\big)^{t,m}_T(x) \cdot \p_x\big(X^{t,m}_T(x)\big) -  Z^{t,m}_t(x)- \int_t^T Z^{t,m}_s(x)\cdot \p_x\big(f^{t,m}_s(x)\big)ds  \\\no
=&\    \big(\p_x k\big)^{t,m}_T(x) \cdot \p_x\big(X^{t,m}_T(x)\big) - Z^{t,m}_t(x)  \\\label{eq_12_29}
&- \int_t^T Z^{t,m}_s(x)\cdot \bigg(\big(\p_x f\big)^{t,m}_s(x) \p_x\big(X^{t,m}_s(x)\big)+ \big(\p_\alpha f\big)^{t,m}_s(x)\p_x\big(\alpha^{t,m}_s(x)\big)\bigg)ds.
\end{align}\normalsize
Substituting \eqref{eq_12_29} into \eqref{eq_53} so as to cancel some similar terms, one has, by using the first-order condition \eqref{first_order_condition} in the last equality,
\begin{align}\label{eq_7_72_2}
 \p_xV(t,x,m)=&\ \int_t^T \p_x\big(\alpha^{t,m}_s(x)\big)\cdot \Big(\big(\p_\alpha g\big)^{t,m}_s(x)+\big(\p_\alpha f\big)^{t,m}_s(x)\cdot Z^{t,m}_s(x)\Big)ds+Z^{t,m}_t(x)=Z^{t,m}_t(x).
\end{align}
Then, the regularity of $\p_x V(t,x,m)$ directly follows from the regularity of $ Z^{t,m}_t(x)=\gamma(t,x,m)$.
Set $s=t$ in \eqref{eq_7_69_2} and then the resulting equation is exactly \eqref{Master_eq} since $X^{t,m}_t(x)=x$ and $X^{t,m}_t\ot m=m$. In addition, the differentiability of $\p_x V(t,x,m)=\gamma(t,x,m)$ follows from that of $\gamma$ in Theorem \ref{GlobalSol} immediately. 

Next, similar to the computations in deriving \eqref{eq_53}-\eqref{eq_7_72_2}, we compute $\p_t V$ as follows. Taking the derivative of the equation \eqref{Master_eq_sol} with respect to $t$ and using the notations in Section \ref{sec:notations}, we can obtain 
\small\begin{align}\nonumber
& \p_t V(t,x,m)=\p_t\bigg( k^{t,m}_T(x) +\int_t^T g^{t,m}_s(x) ds\bigg)\\\nonumber
=&\   {\color{blue}\int_{\R^{d_x}} \big(\p_\mu k\big)^{t,m}_T(x,\wt{x})\cdot \p_t\big(X^{t,m}_T(\wt{x})\big)dm(\wt{x}) +\big(\p_x k\big)^{t,m}_T(x)\cdot \p_t\big(X^{t,m}_T(x)\big)}- g^{t,m}_t(x) \\\nonumber
&+ \int_t^T \bigg(\int_{\R^{d_x}}\big(\p_\mu g\big)^{t,m}_s(x,\wt{x})\cdot \p_t\big(X^{t,m}_s(\wt{x})\big)dm(\wt{x})+\p_t\big(\alpha^{t,m}_s(x)\big)\cdot \big(\p_\alpha g\big)^{t,m}_s(x)  \\\label{eq_52}
&\ \ \ \ \ \ \ \ \ \ \ \ \ \ \ \ {\color{red}-\bigg(\dfrac{d}{ds}Z^{t,m}_s(x)+\big(\p_x f\big)^{t,m}_s(x)\cdot Z^{t,m}_s(x)\bigg)\cdot \p_t\big(X^{t,m}_s(x)\big)}\bigg)ds,
\end{align}\normalsize
where we use the second last equation of \eqref{FBODE} to obtain the red terms of the last equality. By the forward equation \eqref{FBODE} and a simple application of integration-by-parts, the following integral in \eqref{eq_52} can be rewritten as
\small\begin{align}\no
&\int_t^T  \dfrac{d}{ds}Z^{t,m}_s(x)\cdot \p_t\big(X^{t,m}_s(x)\big)  ds\\\no
=&\  {\color{blue}Z^{t,m}_T(x)\cdot \p_t\big(X^{t,m}_T(x)\big)} - Z^{t,m}_t(x)\cdot \p_t\big(X^{t,m}_s(x)\big)|_{s=t}- \int_t^T {\color{purple}Z^{t,m}_s(x)\cdot \dfrac{d}{ds}\p_t\big(X^{t,m}_s(x)\big)}ds  \\\no
=&\  {\color{blue} \big(\p_x k\big)^{t,m}_T(x) \cdot \p_t\big(X^{t,m}_T(x)\big)}+ Z^{t,m}_t(x)\cdot f^{t,m}_t(x)- \int_t^T {\color{purple}Z^{t,m}_s(x)\cdot \p_t\big(f^{t,m}_s(x)\big)}ds  \\\no
=&\  {\color{blue}  \big(\p_x k\big)^{t,m}_T(x) \cdot \p_t\big(X^{t,m}_T(x)\big)} + Z^{t,m}_t(x)\cdot f^{t,m}_t(x) \\\no
&- \int_t^T {\color{purple}Z^{t,m}_s(x)\cdot \bigg(\big(\p_x f\big)^{t,m}_s(x)\cdot\p_t\big(X^{t,m}_s(x)\big)+ \p_t\big(\alpha^{t,m}_s(x)\big)\cdot \big(\p_\alpha f\big)^{t,m}_s(x)}\\\label{eq_12_29_1}
&\ \ \ \ \ \ \ \ \ \ \ \ \ \ \ \ \ \ \ \ \ \ \ \ {\color{purple}+\displaystyle\int_{\R^{d_x}}\big(\p_\mu f\big)^{t,m}_s(x,\wt{x})\cdot\p_t\big(X^{t,m}_s(\wt{x})\big)dm(\wt{x})\bigg)}ds,
\end{align}\normalsize
and thus, by substituting \eqref{eq_12_29_1} into \eqref{eq_52} so as to cancel some similar terms, one has, by using the first-order condition \eqref{first_order_condition} in the last equality,
\small\begin{align}\no
 \p_tV(t,x,m)
=&\ \int_t^T \int_{\R^{d_x}}\Big(\big(\p_\mu g\big)^{t,m}_s(x,\wt{x})+\big(\p_\mu f\big)^{t,m}_s(x,\wt{x})\cdot Z^{t,m}_s(x)\Big)\cdot  \p_t\big(X^{t,m}_s(\wt{x})\big)dm(\wt{x}) ds \\
&+ \int_{\R^{d_x}} \big(\p_\mu k\big)^{t,m}_T(x,\wt{x})\cdot \p_t\big(X^{t,m}_T(\wt{x})\big)dm(\wt{x})- \bigg(Z^{t,m}_t(x)\cdot f^{t,m}_t(x)+g^{t,m}_t(x)\bigg),
\end{align}\normalsize

Finally, we compute $\p_m V$ as follows. By taking the derivative of Equation \eqref{Master_eq_sol} with respect to $m$ and using the notations in Section \ref{sec:notations}, we can obtain 
\small\begin{align}\nonumber
& \p_m V(t,x,m)(y)=\p_m\bigg( k^{t,m}_T(x) +\int_t^T g^{t,m}_s(x) ds\bigg)(y)\\\nonumber
=&\   {\color{blue}\int_{\R^{d_x}} \big(\p_\mu k\big)^{t,m}_T(x,\wt{x})\cdot \p_m\big(X^{t,m}_T(\wt{x})\big)(y)dm(\wt{x})+ \big(\p_\mu k\big)^{t,m}_T(x,y)\cdot \p_y\big(X^{t,m}_T(y)\big) +\big(\p_x k\big)^{t,m}_T(x)\cdot \p_m\big(X^{t,m}_T(x)\big)(y)}\\\nonumber
&+ \int_t^T \bigg(\int_{\R^{d_x}}\big(\p_\mu g\big)^{t,m}_s(x,\wt{x})\cdot \p_m\big(X^{t,m}_s(\wt{x})\big)(y)dm(\wt{x})+\big(\p_\mu g\big)^{t,m}_s(x,y)\cdot \p_y\big(X^{t,m}_s(y)\big) \\\label{eq_54}
&\ \ \ \ \ \ \ \ \ \ \ \ +\p_m\big(\alpha^{t,m}_s(x)\big)(y)\cdot \big(\p_\alpha g\big)^{t,m}_s(x) {\color{red}-\bigg(\dfrac{d}{ds}Z^{t,m}_s(x)+\big(\p_x f\big)^{t,m}_s(x)\cdot Z^{t,m}_s(x)\bigg)\cdot \p_m\big(X^{t,m}_s(x)\big)(y)}\bigg)ds,
\end{align}\normalsize
where we use the second last equation of \eqref{FBODE} to obtain the red terms of the last equality. By the forward equation \eqref{FBODE} and a simple application of integration-by-parts, the following integral in \eqref{eq_54} can be rewritten as 
\small\begin{align}\no
&\int_t^T  \dfrac{d}{ds}Z^{t,m}_s(x)\cdot \p_m\big(X^{t,m}_s(x)\big)(y)  ds\\\no
=&\  {\color{blue}Z^{t,m}_T(x)\cdot \p_m\big(X^{t,m}_T(x)\big)(y)} - Z^{t,m}_t(x)\cdot \p_m\big(X^{t,m}_t(x)\big)(y)- \int_t^T {\color{purple}Z^{t,m}_s(x)\cdot \dfrac{d}{ds}\p_m\big(X^{t,m}_s(x)\big)(y)}ds  \\\no
=&\  {\color{blue} \big(\p_x k\big)^{t,m}_T(x) \cdot \p_m\big(X^{t,m}_T(x)\big)(y)} - \int_t^T {\color{purple}Z^{t,m}_s(x)\cdot \p_m\big(f^{t,m}_s(x)\big)(y)}ds  \\\no
=&\  {\color{blue}  \big(\p_x k\big)^{t,m}_T(x) \cdot \p_m\big(X^{t,m}_T(x)\big)(y)} \\\no
&- \int_t^T {\color{purple}Z^{t,m}_s(x)\cdot \bigg(\big(\p_x f\big)^{t,m}_s(x)\cdot\p_m\big(X^{t,m}_s(x)\big)(y)+\displaystyle\int_{\R^{d_x}}\big(\p_\mu f\big)^{t,m}_s(x,\wt{x})\cdot\p_m\big(X^{t,m}_s(\wt{x})\big)(y)dm(\wt{x})}\\\label{eq_12_30}
&\ \ \ \ \ \ \ \ \ \ \ \ \ \ \ \ \ \ \ \ \ \ \ \ {\color{purple}+\big(\p_\mu f\big)^{t,m}_s(x,y)\cdot\p_y\big(X^{t,m}_s(y)\big)+ \p_m\big(\alpha^{t,m}_s(x)\big)(y)\cdot \big(\p_\alpha f\big)^{t,m}_s(x)\bigg)}ds,
\end{align}\normalsize
and thus, by substituting \eqref{eq_12_30} into \eqref{eq_54} so as to cancel some similar terms, one has, by using the first-order condition \eqref{first_order_condition} in the last equality,
\small\begin{align}\nonumber
&\p_m V(t,x,m)(y)\\\no
=&\int_{\R^{d_x}} \big(\p_\mu k\big)^{t,m}_T(x,\wt{x})\cdot \p_m\big(X^{t,m}_T(\wt{x})\big)(y)dm(\wt{x})+ \big(\p_\mu k\big)^{t,m}_T(x,y)\cdot \p_y\big(X^{t,m}_T(y)\big) \\\nonumber
&+ \int_t^T \bigg(\int_{\R^{d_x}}\Big(\big(\p_\mu f\big)^{t,m}_s(x,\wt{x})
\cdot Z^{t,m}_s(x)+\big(\p_\mu g\big)^{t,m}_s(x,\wt{x})\Big)\cdot \p_m\big(X^{t,m}_s(\wt{x})\big)(y)dm(\wt{x})\\\label{eq_54_1}
&\ \ \ \ \ \ \ \ \ \ \ \ +\bigg(\big(\p_\mu f\big)^{t,m}_s(x,y)\cdot Z^{t,m}_s(x)+\big(\p_\mu g\big)^{t,m}_s(x,y)\bigg)\cdot \p_y\big(X^{t,m}_s(y)\big)\bigg)ds.
\end{align}\normalsize
Therefore, the continuous differentiability of $\p_t V(t,x,m)$ and $\p_m V(t,x,m)(y)$ with respect to $x\in\R^{d_x}$ directly follows from the regularity of $Z^{t,m}_s(x)$ and the coefficient functions $f$, $g$ and $k$.
\end{proof}

\section{Local and Global Uniqueness}\label{sec:unique}

\begin{theorem}\label{Thm6_3} (Local uniqueness of the solution to the FBODE system \eqref{fbodesystem})
Assume that the drift function $f$ and the running cost $g$ satisfy Assumptions ${\bf(a1)}$ and ${\bf(a2)}$ respectively, the terminal function $p$ satisfies Assumption ${\bf(P)}$ with $L_p\leq \wb{L}_p:=\max\left\{\Lambda_k,L^*_0\right\}$ and the inequality relation \eqref{p_aa_f_new} among $f$, $g$ and $p$ is valid, then there exists a constant $\eps_3=\eps_3(\wb{L}_p; L_f,\Lambda_f,\wb{l}_f,L_g,\Lambda_g,\wb{l}_g,L_\alpha)>0$, such that, for any fixed $0\leq t\leq \wt{T}\leq T$ with $\wt{T}-t\leq \eps_3$ and $m\in \mathcal{P}_2(\R^{d_x})$, if there exist two solution pairs $\big(X^{t,m}_s(x),Z^{t,m}_s(x)\big)_{x\in\R^{d_x},s\in[t,\wt{T}]}$ and $\big(\wt{X}^{t,m}_s(x),\wt{Z}^{t,m}_s(x)\big)_{x\in\R^{d_x},s\in[t,\wt{T}]}$ of FBODE system \eqref{fbodesystem}, each of them is continuously differentiable in $s\in[t,\wt{T}]$, then they must be equal for all $s\in[t,\wt{T}]$ and $x\in\R^{d_x}$.
\end{theorem}
The proof of Theorem \ref{Thm6_3} is similar to, but simpler\footnote{The simplification is due to the absence of the terms $\p_\mu f(x,\mu,\alpha)$ and $\p_\mu g(x,\mu,\alpha)$ on the right-hand side of the FBODE system \eqref{fbodesystem}, in comparison with (7.1) of \cite{bensoussan2023theory}.} than, the argument leading to Theorem 9.1 of our previous paper \cite{bensoussan2023theory}. Therefore, we omit the details of the proof here.

\begin{theorem}\label{Thm6_4} (Global uniqueness of the solution to the FBODE system \eqref{FBODE})
Under Assumptions ${\bf(a1)}$-${\bf(a3)}$ and Hypotheses ${\bf(h1)}$-${\bf(h3)}$, for any fixed terminal time $T>0$ and initial distribution $m\in\mc{P}_2(\R^{d_x})$, if there exists a continuously differentiable, in $s\in[0,T]$, solution pair $\big(\wt{X}^{0,m}_s(x),\wt{Z}^{0,m}_s(x)\big)_{x\in\R^{d_x},s\in[0,T]}$ of FBODE system \eqref{FBODE} subject to the initial distribution $m$ at the initial time $t=0$, then it must be equal to the solution pair $\big(X^{0,m}_s(x),Z^{0,m}_s(x):=\gamma(s,X^{0,m}_s(x),X^{0,m}_s\ot m)\big)_{x\in\R^{d_x},s\in[0,T]}$, where the forward dynamics $X^{0,m}_s(x)$ was obtained by solving \eqref{xeq} over $[0,T]$ and the decoupling field $\gamma$ was constructed in Theorem \ref{GlobalSol}.
\end{theorem}
\begin{proof}
Under Assumptions ${\bf(a1)}$-${\bf(a3)}$ and Hypotheses ${\bf(h1)}$-${\bf(h3)}$, according to Theorem \ref{GlobalSol}, we have a global decoupling field $\gamma(s,x,\mu)$ for the FBODE system \eqref{FBODE} on $[0,T]$, which satisfies the equation \eqref{gammaeq}. Let $\wb{L}_p:=\max\left\{\Lambda_k,L^*_0\right\}=L^*_0$ and $\eps_3=\eps_3(\wb{L}_p; L_f,\Lambda_f,\wb{l}_f,L_g,\Lambda_g,\wb{l}_g,L_\alpha)>0$ given in Theorem \ref{Thm6_3}. Partition the interval $[0 ,T]$ into $0  < t_1 <...<t_{N-1}:=T-\eps_3<t_N =:T$, where $t_{N-i}=T-i\cdot \eps_3$, $i=0,1,...,N-1$ and $N$ is the smallest integer such that $N\geq T/\eps_3$. Note that the interval $[0 ,t_1=T-(N-1)\eps_3]$ has a length not longer than those other intervals $[t_1,t_2]$,..., and $[t_{N-1},t_N]$, while every others has a common length of size $\eps_3$. \\

Step 1. On the interval $[t_{N-1},t_N=T]$, define $y:=\wt{X}^{0,m}_{t_{N-1}}(x)$, $\wt{m}:=\wt{X}^{0,m}_{t_{N-1}}\ot m$ and $p(x,\mu):=\p_x k(x,\mu)$, then $\big(X^{t_{N-1},\wt{m}}_{s}(y),Z^{t_{N-1},\wt{m}}_{s}(y)\big):=\big(\wt{X}^{0,m}_s(x),\wt{Z}^{0,m}_s(x)\big)$ \footnote{Since the initial measure at time $t_N$ is $\wt{m}=\wt{X}^{0,m}_{t_{N-1}}\ot m$ where $\wt{X}^{0,m}_{t_{N-1}}(\cdot)$ is a push-forward mapping, so the complement set $S^c$ of the set $S:=\{y:y=\wt{X}^{0,m}_{t_{N-1}}(x),x\in\R^{d_x}\}=\wt{X}^{0,m}_{t_{N-1}}(\R^{d_x})$ is $\wt{m}$-measure zero, that is $\wt{m}(S^c)\leq m(\emptyset)=0$, and thus we can ignore those elements $y\in S^c$ while solving the FBODE system \eqref{fbodesystem} over $[t_{N-1},T]$.} is a solution pair of FBODE system \eqref{fbodesystem} with the terminal data $Z^{t_{N-1},\wt{m}}_{T}(y)=p(X^{t_{N-1},\wt{m}}_{T}(y),X^{t_{N-1},\wt{m}}_{T}\ot \wt{m})$ and the initial data $X^{t_{N-1},\wt{m}}_{t_{N-1}}(y)=y$. Note that,  for any $0\leq s\leq \tau \leq T$, $\wh{x}\in\R^{d_x}$ and $\wh{m}\in\mc{P}_2(\R^{d_x})$, the decoupling field $\gamma$ solves \eqref{gammaeq} with \eqref{xeq}:\\
\small\begin{align}\no
\gamma (s,\wh{x},\wh{m})
= &\ p(\wh{X}^{s,\wh{m}}_T(\wh{x}),\wh{X}^{s,\wh{m}}_T\ot \wh{m})\\\no
&+ \int_s^T \Bigg(\p_x f(\wh{X}^{s,\wh{m}}_\tau(\wh{x}),\wh{X}^{s,\wh{m}}_\tau\ot \wh{m},\alpha(\wh{X}^{s,\wh{m}}_\tau(\wh{x}),\wh{X}^{s,\wh{m}}_\tau\ot \wh{m},\wh{Z}^{s,\wh{m}}_\tau(\wh{x})))\cdot \wh{Z}^{s,\wh{m}}_\tau(\wh{x})\\\no
&\ \ \ \ \ \ \ \ \ \ \ \ \ \ +\p_x g(\wh{X}^{s,\wh{m}}_\tau(\wh{x}),\wh{X}^{s,\wh{m}}_\tau\ot \wh{m},\alpha(\wh{X}^{s,\wh{m}}_\tau(\wh{x}),\wh{X}^{s,\wh{m}}_\tau\ot \wh{m},\wh{Z}^{s,\wh{m}}_\tau(\wh{x})))\Bigg)d\tau,
\end{align}\normalsize
where $\wh{Z}^{s,\wh{m}}_\tau(\wh{x}):=\gamma\big(\tau,\wh{X}^{s,\wh{m}}_\tau(\wh{x}),\wh{X}^{s,\wh{m}}_\tau\ot \wh{m}\big)$ and
\small\begin{align} \no    
\wh{X}^{s,\wh{m}}_\tau(\wh{x}) = \wh{x}+\int_s^\tau f\Big(\wh{X}^{s,\wh{m}}_{\wt{\tau}}(\wh{x}),\wh{X}^{s,\wh{m}}_{\wt{\tau}}\ot \wh{m},\alpha\big(\wh{X}^{s,\wh{m}}_{\wt{\tau}}(\wh{x}),\wh{X}^{s,\wh{m}}_{\wt{\tau}}\ot \wh{m},\gamma(\tau,\wh{X}^{s,\wh{m}}_{\wt{\tau}}(\wh{x}),\wh{X}^{s,\wh{m}}_{\wt{\tau}}\ot \wh{m})\big)\Big)d\wt{\tau}.
\end{align}\normalsize
Then, one can directly check that $\big(\wh{X}^{t_{N-1},\wt{m}}_s(y),\wh{Z}^{t_{N-1},\wt{m}}_s(y)\big)$ is also a solution pair of the FBODE system \eqref{fbodesystem} on the interval $[t_{N-1},t_N=T]$ with the same initial and terminal data as $\big(X^{t_{N-1},\wt{m}}_{s}(y),Z^{t_{N-1},\wt{m}}_{s}(y)\big)$ does. By Theorem \ref{Thm6_3}, $\big(\wh{X}^{t_{N-1},\wt{m}}_s(y),\wh{Z}^{t_{N-1},\wt{m}}_s(y)\big)=\big(X^{t_{N-1},\wt{m}}_{s}(y),Z^{t_{N-1},\wt{m}}_{s}(y)\big)=\big(\wt{X}^{0,m}_s(x),\wt{Z}^{0,m}_s(x)\big)$ and thus $\wt{Z}^{0,m}_s(x)=\gamma(s,\wt{X}^{0,m}_s(x),\wt{X}^{0,m}_s\ot m)$ on the interval $s\in[t_{N-1},t_N=T]$.

Step 2. On the interval $[t_{N-2},t_{N-1}]$, define $y:=\wt{X}^{0,m}_{t_{N-2}}(x)$, $\wt{m}:=\wt{X}^{0,m}_{t_{N-2}}\ot m$ and $p(x,\mu):=\gamma(t_{N-1},x,\mu)$, then $\big(X^{t_{N-2},\wt{m}}_{s}(y),Z^{t_{N-2},\wt{m}}_{s}(y)\big):=\big(\wt{X}^{0,m}_s(x),\wt{Z}^{0,m}_s(x)\big)$ is a solution pair of FBODE system \eqref{fbodesystem} with terminal data $Z^{t_{N-2},\wt{m}}_{t_{N-1}}(y)=p(X^{t_{N-2},\wt{m}}_{t_{N-1}}(y),X^{t_{N-2},\wt{m}}_{t_{N-1}}\ot \wt{m})$ and initial data $X^{t_{N-2},\wt{m}}_{t_{N-2}}(y)=y$. Same as Step 1, one can also directly check that $\big(\wh{X}^{t_{N-2},\wt{m}}_s(y),\wh{Z}^{t_{N-2},\wt{m}}_s(y)\big)$ is also a solution pair of FBODE system \eqref{fbodesystem} on the interval $[t_{N-2},t_{N-1}]$ with the same initial and terminal data as $\big(X^{t_{N-2},\wt{m}}_{s}(y),Z^{t_{N-2},\wt{m}}_{s}(y)\big)$ does. Since $\vertiii{\gamma }_{2,[0,T]}\leq L^*_0$, by Theorem \ref{Thm6_3} again, $\big(\wh{X}^{t_{N-2},\wt{m}}_s(y),\wh{Z}^{t_{N-2},\wt{m}}_s(y)\big)=\big(X^{t_{N-2},\wt{m}}_{s}(y),Z^{t_{N-2},\wt{m}}_{s}(y)\big)=\big(\wt{X}^{0,m}_s(x),\wt{Z}^{0,m}_s(x)\big)$, and thus $\wt{Z}^{0,m}_s(x)=\gamma(s,\wt{X}^{0,m}_s(x),\wt{X}^{0,m}_s\ot m)$ on the interval $s\in[t_{N-2},t_{N-1}]$. We can repeat the same procedure for all the rest of the other intervals $[t_{N-3},t_{N-2}]$,...,$[0 ,t_1]$ in a backward manner since Theorem \ref{GlobalSol} can ensure $\vertiii{\gamma }_{2,[0,T]}\leq L^*_0$ independent of the current subinterval. Therefore, we can obtain $\wt{Z}^{0,m}_s(x)=\gamma(s,\wt{X}^{0,m}_s(x),\wt{X}^{0,m}_s\ot m)$ on the whole interval $s\in[0,T]$.
\end{proof}

\begin{remark}
In the proof of the global existence of a continuously differentiable decoupling field $\gamma$ (refer to \eqref{displacement}), we relied on the monotonicity conditions \eqref{positive_k_mu} and \eqref{positive_H_mu}. Therefore, it is not unexpected that the global existence of such a decoupling field $\gamma$ could lead to the global uniqueness of the solution to the FBODE system \eqref{FBODE}.
\end{remark}

\begin{theorem}\label{Unique_B} (Global uniqueness of the solution to the master equation \eqref{Master_eq})
Under Assumptions ${\bf(a1)}$-${\bf(a3)}$ and Hypotheses ${\bf(h1)}$-${\bf(h3)}$, the master equation \eqref{Master_eq} has at most one solution $V(t,x,m)$ that has the regularity stated in Theorem \ref{thm_master}.
\end{theorem}
\begin{proof}
Suppose that there exist two solutions $V(t,x,m)$ and $\wb{V}(t,x,m)$ to the master equation \eqref{Master_eq}, both of which have the regularity stated in Theorem \ref{thm_master}. Define $\gamma(t,x,m):=\p_x V(t,x,m)$, $\wb{\gamma}(t,x,m):=\p_x \wb{V}(t,x,m)$, and
\begin{align}\label{Unique_X}
\begin{cases}
X^{t,m}_s(x):= x+\int_t^s f\Big(X^{t,m}_\tau(x),X^{t,m}_\tau\ot m,\alpha\big(X^{t,m}_\tau(x),X^{t,m}_\tau\ot m,\gamma(\tau,X^{t,m}_\tau(x),X^{t,m}_\tau\ot m)\big)\Big)d\tau,\\
\wb{X}^{t,m}_s(x):= x+\int_t^s f\Big(\wb{X}^{t,m}_\tau(x),\wb{X}^{t,m}_\tau\ot m,\alpha\big(\wb{X}^{t,m}_\tau(x),\wb{X}^{t,m}_\tau\ot m,\wb{\gamma}(\tau,\wb{X}^{t,m}_\tau(x),\wb{X}^{t,m}_\tau\ot m)\big)\Big)d\tau,
\end{cases}
\end{align}
$Z^{t,m}_s(x):=\gamma(s,X^{t,m}_s(x),X^{t,m}_s\ot m)$ and $\wb{Z}^{t,m}_s(x):=\wb{\gamma}(s,\wb{X}^{t,m}_s(x),\wb{X}^{t,m}_s\ot m)$. Then, by taking the derivative $\p_x$ to the master equation \eqref{Master_eq} and using the first-order condition \eqref{first_order_condition} and $\gamma(t,x,m)=\p_x V(t,x,m)$, we have 
\footnotesize\begin{align}\no
&\p_t \gamma(t,x,m)+f(x,m,\alpha(x,m,\gamma(t,x,m)))\cdot \p_x\gamma(t,x,m)+\int_{\R^{d_x}}f(y,m,\alpha(y,m,\gamma(t,y,m)))\cdot \p_\mu\gamma(t,x,m)(y)dm(y)\\\no
=&\ -\p_x f(x,m,\alpha(x,m,\gamma(t,x,m)))\cdot \gamma(t,x,m)-\p_x g(x,m,\alpha(x,m,\gamma(t,x,m))),
\end{align}\normalsize
which implies that, by evaluating at $t=s$, $x=X^{t,m}_s(x)$ and $m=X^{t,m}_s\ot m$,
\footnotesize\begin{align}\no
&\p_s \gamma(s,X^{t,m}_s(x),X^{t,m}_s\ot m)+f(X^{t,m}_s(x),X^{t,m}_s\ot m,\alpha(X^{t,m}_s(x),X^{t,m}_s\ot m,\gamma(s,X^{t,m}_s(x),X^{t,m}_s\ot m)))\cdot \p_x\gamma(s,X^{t,m}_s(x),X^{t,m}_s\ot m)\\\no
&+\int_{\R^{d_x}}f(X^{t,m}_s(y),X^{t,m}_s\ot m,\alpha(X^{t,m}_s(y),X^{t,m}_s\ot m,\gamma(s,X^{t,m}_s(y),X^{t,m}_s\ot m)))\cdot \p_\mu\gamma(s,X^{t,m}_s(x),X^{t,m}_s\ot m)(X^{t,m}_s(y))dm(y)\\\no
=&\ -\p_x f(X^{t,m}_s(x),X^{t,m}_s\ot m,\alpha(X^{t,m}_s(x),X^{t,m}_s\ot m,\gamma(s,X^{t,m}_s(x),X^{t,m}_s\ot m)))\cdot \gamma(s,X^{t,m}_s(x),X^{t,m}_s\ot m)\\\label{eq_9_17_1}
&\ -\p_x g(X^{t,m}_s(x),X^{t,m}_s\ot m,\alpha(X^{t,m}_s(x),X^{t,m}_s\ot m,\gamma(s,X^{t,m}_s(x),X^{t,m}_s\ot m))).
\end{align}\normalsize 
Applying the characteristics method to the derived equation \eqref{eq_9_17_1}, we can observe that the left-hand side of \eqref{eq_9_17_1} is precisely equal to $\frac{d}{ds}Z^{t,m}_s(x)$. The corresponding characteristic equations satisfied by both $\Big(X^{t,m}_s(x),Z^{t,m}_s(x)\Big)$ and $\Big(\wb{X}^{t,m}_s(x),\wb{Z}^{t,m}_s(x)\Big)$ are exactly the same as the FBODE system \eqref{FBODE}. This implies that $\Big(X^{t,m}_s(x),Z^{t,m}_s(x)\Big)=\Big(\wb{X}^{t,m}_s(x),\wb{Z}^{t,m}_s(x)\Big)$ by Theorem \ref{Thm6_4}. Therefore, we obtain $\p_x V(t,x,m)=\gamma(t,x,m)=Z^{t,m}_t(x)=\wb{Z}^{t,m}_t(x)=\wb{\gamma}(t,x,m)=\p_x \wb{V}(t,x,m)$ for all $t\in[0,T]$, $x\in\R^{d_x}$ and $m\in\mc{P}_2(\R^{d_x})$. Thus, by \eqref{Unique_X}, $X^{t,m}_s(x)=\wb{X}^{t,m}_s(x)$. Next, using the master equation \eqref{Master_eq} again, we can conclude that $\p_s \big(V(s,X^{t,m}_s(x),X^{t,m}_s(x)\ot m)-\wb{V}(s,\wb{X}^{t,m}_s(x),\wb{X}^{t,m}_s\ot m)\big)=0$ for all $0\leq t\leq s\leq T$ and $V(T,X^{t,m}_T(x),X^{t,m}_T\ot m)-\wb{V}(T,\wb{X}^{t,m}_T(x),\wb{X}^{t,m}_T\ot m)=0$ since both initial data are the same. Therefore, we have $V(s,X^{t,m}_s(x),X^{t,m}_s(x)\ot m)=\wb{V}(s,\wb{X}^{t,m}_s(x),\wb{X}^{t,m}_s\ot m)$ for all $0\leq t\leq s\leq T$, and thus $V(t,x,m)=\wb{V}(t,x,m)$ for all $t$ by setting $s=t$.
\end{proof}

\section{Non-linear-quadratic Example}\label{sec:nonLQ}
We here provide a non-trivial non-linear-quadratic example in which the drift function $f$, the running cost function $g$ and the terminal cost function $k$ are defined as follows; for simplicity, we just consider the process $x_t$ living in $\R$. Let $\varepsilon_2\in(0,1/2]$, $\varepsilon_3\in(0,1/8]$, $\varepsilon_4\in(0,1)$ be fixed constants. Define, for any $x$, $\alpha\in\R$, $\mu\in\mathcal{P}_2(\R)$, 
\begin{align}\label{example_f}
f(x,\mu,\alpha):= &\ x+ \alpha +\int_{\R} yd\mu(y)+\varepsilon_1 x \exp\Big(-x^2-\alpha^2-\Big(\int_\R \phi(y)d\mu(y)\Big)^2\Big),\\\label{example_g}
g(x,\mu,\alpha):=&\ \frac{1}{2}\alpha^2+\frac{1}{2} x^2-\varepsilon_2x \int_{\R} yd\mu(y)+\varepsilon_3\alpha \int_{\R} y d\mu(y),\\\label{example_k}
k(x,\mu):=&\ \frac{1}{2}x^2-\varepsilon_4 x \int_{\R} yd\mu(y),
\end{align}
where the constant $\varepsilon_1\in(0,1)$ will be chosen sufficiently small below so that $(\bf{h2})$ will be satisfied, and 
\begin{align}\label{phi_1}
\phi(y):=\begin{cases}
|y|,\ \text{ for }|y|\geq 1;\\
-\frac{1}{8}y^4+\frac{3}{4}y^2+\frac{3}{8},\ \text{ for }|y|<1.
\end{cases}
\end{align}
\begin{remark}
The example of the drift function $f(x,\mu,\alpha)$ is clearly nonlinear and non-separable in all $x$, $\mu$ and $\alpha$. The running cost $g(x,\mu,\alpha)$ is non-separable in $\alpha$ and $\mu$ and it is not convex in $\mu$ but it is of quadratic-growth in $x$. The terminal cost $k(x,\mu)$ is not convex in $\mu$ yet it is of quadratic-growth in $x$. To the best of our knowledge, this example cannot be covered as a workable one in the theory proposed in the contemporary literature. 
\end{remark}
One can check that $\phi\in C^2(\R)$ satisfies, $\phi(y)\geq |y|$ for all $y\in\R$, 
\begin{align}
\phi'(y)=\begin{cases}
1,\ \text{ for }y\geq 1;\\
-1,\ \text{ for }y\leq 1;\\
-\frac{1}{2}y^3+\frac{3}{2}y,\ \text{ for }|y|<1,
\end{cases}\text{ and  }\ \  
\phi''(y)=\begin{cases}
0,\ \text{ for }|y|\geq 1,\\
-\frac{3}{2}y^2+\frac{3}{2},\ \text{ for }|y|<1.
\end{cases}
\end{align}
In addition, we can check that $\|\mu\|_1\leq \int_\R \phi(y)d\mu(y)$. 
It follows from routine calculations that
\small\begin{align*}
(i)\ &\p_x f(x,\mu,\alpha)=1+\varepsilon_1(1-2 x^2) \exp\Big(-x^2-\alpha^2-\Big(\int_\R \phi(y)d\mu(y)\Big)^2\Big);\\ 
&\p_\mu f(x,\mu,\alpha)(\wt{x})= 1-2\varepsilon_1 \phi'(\wt{x}) x \int_\R \phi(y)d\mu(y)  \exp\Big(-x^2-\alpha^2-\Big(\int_\R \phi(y)d\mu(y)\Big)^2\Big);\\ 
&\p_\alpha f(x,\mu,\alpha)= 1-2\varepsilon_1 \alpha x \exp\Big(-x^2-\alpha^2-\Big(\int_\R \phi(y)d\mu(y)\Big)^2\Big);\\
&\p_x\p_x f(x,\mu,\alpha)=-2x\Big(3-2x^2\Big)\varepsilon_1  \exp\Big(-x^2-\alpha^2-\Big(\int_\R \phi(y)d\mu(y)\Big)^2\Big);\\
&\p_x\p_\mu f(x,\mu,\alpha)(\wt{x})=-2\varepsilon_1 \phi'(\wt{x}) (1-2x^2) \int_\R \phi(y)d\mu(y)  \exp\Big(-x^2-\alpha^2-\Big(\int_\R \phi(y)d\mu(y)\Big)^2\Big)=\p_\mu \p_x f(x,\mu,\alpha)(\wt{x});\\ 
&\p_x\p_\alpha f(x,\mu,\alpha)=-2\varepsilon_1 \alpha (1-2x^2) \exp\Big(-x^2-\alpha^2-\Big(\int_\R \phi(y)d\mu(y)\Big)^2\Big)=\p_\alpha \p_x f(x,\mu,\alpha);\\
&\p_\alpha\p_\mu f(x,\mu,\alpha)(\wt{x})=4\varepsilon_1 \phi'(\wt{x}) x\alpha \int_\R \phi(y)d\mu(y)  \exp\Big(-x^2-\alpha^2-\Big(\int_\R \phi(y)d\mu(y)\Big)^2\Big)=\p_\mu\p_\alpha f(x,\mu,\alpha)(\wt{x});\\ 
&\p_\alpha\p_\alpha f(x,\mu,\alpha)=-2\varepsilon_1 (1-2\alpha^2) x \exp\Big(-x^2-\alpha^2-\Big(\int_\R \phi(y)d\mu(y)\Big)^2\Big).\\ 
(ii)\ &\p_x g(x,\mu,\alpha)=x-\varepsilon_2\int_{\R} yd\mu(y),\ \p_\alpha g(x,\mu,\alpha)=\alpha+\varepsilon_3\int_{\R} y d\mu(y);\\
&\p_\mu g(x,\mu,\alpha)(\wt{x})=-\varepsilon_2x+\varepsilon_3\alpha;\\
&\p_\alpha\p_\alpha g(x,\mu,\alpha)=\p_x\p_x g(x,\mu,\alpha)=1,\ \p_\alpha\p_\mu g(x,\mu,\alpha)(\wt{x})=\p_\mu\p_\alpha g(x,\mu,\alpha)(\wt{x})=\varepsilon_3;\\ 
&\p_x\p_\alpha g(x,\mu,\alpha)=\p_\alpha\p_x g(x,\mu,\alpha)=0,\ \p_{x}\p_\mu g(x,\mu,\alpha)(\wt{x})=\p_{\mu} \p_x g(x,\mu,\alpha)(\wt{x})=-\varepsilon_2. \\
(iii)\ &\p_x k(x,\mu)=x-\varepsilon_4 \int_{\R} y d\mu(y),\ \p_\mu k(x,\mu)(\wt{x})= -\varepsilon_4 x,\ \p_{xx} k(x,\mu)=1;\\
&\p_{x}\p_\mu k(x,\mu)(\wt{x})=\p_{\mu} \p_x k(x,\mu)(\wt{x})=-\varepsilon_4.
\end{align*}\normalsize
All of these functions are clearly jointly Lipschitz continuous in their corresponding arguments.  Note also that $0\leq x\exp\big(-x^2\big)\leq \frac{1}{\sqrt{2}}\exp\big(-\frac{1}{2}\big)\leq \frac{1}{2}$ and $0\leq x^2\exp\big(-x^2\big)\leq \exp\big(-1\big)\leq \frac{1}{2}$. We emphisize here that both $g(x,\mu,\alpha)$ and $k(x,\mu)$ do not satisfy Lasry-Lions monotonicity condition since $\p_{x}\p_\mu g(x,\mu,\alpha)(\wt{x})=-\varepsilon_2<0$ and $\p_x\p_\mu k(x,\mu)(\wt{x})=-\varepsilon_4<0$.

Now, we shall check that whether the $f$, $g$ and $k$ defined in \eqref{example_f}-\eqref{example_k} do satisfy Assumptions ${\bf(a1)}$-${\bf(a3)}$ and Hypotheses ${\bf(h1)}$-${\bf(h3)}$ in the following:\\
${\bf(a1)}$(i) Since $|\p_\alpha f(x,\mu,\alpha)|^2 \geq (1-\frac{1}{2}\varepsilon_1)^2$, then \eqref{positive_f} is valid with $\lambda_f:=0.99$ for any $0<\varepsilon_1\leq 0.01$.\\
(ii) Since $|\p_\alpha f(x,\mu,\alpha)| \leq 1+\frac{1}{2}\varepsilon_1=:\Lambda_1$, $|\p_\mu f(x,\mu,\alpha)(\wt{x})| \leq  1+\frac{1}{2}\varepsilon_1=:\Lambda_2$ and $|\p_x f(x,\mu,\alpha)| \leq 1+2\varepsilon_1=:\Lambda_3$, then \eqref{bdd_d1_f} is valid with $\Lambda_f:=1.02$ for any $0<\varepsilon_1\leq 0.01$.\\
(iii) Using the decay nature of $\exp\Big(-x^2-\alpha^2-\Big(\int_\R \phi(y)d\mu(y)\Big)^2\Big)$ and the fact that $\|\mu\|_1\leq \int_\R \phi(y)d\mu(y)$, one can show that there exists a generic positive constant $C_1$ such that \eqref{bdd_d2_f} is valid with $\wb{l}_f:=C_1\varepsilon_1$.\\
(iv) The Lipschitz continuity can be verified in the same manner as in ${\bf(a1)}$(iii).\\
${\bf(a2)}$(i) Since $\p_x\p_x g(x,\mu,\alpha)=1$, \eqref{positive_g_alpha} is valid with $\lambda_g:=1$.\\
(ii) First of all, since $\p_x\p_x g(x,\mu,\alpha)=1$, \eqref{positive_g_x} is valid with $\lambda_g=1$; since $\p_x\p_\mu g(x,\mu,\alpha)(\wt{x})=\p_\mu\p_x g(x,\mu,\alpha)(\wt{x})=-\varepsilon_2$, \eqref{positive_g_mu_1} is valid with $l_g:=\varepsilon_2<\lambda_g$ for any $\varepsilon_2<1$; in addition, $\int_{\R}g(x,\mu,\alpha)-g(x,\mu',\alpha)d(\mu-\mu')(x)=-\varepsilon_2\left(\int_{\R} yd(\mu-\mu')(y)\right)^2$, which implies \eqref{positive_g_mu}.\\
(iii) Using the explicit formulae for all second-order derivatives of $g$, one can check that \eqref{bdd_d2_g_1} is valid with $\Lambda_g:=1$ for any $\varepsilon_2\leq 1$ and \eqref{bdd_d2_g_2} is valid with $\wb{l}_g:=\varepsilon_3$.\\
${\bf(a3)}$ (i) Since $\p_x\p_x k(x,\mu)=1$, \eqref{positive_k} are valid with $\lambda_k:=1$.\\
(ii) Since $\p_x\p_\mu k(x,\mu)(\wt{x})=\p_\mu\p_x k(x,\mu)(\wt{x})=-\varepsilon_4$, \eqref{positive_k_mu_1} is valid with $l_k:=\varepsilon_4<\lambda_k=1$ for any $\varepsilon_4< 1$; in addition, $\int_{\R}k(x,\mu)-k(x,\mu')d(\mu-\mu')(x)=-\varepsilon_4 \left(\int_{\R} yd(\mu-\mu')(y)\right)^2$, which implies \eqref{positive_k_mu}.\\
(iii) Using the explicit formulae for all second-order derivatives of $k$, one can check that \eqref{bdd_d2_k_1} is valid with $\Lambda_k:=1$ for any $\varepsilon_4\leq 1$.\\
${\bf(h1)}$ Evaluating \eqref{example_f} and formulae for derivatives of $g$ and $k$ at $(x,\mu,\alpha,\wt{x})=(0,\delta_0,0,0)$ shows the fulfillment of ${\bf(h1)}$.\\
${\bf(h2)}$ It is worth noting that for the $\wb{l}_g$ and $\lambda_g$ chosen above, we have $\wb{l}_g=\varepsilon_3\leq \frac{1}{8}=\frac{1}{8}\lambda_g$. Furthermore, the constant $L^*_0$ defined in \eqref{L_star_0} 
depends only on $\lambda_f=0.99$, $\Lambda_f=1.02$, $\lambda_g=1$, $\Lambda_g=1$, $\lambda_k=1$ and $\Lambda_k=1$ but not on $\wb{l}_f$ or $\varepsilon_1$, so we can always choose an $\varepsilon_1$ small enough, such that $\wb{l}_f=C_1\varepsilon_1\leq \frac{1}{40 \max\{\Lambda_k,L^*_0\}}\lambda_g=\frac{1}{40 L^*_0}$. Therefore, ${\bf(h2)}$ is satisfied. \\
${\bf(h3)}$ By some simple computations, we can show that $\frac{25}{16}\Lambda_f^2<1.63<3.733(1-\varepsilon_2)<4(\lambda_g-l_g)\cdot\frac{\lambda_f^2}{\Lambda_g+\lambda_g/20}$ for any $\varepsilon_2\leq\frac{1}{2}<1-\frac{1.63}{3.733}$, and then the validity of ${\bf(h3)}$ can be inferred from Remark \ref{remark_h2}.

\section*{Acknowledgment}
Alain Bensoussan is supported by the National Science Foundation under grant NSF-DMS-2204795.
Tak Kwong Wong was partially supported by the HKU Seed Fund for Basic Research under the project code 201702159009, the Start-up Allowance for Croucher Award Recipients, and Hong Kong General Research Fund (GRF) grants with project numbers 17306420, 17302521, and 17315322.
Phillip Yam acknowledges the financial supports from HKSAR-GRF 14301321 with the project title ``General Theory for Infinite Dimensional Stochastic Control: Mean Field and Some Classical Problems'', HKSAR-GRF 14300123 with the project title ``Well-posedness of Some Poisson-driven Mean Field Learning Models and their Applications'', and NSF-DMS-2204795. He also thanks Columbia University for the kind invitation to be a visiting faculty member in the Department of Statistics during his sabbatical leave. Hongwei Yuan thanks the Department of Statistics of The Chinese University of Hong Kong and the Department of Mathematics of University of Macau for the financial support.

\bibliography{Bib}
\bibliographystyle{plain}

\end{document}